\documentclass[a4paper,10pt]{article}

\usepackage[top=4cm, bottom=2.7cm, left=3cm, right=3cm]{geometry}
\parskip=\baselineskip

\usepackage{titlesec}
\titlespacing{\section}{0pt}{0pt}{0pt}
\titlespacing{\subsection}{0pt}{0pt}{0pt}

\usepackage[titletoc,toc,title]{appendix}

\usepackage[centertags]{amsmath}
\usepackage{amsmath}
\usepackage[british]{babel}
\usepackage{amsfonts}
\usepackage{amssymb}
\usepackage{amsthm}

\usepackage{newlfont}
\usepackage{color}

\usepackage{bbm}

\usepackage{stmaryrd}
\usepackage{mathrsfs}
\usepackage{pstricks,pst-node}

\usepackage{fancyhdr}
\usepackage{graphicx}
\usepackage{fancybox}
\usepackage{setspace}
\usepackage{hyperref,cleveref}

\usepackage{caption}
\usepackage{floatrow}   
\usepackage{subcaption}
\usepackage{graphicx,xcolor} 
\usepackage[framemethod=tikz]{mdframed}

\numberwithin{equation}{section}
\newtheorem{theorem}{Theorem}
\newtheorem{corollary}[theorem]{Corollary}
\newtheorem{lemma}[theorem]{Lemma}
\newtheorem{proposition}[theorem]{Proposition}
\newtheorem{assump}[theorem]{Assumption}
\newtheorem{definition}[theorem]{Definition}
\theoremstyle{definition}
\theoremstyle{remark}
\newtheorem{remark}[theorem]{Remark}
\numberwithin{theorem}{section}

\newcommand{\R}{\mathbb{R}}
\newcommand{\C}{\mathbb{C}}
\newcommand{\N}{\mathbb{N}}
\newcommand{\Z}{\mathbb{Z}}

\newcommand{\E}{\mathbb{E}}
\newcommand{\T}{\mathbb{T}}
\newcommand{\p}{\mathbb{P}}

\newcommand{\im}{\mathrm{Im }\,}
\newcommand{\re}{\mathrm{Re }\,}
\newcommand{\dd}{\mathrm{d}}
\newcommand{\ii}{\mathrm{i}}
\renewcommand{\a}{\mathrm{a}}
\renewcommand{\b}{\mathrm{b}}

\newcommand{\mfc}{m_{fc}}
\newcommand{\whmfc}{\wh m_{fc}}
\newcommand{\deq}{\mathrel{\mathop:}=}
\newcommand{\beq}{ \begin{equation} }
\newcommand{\eeq}{ \end{equation} }
\newcommand{\wt}{\widetilde}
\newcommand{\wh}{\widehat}
\newcommand{\ve}{\varepsilon}

\newcommand{\caD}{{\mathcal D}}

\newcommand{\caL}{{\mathcal L}}

\newcommand{\caO}{{\mathcal O}}

\newcommand{\caQ}{{\mathcal Q}}
\newcommand{\caR}{{\mathcal R}}

\newcommand{\caX}{{\mathcal X}}
\newcommand{\caY}{{\mathcal Y}}

\newcommand{\bfy}{{\mathbf y}}

\newcommand{\fb}{{\mathfrak b}}
\newcommand{\fc}{{\mathfrak c}}
\newcommand{\lone}{\mathbbm{1}}

\DeclareMathOperator{\Tr}{Tr}

\DeclareMathOperator{\supp}{supp}

\begin{document}
	\begin{minipage}{0.85\textwidth}
		\vspace{3cm}
	\end{minipage}
	\begin{center}
		\large\bf
		Extremal eigenvalues of sample covariance matrices with general population
	\end{center}

\begin{center}
	Jinwoong Kwak\footnote{Department of Mathematical Sciences, KAIST, Daejeon, South Korea. e-mail: jw-kwak@kaist.ac.kr}, Ji Oon Lee\footnote{Department of Mathematical Sciences, KAIST, Daejeon, South Korea. e-mail: jioon.lee@kaist.edu}, Jaewhi Park\footnote{Department of Mathematical Sciences, KAIST, Daejeon, South Korea. e-mail:
		jebi1991@kaist.ac.kr}
\end{center}
\vspace{1.5cm}
\begin{abstract}
	We consider the eigenvalues of sample covariance matrices of the form $\mathcal{Q}=(\Sigma^{1/2}X)(\Sigma^{1/2}X)^*$. The sample $X$ is an $M\times N$ rectangular random matrix with real independent entries and the population covariance matrix $\Sigma$ is a positive definite diagonal matrix independent of $X$. Assuming that the limiting spectral density of $\Sigma$ exhibits convex decay at the right edge of the spectrum, in the limit $M, N \to \infty$ with $N/M \to d\in(0,\infty)$, we find a certain threshold $d_+$ such that for $d>d_+$ the limiting spectral distribution of $\mathcal{Q}$ also exhibits convex decay at the right edge of the spectrum. In this case, the largest eigenvalues of $\mathcal{Q}$ are determined by the order statistics of the eigenvalues of $\Sigma$, and in particular, the limiting distribution of the largest eigenvalue of $\mathcal{Q}$ is given by a Weibull distribution. In case $d<d_+$, we also prove that the limiting distribution of the largest eigenvalue of $\caQ$ is Gaussian if the entries of $\Sigma$ are i.i.d. random variables. While $\Sigma$ is considered to be random mostly, the results also hold for deterministic $\Sigma$ with some additional assumptions.
\end{abstract}

{\textit{AMS Subject Classification (2010)}: 60B20, 62H10, 15B52
	
	
	\vspace{3mm}
	\textit{\today}
	\vspace{3mm}}

\section{Introduction}
For a vector-valued, centered random variable $\bfy \in \R^M$, its population covariance matrix is given by $\Sigma\deq \E[\bfy\bfy^T]$. For $N$ independent samples $(\bfy_1,\cdots.\,\bfy_N)$ of $\bfy$, the sample covariance matrix $\frac1N \sum_{i=1}^{N}\bfy_i\bfy_i^T$ can be a simple and unbiased estimator of $\Sigma$ when $N$ is much larger than $M$. On the other hand, if the sample number $N$ is comparable to the population size $M$, the sample covariance matrix is no more a reasonable estimator for the population covariance matrix. Nevertheless, even in such a case, the characteristic of the population covariance matrix may appear in the sample covariance matrix, as we consider in this paper.

We are interested in a matrix of the form 
\begin{equation}\label{model}
\caQ=(\Sigma^{1/2}X)(\Sigma^{1/2}X)^*,
\end{equation}
where the sample $X$ is an $M\times N$ matrix whose entries are independent real random variables with variance $1/N$, and the general population covariance $\Sigma$ is an $M\times M$ real diagonal positive definite matrix. 
We focus on the case that $M$ and $N$ tend to infinity simultaneously with $\wh d\deq N/M\rightarrow d\in (0,\infty)$, as $M, N\rightarrow \infty$. 

The asymptotic behavior of the empirical spectral distribution (ESD) of sample covariance matrices was first considered by Marchenko and Pastur~\cite{mp}; they derived a core structure of the limiting spectral distribution (LSD) for a class of sample covariance matrices and the LSD is called the Marchenko--Pastur (MP) law. In the null case, $\Sigma = I$, the distribution of the rescaled largest eigenvalue converges to the Tracy--Widom law~\cite{f,kj,jm,p}. For the non-null case, i.e. $\Sigma\neq I$, the location and the distribution of the outlier eigenvalues, including the celebrated BBP transition, have been studied extensively when $\Sigma$ is a finite rank perturbation of the identity. For more detail, we refer to~\cite{jb,s2,s1,r1,paul,wd} and references therein. 

When $\Sigma$ has more complicated structure, e.g., the LSD of $\Sigma$ has no atoms, the limiting distribution of the largest eigenvalue is given by the Tracy--Widom distribution under certain conditions. It was first proved by El Karoui~\cite{k} for complex sample covariance matrices and extended to the real case~\cite{uzb,scjl,ky}. In these works, one of the key assumptions is that the LSD exhibits the ``square-root'' type behavior at the right edge of the spectrum, which also appears in the Wigner semicircle distribution or the Marchenko--Pastur distribution. It is then natural to consider the local behavior of the eigenvalues when square-root type behavior is absent. Note that if the LSD of $\Sigma$ decays concavely at the right edge the LSD of $\caQ$ exhibits the square-root behavior at the right edge~\cite{HCJi}.

\vskip5pt

\textbf{Main contribution} 

The sample covariance matrix, as $N$ gets relatively larger than $M$, approximates the population covariance matrix more accurately. Thus, it is natural to conjecture that the behavior of the largest eigenvalues of the sample covariance matrix must be similar to that of $\Sigma$ if $d$ is above a certain threshold.
Our main result of this paper establishes the conjecture rigorously. We find that there exists $d_+$ such that for $d>d_+$
\begin{itemize}
\item
The LSD of $\caQ$ is convex near the right edge of its support (Theorem~\ref{general case - large d}), and
\item
the distribution of the largest eigenvalues of $\caQ$ are determined by the order statistics of the eigenvalues of $\Sigma$ (Theorem~\ref{theorem:main}).
\end{itemize}
We also prove that the largest eigenvalue of $\caQ$ converges to a Gaussian for $d<d_+$, when the entries of $\Sigma$ are i.i.d. (Theorem~\ref{lemma:gaussian})

\vskip5pt

\textbf{Main idea of the proof}

In the first step, we prove general properties of the LSD of $\mathcal{Q}$. The proof is based on the fact that the LSD of $Q$ can be defined by a functional equation whose unique solution is the Stieltjes transform of LSD of $Q$; see also~\cite{mp}. 

In the second step, we prove a local law for the resolvents of $Q$ and $\mathcal{Q}$. The main technical difficulty of the proof stems from that it is not applicable the usual approach based on the self-consistent equation as in~\cite{ky}. Technically, this is due to the lack of the stability bounds as in equation A.8 of~\cite{ky}, which is not guaranteed when the LSD of $\Sigma$ decays convexly at the edge. Thus, we adapt the strategy of~\cite{eejl} for deformed Wigner matrices in the analysis of the self-consistent equation. For the analysis of the resolvents, we use the linearization of $Q$ whose inverse is conveniently related to the resolvents of $Q$ and $\mathcal{Q}$. Together with Schur's complement formula and other useful formulas for the resolvents of $Q$ or $\mathcal{Q}$, we prove a priori estimates for the local law. 

In the last step, we apply the ``fluctuation averaging'' argument to control the imaginary part of the resolvent of $Q$ on much smaller scale than $N^{-1/2}$. The analysis is different from other works involving the same idea such as~\cite{p,sce}, due to the unboundedness of the diagonal entries of the resolvent of $\mathcal{Q}$. Finally, by precisely controlling the imaginary part of the argument in the resolvent, we track the location of the eigenvalues at the edge.

\vskip5pt

\textbf{Related works}

In the context of Wigner matrices, the edge behavior of the LSD of a Wigner matrix can be altered by deforming it. The deformed Wigner matrix is of the form $H=W+\lambda V$ where $W$ is a Wigner matrix and $V$ is a real diagonal matrix independent of $W$. If $\lambda$ is chosen so that the spectral norm of $W$ is of comparable order with that of $V$, and the LSD of $V$ has convex decay at the edge of its spectrum, then the LSD of $H$ also exhibits the same decay at the edge if the strength of the deformation $\lambda$ is above a certain threshold. In that case, the limiting fluctuation of the largest eigenvalues is given by a Weibull distribution instead of the Tracy--Widom distribution. See~\cite{dsjl,eejl} for more precise statements.

The largest eigenvalues of sample covariance matrices are frequently used in the analysis of signal-plus-noise models. In systems biology, the largest eigenvalues derived from single-cell data sets can be used for identification of biological information~\cite{aparicio2018}. In the context of machine learning, the behavior of the largest eigenvalues indicate different phases of training in deep neural networks~\cite{mahoney2019}.

\vskip5pt

\textbf{Organization of the paper} 

The rest of the paper is organized as follows: In section~\ref{Definition and Results}, we define the model and state the main results. In section~\ref{prelim}, we introduce basic notations and tools that will be used in the analysis. In section~\ref{pf main theorem}, we prove main theorems. Section~\ref{location} is devoted to the proof of Proposition~\ref{proposition:lambda_k}, one of the key results used in the proof of main theorems. Proofs of some technical lemmas are collected in Supplementary Material.

\section{Definition and Results} \label{Definition and Results}

In this section, we define our model and state the main result.

\subsection{Definition of the model}

\begin{definition}[Sample covariance matrix with general population] \label{assumption sample}
A sample covariance matrix with general population $\Sigma$ is a matrix of the form 
\begin{align} \label{thematrix}
	\caQ \deq (\Sigma^{1/2} X)(\Sigma^{1/2} X)^*,
\end{align}
where $X$ and $\Sigma$ are given as follows:

Let $X$ be an $M\times N$ real random matrix whose entries $(x_{ij})$ are independent, zero-mean random variables with variance $1/N$ and satisfying 
\begin{equation}\label{p moment bound}
	\E[|x_{ij}|^p]\le \frac{c_p}{N^{p/2}}
\end{equation}
for some positive constants $c_p>0$ depending only on $p\in \N$.

Let $\Sigma$ be an $M\times M$ real diagonal matrix whose LSD is $\nu$, entries $(\sigma_\alpha)$ are nonnegative and independent from $X$.
Also, the measure $\nu$ has density
\begin{align} \label{jacobi measure}
	\rho_\nu(t)=Z^{-1} (1-t)^{\b} f(t)\lone_{[l, 1]}(t)\,, \qquad 0<l<1
\end{align}
where $\b>-1$, $f\in C^1 [l, 1]$ such that $f(t)>0$ for $t \in [l, 1]$, and $Z$ is a normalizing constant.

The dimensions $N\equiv N(M)$ 
and
\begin{equation}
	\wh d=\frac{N}{M}\rightarrow d\in(0,\infty),
\end{equation}
as $n\rightarrow\infty$.  (For simplicity, we assume that $\wh d$ is constant, so we use $d$ instead of $\wh d$.)

We denote the eigenvalues of $\caQ$ by $(\lambda_i)$ with the ordering $\lambda_1 \ge \lambda_2 \ge \ldots \ge \lambda_M$.
\end{definition}

The measure~$\nu$ is called a \textbf{Jacobi measure}. We remark that the measure~$\nu$ has support $[l,1]$ for some $l>0$. In this paper, we only consider the case $\b>1$ in~\eqref{jacobi measure}.

Note that in Definition~\ref{assumption sample}, we only assume the independence of the entries $(x_{ij})$ and do not assume that $(x_{ij})$ are identically distributed. We mainly assume that $\Sigma$ is random.

\begin{remark}\label{MNM}
In the sequel, we often interchange $N$ to $CM$ in the middle of several inequalities for some absolute constant $C$ reasoning that $M$ and $N$ have the same order.
\end{remark}

\begin{remark}
With the assumption on the Jacobi measure, we have that $\liminf \sigma_M \ge l $ and $\limsup \sigma_1 \le 1$, which were also assumed in~\cite{k}. 
\end{remark}

\begin{remark}\label{Q and wt Q}
Let $Q\deq X^*\Sigma X$, then $\caQ$ is an $M\times M$ matrix and $Q$ is an $N\times N$. The eigenvalues of $Q$ can be described as the following; $Q$ shares the nonzero eigenvalues with $\caQ$ and has $0$ eigenvalue with multiplicity $N-M$ when $N\geq M$.
Thus, we denote the eigenvalues of $Q$ by $(\lambda_i)_{i=1}^N$ where $\lambda_i=0$ for $M+1\le i\le N$.
\end{remark}
\subsection{Assumptions on $\Sigma$}

For our main result, Theorem~\ref{theorem:main}, to hold, it requires that the gaps between the largest eigenvalues $(\sigma_\alpha)$, $\alpha\in \llbracket1,n_0 \rrbracket$, of $\Sigma$ must not be too small. Heuristically, when $\b>1$ in \eqref{jacobi measure}, the Jacobi measure has convex decay at the edge so that we can regard that the distance of immediate eigenvalues is typically large near the edge. Due to the distance, a few largest eigenvalues of $\Sigma$ significantly affect the edge of LSD of $\mathcal{Q}$ more than any other small eigenvalues. In order to describe the condition mathematically, we introduce the following event $\Omega$, which is a ``good configuration'' of the largest eigenvalues of $\Sigma$.

Denote by $\fb$ the constant
\begin{align} \label{fb}
\fb \deq \frac{1}{2} - \frac{1}{\b +1} = \frac{\b -1}{2(\b + 1)} = \frac{\b}{\b +1} - \frac{1}{2}\,,
\end{align}
which depends only on $\b$ in~\eqref{jacobi measure}. Fix some (small) $\phi > 0$ satisfying
\begin{align}\label{phi condition}
\phi < \left( 10 + \frac{\b +1}{\b -1} \right) \fb\,,
\end{align}
and define the domain $\caD_{\phi}$ of the spectral parameter $z$ by 
\begin{align} \label{domain}
\caD_{\phi} \deq \{ z = E + \ii \eta\in\C^+\, :\, l \le E \le 2 + \tau_+, \; M^{-1/2 - \phi} \le \eta \le M^{-1/(\b+1) + \phi} \}\,.
\end{align}
Further, we define $N$-dependent constants $\kappa_0$ and $\eta_0$ by
\begin{align}\label{definition of kappa0}
\kappa_0 \deq M^{-1/(\b+1)}, \qquad\quad \eta_0 \deq \frac{M^{-\phi}}{\sqrt M}\,.
\end{align}
In the following, typical choices for $z \equiv L_+ - \kappa + \ii \eta$ will be $\kappa$ and $\eta$ with $\kappa \le M^{\phi} \kappa_0$ and $\eta \ge \eta_0$.

We are now ready to give the definition of a ``good configuration'' $\Omega$. Let $\mu_{fc}$ be the limiting spectral measure of $\caQ$ and $m_{fc}$ the Stieltjes transform of $\mu_{fc}$. (See section~\ref{subsec:deformed_mp} for the precise definition.) Without loss of generality, we assume that the entries of $\Sigma$ satisfy the following inequality,
\begin{align}\label{the ordering}
\sigma_1 \ge \sigma_2 \ge \ldots \ge \sigma_M\ge 0\,.
\end{align}
\begin{definition} \label{sigma assumptions}
Let $n_0 > 10$ be a fixed positive integer independent of $M$. We define $\Omega$ to be the event on $\Sigma$ which for any $\gamma \in \llbracket 1, n_0 -1 \rrbracket$, the following conditions:

\begin{enumerate}
	\item The $\gamma$-th largest eigenvalue $\sigma_\gamma$ satisfies, for all $\beta\in\llbracket 1,n_0\rrbracket$ with $\beta \neq \gamma$,
	\begin{align}\label{eq4.3}
		M^{-\phi} \kappa_0 < |\sigma_\beta - \sigma_\gamma| < (\log M) \kappa_0\,.
	\end{align}
	In addition, for $\gamma=1$, we have
	\begin{align}\label{eq4.4}
		M^{-\phi}\kappa_0 < |1 - \sigma_1| < (\log M)\kappa_0\,,
	\end{align}
	hence for $\alpha\in \llbracket n_0+1, M \rrbracket$,
	\begin{align}
		M^{-\phi}\kappa_0 < |\sigma_\alpha - \sigma_\gamma|.
	\end{align}

	\item There exists a constant $\mathfrak{c} <1$ independent of $M$ such that for any $z \in \caD_{\phi}$ satisfying
	\begin{align} \label{assumption near sigma_k}
		\min_{\alpha \in \llbracket 1, M \rrbracket} \left|\re \left(1+\frac{1}{\sigma_\alpha \mfc}\right)\right| = \left|\re \left(1+\frac{1}{\sigma_\gamma \mfc}\right)\right|\,,
	\end{align}
	we have
	\begin{align}\label{assumption_CLT_2}
		\frac{1}{N} \sum_{\substack{\alpha=1 \\ \alpha\neq \gamma}}^{M} \frac{\sigma_\alpha^2 \lvert \mfc\rvert^2}{|1+\sigma_\alpha \mfc|^2} <\mathfrak{c} < 1\,.
	\end{align}
	We remark that, together with~\eqref{eq4.3} and~\eqref{eq4.4},~\eqref{assumption near sigma_k} implies
	\begin{align} \label{farfromhome}
		\left|\re \left(1+\frac{1}{\sigma_\alpha \mfc}\right)\right| \geq \frac{ M^{-\phi}\kappa_0}{2}\,,
	\end{align}
	for all $\alpha \neq \gamma$.
	
	\item For any $\epsilon>0$, there exists $C_\epsilon>0$ and $M_\epsilon$ (large) such that for any $z \in \caD_{\phi}$ and $M\geq M_\epsilon$,
	\begin{align} \label{assumption_CLT_1}
		\left| \frac{1}{N} \sum_{\alpha=1}^M \frac{\sigma_\alpha}{\sigma_\alpha \mfc+1} - d^{-1}\int \frac{t\dd \nu(t)}{t \mfc+1} \right| \leq \frac{C_\epsilon M^{\phi+\epsilon}}{\sqrt M}\,.
	\end{align}
	
\end{enumerate}
\end{definition}
Throughout the paper, we assume that $\Sigma$ satisfy Definition~\ref{sigma assumptions}, and ESD of $\Sigma$ converges weakly to a Jacobi measure with $\b>1$.  
\begin{assump}\label{assumption1}
Let $\Sigma$ be $M\times M$ real diagonal random matrix satisfying the conditions in Definition \ref{assumption sample} with $\b>1$. We assume that:
\begin{enumerate}
	\item [i.] When $\Sigma$ is deterministic, $\Sigma$ satisfies 
	$\Omega$ in Definition~\ref{sigma assumptions}.
	\item [ii.] When $\Sigma$ is random, $\mathbb{P}(\Omega) \rightarrow 1$ as $M\rightarrow \infty$.
\end{enumerate}
\end{assump}

We remark that if $\Sigma$ is a diagonal random matrix whose entries are i.i.d Jacobi measure $\nu$ with $\b>1$, the Glivenko--Cantelli theorem asserts that the LSD of $\Sigma$ converges to $\nu$ itself. Furthermore, in Appendix~\ref{app:omega} we show that 
\begin{align} \label{Omega}
\p (\Omega) \ge 1 - C (\log M)^{1 + 2\b} M^{-\phi}, 
\end{align}
thus the ``bad configuration'' $\Omega^c$ occurs rarely. In other words, when $\Sigma$ has i.i.d. entries with law $\nu$, it automatically satisfies the properties in Definition~\ref{sigma assumptions} with high probability. For the non i.i.d random or deterministic $\Sigma$, we assume Assumption~\ref{assumption1}.

\subsection{Main results}
Our first main result is about the behavior of the limiting spectral measure of $\mathcal{Q}$, $\mu_{fc}$, near its right edge. The following theorem establishes not only the explicit location of the right edge of $\mu_{fc}$ but also the local behavior of $\mu_{fc}$ near the right edge. In the sequel, we denote by $L_+$ the right endpoint of the support of $\mu_{fc}$ and $\kappa \equiv \kappa(E) \deq |E-L_+|$ where $z=E+\ii\eta$.

\begin{theorem}\label{general case - large d}
Suppose that $\caQ$ is a sample covariance matrix with general population $\Sigma$ defined in Definition~\ref{assumption sample}. Let $\nu$ be a Jacobi measure defined in~\eqref{jacobi measure} with $\b>1$. Define
\begin{align}\label{definition of lambda+}
	d_+ \deq \int_{l}^1 \frac{t^2\dd\nu(t)}{(1-t)^2}, \qquad \tau_+ \deq d^{-1}\int_{l}^1 \frac{t \dd\nu(t) }{1-t}\,.
\end{align}
If $d>d_+$, then $L_+=1+\tau_+$. Moreover, for $0 \le \kappa \le L_+$, there exists a constant $C>1$ such that
\begin{align} \label{exponent beta}
	C^{-1}{\kappa}^{\b} \le \mu_{fc}(L_+-\kappa) \le C \kappa^{\b}.
\end{align}	
\end{theorem}

We prove Theorem~\ref{general case - large d} in Section~\ref{pf edge}.

Our second result concerns the locations of the largest eigenvalues of $\mathcal{Q}$ in the supercritical case, which are determined by the order statistics of the eigenvalues of $\Sigma$.
In the following, we fix some $n_0\in\N$ independent of $M$ and consider the largest eigenvalues $(\lambda_\alpha)_{\alpha=1}^{n_0}$ of~$\caQ$.

\begin{theorem} \label{theorem:main}
Suppose that Assumption~\ref{assumption1}, assumptions in Theorem~\ref{general case - large d} and $d>d_+$ hold. Let $n_0>10$ be a fixed constant independent of $M$ and let $1\le \gamma<n_0$. Then the joint distribution function of the $\gamma$ largest rescaled eigenvalues,
\begin{align}\label{converging expression 1}
	\p \left( M^{1/(\b+1)} (L_+ - \lambda_1)\le s_1,\, M^{1/(\b+1)} (L_+ - \lambda_2)\le s_2,\, \ldots,\, M^{1/(\b+1)} (L_+ - \lambda_\gamma)\le s_\gamma \right)\,,
\end{align}
converges to the joint distribution function of the $\gamma$ largest rescaled order statistics of $(\sigma_\alpha)$,
\begin{align}\label{converging expression}
	\p \left(C_d M^{1/(\b+1)} (1-\sigma_1)\le s_1,\,C_d M^{1/(\b+1)} (1-\sigma_2)\le s_2,\, \ldots,\,C_d M^{1/(\b+1)}  (1-\sigma_\gamma)\le s_\gamma \right)\,,
\end{align}
as $N \to \infty$, where $C_d= \frac{d-d_+}{d}$.
In particular, when $\Sigma$ has i.i.d. entries with law $\nu$, the cumulative distribution function of the rescaled largest eigenvalue $M^{1/(\b+1)} (L_+ - \lambda_1)$ converges to the cumulative distribution function of the Weibull distribution,
\begin{align}\label{Weibull}
	G_{\b+1}(s)\deq 1 - \exp \left( -\frac{C_\nu s^{\b+1}}{(\b+1)} \right)\,,
\end{align} 
where
\begin{align}
	C_\nu \deq \left( \frac{d}{d-d_+} \right)^{\b+1} \lim_{t \to 1} \frac{\rho_\nu(t)}{(1-t)^{\b}}\,.
\end{align}
\end{theorem}

Our third result states that the largest eigenvalue of $\caQ$ exhibits Gaussian fluctuation when $d<d_+$ and the eigenvalues of $\Sigma$ are i.i.d. random variables. 
\begin{theorem}[Gaussian fluctuation for the regime $d<d_+$] \label{lemma:gaussian}
Suppose that assumptions in Theorem~\ref{general case - large d} hold except that $d<d_+$. Further, assume that the eigenvalues of $\Sigma$ are i.i.d. random variables. Then, the rescaled fluctuation $M^{1/2} (\lambda_1 - L_+)$ converges in distribution as $N \to \infty$ to a centered Gaussian random variable with variance 
\begin{equation}
	(d^2 M)^{-1} \left( \int\left|\frac{t\tau}{t+\tau}\right|^2\dd\nu(t)- \left(\int\frac{t\tau}{t+\tau}\dd\nu(t)\right)^2 \right) 
\end{equation}
where $\tau$ and $L_+$ are defined in the proof.
\end{theorem}

We prove Theorems~\ref{theorem:main} and~\ref{lemma:gaussian} in Section~\ref{pf main theorem}.

If $X$ is Gaussian, our main results still hold for general, non-diagonal $\Sigma$ satisfying Definition~\ref{sigma assumptions}.

\begin{corollary}\label{cor:Gaussian}
Suppose that assumptions in Theorem~\ref{general case - large d} hold with the following changes: the entries of $X$ are Gaussian, and $\Sigma$ is not necessarily diagonal, with eigenvalues $(\sigma_{\alpha})$. Then, the results in Theorems~\ref{general case - large d},~\ref{theorem:main}, and~\ref{lemma:gaussian} hold without any change.
\end{corollary}

\subsection{Numerical experiment} \label{subsec:simul}

We conduct numerical simulations to observe the local behavior of the empirical spectral distribution of deformed sample covariance matrices. In each simulation done with MATLAB, we generate 10 sample covariance matrices of the form
\begin{align}
Q=\frac{1}{N} X^*\Sigma X
\end{align} 
under fixed $\Sigma$ and plot the histograms of non-zero eigenvalues of $Q$ to find the behavior of the ESD of $Q$ at the right edge.

\subsubsection{Convex, super-critical case, $\b>1$ } \label{subsec:case1}
We first generate $4000 \times 6000$ matrices $X$ with i.i.d standard Gaussian entries and a $4000 \times 4000$ diagonal matrix $\Sigma$ with i.i.d. entries sampled from the density function $f_1$ given by
\begin{align}
f_1(t) = \mathcal{Z}_1^{-1}e^t(1-t)^3 \lone_{[1/10,1]}(t)
\end{align}
with a normalization constant $\mathcal{Z}_1$. In this setting, $d = N/M = 1.5$ and $d_+\approx 0.703908$. The histogram of nonzero eigenvalues of $Q$ can be seen from Figure~\ref{fig:image1}, which shows that the ESD exhibits convex decay at the right edge.

\subsubsection{Concave, sub-critical case, $\b>1$} \label{subsec:case2}

We next generate $4000 \times 2000$ matrices $X$ with i.i.d. standard Gaussian entries. The diagonal matrix $\Sigma$ is the same as in Section~\ref{subsec:case1}. In this setting, $d < d_+$, and the ESD exhibits concave decay at the right edge, as can be seen from Figure~\ref{fig:image12}.

\subsubsection{Concave case, $\b<1$} \label{subsec:case3}

In this setting, we generate $4000 \times 6000$ matrices $X$ with i.i.d standard Gaussian entries again as in Section~\ref{subsec:case1}, but we use a $4000 \times 4000$ diagonal matrix $\Sigma$ with i.i.d. entries sampled from the density function $f_1$ given by
\begin{align}
f_2(t) = \mathcal{Z}_2^{-1}e^t(1-t)^{1/2} \lone_{[1/10,1]}(t)
\end{align}
with a normalization constant $\mathcal{Z}_2$. Formally, $d_+ = \infty$ in this case, and the ESD exhibits concave decay at the right edge as in Figure~\ref{fig:image13}. \\

\begin{figure}[h]
\centering
\begin{subfigure}{.32\linewidth}
	\centering
	\includegraphics[width=1\linewidth]{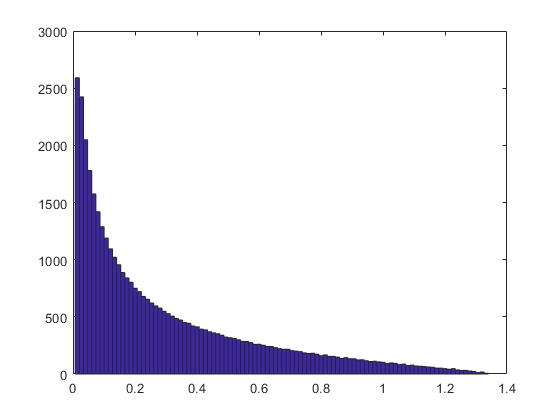}
	\caption{$d>d_+$}\label{fig:image1}
\end{subfigure}
\hfill
\begin{subfigure}{.32\linewidth}
	\centering
	\includegraphics[width=1\linewidth]{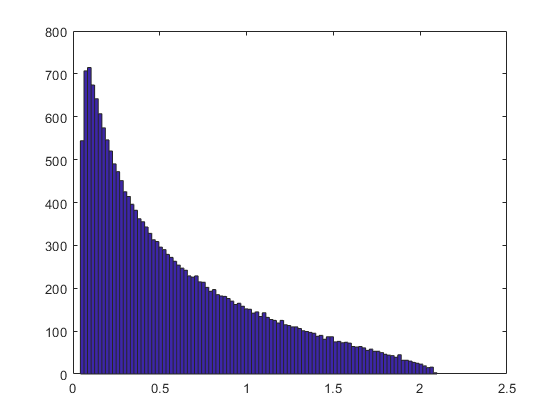}
	\caption{$d<d_+$}\label{fig:image12}
\end{subfigure}
\hfill
\begin{subfigure}{.32\linewidth}
	\centering
	\includegraphics[width=1\linewidth]{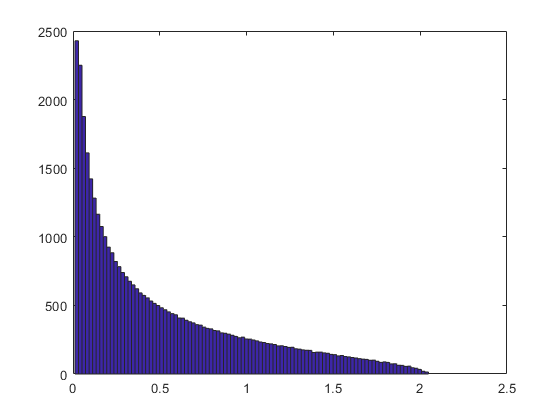}
	\caption{$\b<1$}\label{fig:image13}
\end{subfigure}

\RawCaption{\caption{The limiting ESDs of $Q$}
	\label{fig:images}}
\end{figure}

\section{Preliminaries} \label{prelim}

In this section, we collect some basic notations and identities.

\subsection{Notations}
We adopt the following shorthand notation introduced in~\cite{afe} for high-probability estimates:

\begin{definition}[Stochastic dominance]
Let 
\begin{align}
	X=(X^{(N)}(u):N\in\N,u\in U^{(N)}).\qquad Y=(Y^{(N)}(u):N\in\N,u\in U^{(N)})
\end{align}  be two families of nonnegative random variables where $U^{(N)}$ is a (possibly $N$-dependent) parameter set. We say $X$ is stochastically dominated by $Y$, uniformly in $u$, if for all (small) $\epsilon>0$ and (large) $D>0$, \begin{equation}\label{stochastic dominance}
	\sup_{u\in U^{(N)}} \p[X^{(N)}(u)>N^{\epsilon} Y^{(N)}(u)]\le N^{-D},
\end{equation} for sufficiently large $N\ge N_0(\epsilon, D)$. If $X$ is stochastically dominated by $Y$, uniformly in $u$, we write $X\prec Y$. If for some complex family $X$, we have $|X|\prec Y$ we also write $X=\caO_\prec(Y)$.
\end{definition}

We remark that the relation $\prec$ is a partial ordering with the aritheoremetic rules of an order relation; e.g., if $X_1\prec Y_1$ and $X_2\prec Y_2$ then $X_1+X_2\prec Y_1+Y_2$ and $X_1X_2 \prec Y_1Y_2$.

\begin{definition}[high probability event] We say an event $\Omega$ occurs with high probability if for given $D>0$, $\p(\Omega)\ge 1-N^{-D}$ whenever $N\ge N_0(D)$. Also, we say an event $\Omega_2$ occurs with high probability on $\Omega_1$ if for given $D>0$, $\p(\Omega_2\mid \Omega_1)\ge 1-N^{-D}$ whenever $N\ge N_0(D)$.
\end{definition}
Equivalently, $\Omega$ holds with high probability if $1-\lone(\Omega)\prec 0$.

For convenience, we use double brackets to denote the index set, i.e., for $n_1, n_2 \in \R$,
\begin{align}
\llbracket n_1, n_2 \rrbracket \deq [n_1, n_2] \cap \Z\,.
\end{align}

Throughout the paper, we use lowercase Latin letters $a,b,\cdots$ for indices in $\llbracket 1,N \rrbracket$, uppercase letters $A,B,\cdots$ for indices in $\llbracket 1,N+M \rrbracket$, and Greek letters $\alpha,\beta,\cdots$ for indices in $\llbracket 1,M \rrbracket$. If necessary, we use Greek letters with tilde for indices in $\llbracket N+1,N+M \rrbracket$, e.g., $\wt\alpha=N+\alpha$.

We use the symbols $O(\,\cdot\,)$ and $o(\,\cdot\,)$ for the standard big-O and little-o notation. The notations $O$, $o$, $\ll$, $\gg$, refer to the limit $N \to \infty$ unless stated otherwise, where the notation $a \ll b$ means $a=o(b)$. We use $c$ and $C$ to denote positive constants that are independent on $N$. Their values may change line by line but in general we do not track the change. We write $a \sim b$, if there is $C \ge 1$ such that $C^{-1}|b| \le |a| \le C |b|$.

\subsection{Deformed Marchenko--Pastur law}\label{subsec:deformed_mp}

As shown in~\cite{mp}, if the empirical spectral distribution (ESD) of $\Sigma$, $\nu_N$, converges in distribution to some probability measure $\nu$, then the ESD of $\mathcal{Q}$ converges weakly in probability to a certain deterministic distribution $\mu_{fc}$ which is called the \textbf{deformed Marchenko--Pastur law}. It was also proved in~\cite{mp} that $\mu_{fc}$ can be expressed in terms of its Stieltjes transform as follows:

For a (probability) measure $\omega$ on $\R$, its Stieltjes transform is defined by
\begin{align}
m_{\omega}(z)\deq\int_{\R}\frac{\dd\omega(x)}{x-z}\,,\quad\quad\quad (z\in\C^+)\,.
\end{align}
Notice that $m_{\omega}(z)$ is an analytic function in the upper half plane and $\im m_{\omega}(z)\ge 0$ for $z \in \C^+$.

Let $\mfc$ be the Stieltjes transform of $\mu_{fc}$. It was proved in~\cite{mp} that $\mfc$ satisfies the self-consistent equation
\begin{align}\label{mfc equation}
\mfc(z)=\left\{-z+d^{-1}\int_\R\frac{t\dd\nu(t)}{1+t \mfc(z)}\right\}^{-1},\quad\quad \im \mfc(z)\ge 0\,,\quad\quad (z\in\C^+)\,,
\end{align}
where $\nu$ is the limiting spectral distribution (LSD) of $\Sigma$. It was also shown that~\eqref{mfc equation} has a unique solution. Moreover, $\limsup_{\eta\searrow 0} \im \mfc(E+\ii\eta)<\infty$, and $\mfc(z)$ determines an absolutely continuous probability measure $\mu_{fc}$ whose density is given by
\begin{align}\label{stieltjes inversion}
\rho_{fc}(E)=\frac{1}{\pi}\lim_{\eta\searrow 0} \im \mfc(E+\ii\eta)\,,\quad\quad (E\in\R)\,.
\end{align}
For the properties of $\mu_{fc}$, we refer to~\cite{ss}. We remark that the density $\rho_{fc}$ is analytic inside its support.

\begin{remark}
The measure $\mu_{fc}$ is identified with the multiplicative free convolution of the Marchenko--Pastur measure $\mu_{MP}$ and the measure $\nu$ and is denoted by $\mu_{fc} \deq \nu \boxtimes \mu_{MP}$.
\end{remark}

\subsection{Resolvent and Linearization of $Q$}\label{linearization}
We define the resolvent, or Green function, $G_Q(z)$, and its normalized trace, $m_Q(z)$, of $Q$ by
\begin{align}
G_Q(z) = ((G_Q)_{AB}(z)) \deq (Q-z)^{-1}\,,\quad m_Q(z) \deq \frac{1}{N} \Tr G_Q(z)\,, \quad( z \in \C^{+})\,.
\end{align}
We refer to $z$ as the spectral parameter and set $z=E+\ii\eta$, $E\in\R$, $\eta>0$. 

For the analysis of the resolvent $G_Q(z)$, we use the following linearization trick as in~\cite{scjl}. Define a partitioned $(N+M)\times (N+M)$ matrix 
\begin{equation}\label{linear}
H(z):=\begin{bmatrix} -zI_N & X^*\\X & -\Sigma^{-1} \end{bmatrix},\quad z\in \C^+ \,
\end{equation} 
where $I_N$ is the $N\times N$ identity matrix. Note that $H$ is invertible, as proved in~\cite{scjl}.
Set $G(z):=H(z)^{-1}$ and define the normalized (partial) traces, $m$ and $\wt m$, of $G$ by 
\begin{equation} 
m(z):=\dfrac{1}{N}\sum_{a=1}^N G_{aa}(z), \qquad \wt m(z):=\dfrac{1}{M}\sum_{\wt\alpha=N+1}^{N+M}G_{\wt\alpha\wt\alpha}. 
\end{equation}
With abuse of notation, when we use Greek indices with tilde such as $G_{\wt\alpha \wt\alpha} = G_{N+\alpha, N+\alpha}$, we omit the tilde and set $G_{\alpha \alpha} \equiv G_{\wt\alpha \wt\alpha}$ if it does not causes any confusion.

Frequently, we abbreviate $G \equiv G(z)$, $m\equiv m(z)$, etc. In addition, $m(z)=m_Q(z)$ holds as a consequence of the Schur complement formula, see~\cite{scjl}. Furthermore, from (4.1) of~\cite{ky} and Remark~\ref{Q and wt Q}, we have
\begin{align}\label{rel between m and wt m}
m(z)=\frac{1}{Nz}\sum_\alpha \sigma_\alpha^{-1}G_{\alpha\alpha}-\frac{N-M}{Nz}.
\end{align}

\subsection{Minors}
For $\T\subset \llbracket1, N+M \rrbracket$, the matrix minor $H^{(\T)}$ of $H$ is defined as \begin{equation}(H^{(\T)})_{AB}=\lone(A\notin\T)\lone(B\notin\T)H_{AB},\end{equation} i.e., the entries in the $\T$-indexed columns/rows are replaced by zeros. We define the resolvent $G^{(\T)}(z)$ of $H^{(\T)}$ by
\begin{equation} 
G_{AB}^{(\T)}(z):=\left(\frac{1}{H^{(\T)}-z}\right)_{AB}.
\end{equation}
For simplicity, we use the notations
\begin{align}
\sum_{a}^{(\T)}\deq\sum_{\substack{a=1\\a\not\in\T}}^N\,\,,\qquad \sum_{a\not=b}^{(\T)}\deq\sum_{\substack{a=1,\, b=1\\ a\not=b\,,\,a,b\not\in\T}}^N\,,\qquad\sum_{\alpha}^{(\T)}\deq\sum_{\substack{\alpha=1\\\alpha\not\in\T}}^{M}\,\,,\qquad\sum_{\alpha\not=\beta}^{(\T)}\deq\sum_{\substack{\alpha=1,\, \beta=1\\ \alpha\not=\beta\,,\,\alpha,\beta\not\in\T}}^{M}
\end{align}
and abbreviate $(A)=(\{A\})$, $(\T A)=(\T\cup\{A\})$. In Green function entries $(G_{AB}^{(\T)})$ we refer to $\{A,B\}$ as lower indices and to $\T$ as upper indices.

Finally, we set
\begin{align}
m^{(\T)}\deq\frac{1}{N}\sum_{a}^{(\T)}G_{aa}^{(\T)}\,,\qquad \wt m^{(\T)}\deq\frac{1}{M}\sum_{\alpha}^{(\T)}G_{\alpha\alpha}^{(\T)}.
\end{align}
Note that we use the normalization $N^{-1}$ instead of $(N-|\T|)^{-1}$.

\subsection{Resolvent identities}
The next lemma collects the main identities between the matrix elements of $G$ and its minor $G^{(\T)}$.

\begin{lemma}\label{res id}
Let $G(z)=H^{-1}(z),\; z\in\C^+$ be a Green function defined by~\eqref{linear} and $\Sigma$ is diagonal. For $a,b\in\llbracket 1,N\rrbracket$, $\alpha,\beta\in\llbracket 1,M \rrbracket$, $A,B,C\in\llbracket 1,N+M \rrbracket$, the following identities hold:
\begin{itemize}
	\item[-] {\it Schur complement/Feshbach formula:}\label{feshbach} For any $a$ and $\alpha$,
	\begin{align}\label{schur} 
		G_{aa}=\frac{1}{-z-\sum_{\alpha,\beta}{x_{\alpha a} G_{\alpha\beta}^{(a)}}x_{\beta a}},\qquad
		G_{\alpha \alpha}=\frac{1}{-\sigma_\alpha^{-1}-\sum_{a,b}{x_{\alpha a} G_{ab}^{(\alpha)}}x_{\alpha b}}\,.
	\end{align}
	\item[-] For $a\not=b$,
	\begin{align}\label{roman}
		G_{ab}=-G_{aa}\sum_{\alpha} x_{\alpha a}G_{\alpha b}^{(a)}=-G_{bb}\sum_{\beta} G_{a\beta}^{(b)}x_{\beta b}\,.
	\end{align}
	\item[-] For $\alpha\not=\beta$,
	\begin{align}\label{greek}
		G_{\alpha\beta}=-G_{\alpha\alpha}\sum_{a}x_{\alpha a}G_{a\beta}^{(\alpha)}=-G_{\beta\beta}\sum_{b} G_{\alpha b}^{(\beta)} x_{\beta b}\,.
	\end{align}
	
	\item[-] For any $a$ and  $\alpha$,
	\begin{align}\label{intersect}
		G_{a\alpha}=-G_{aa}\sum_{\beta}x_{\beta a}G_{\beta\alpha}^{(a)}=-G_{\alpha\alpha}\sum_{b} G_{ab}^{(\alpha)} x_{\alpha b}\,.
	\end{align}
	
	\item[-] For $A,B\not=C$,
	\begin{align}\label{basic resolvent}
		G_{AB}=G_{AB}^{(C)}+\frac{G_{AC}G_{CB}}{G_{CC}}\,.
	\end{align}
	
	\item[-]{\it Ward identity:} For any $a$,
	\begin{align}\label{ward}
		\sum_{b}|G_{ab}|^2=\frac{\im G_{aa}}{\eta}\,,
	\end{align}
	where $\eta=\im z$.
\end{itemize}
\end{lemma}
For the proof of Lemma~\ref{res id}, we refer to Lemma 4.2 in~\cite{bue}, Lemma 6.10 in~\cite{sse}, and equation (3.31) in~\cite{re}. 

Denote by $\E_A$ the partial expectation with respect to the $A$-th column/row of $H$ and set 
\begin{equation}\label{schur using Z}
Z_a\deq(1-\E_a)(X^*G^{(a)}X)_{aa},\qquad Z_\alpha\deq (1-\E_\alpha) (XG^{(\alpha)}X^*)_{\alpha\alpha}.
\end{equation}
Using $Z_A$, we can rewrite $G_{AA}$ as $G_{aa}^{-1}=-z-d^{-1}\wt m^{(a)}-Z_a$ and $G_{\alpha\alpha}^{-1}=-\sigma_\alpha^{-1}-m^{(\alpha)}-Z_\alpha$.

\begin{lemma}\label{cauchy interlacing} There is a constant $C$ such that, for any $z \in \C^+$, $A\in\llbracket 1,N+M\rrbracket$, we have
\begin{align}\label{m minus m^A}
	|m(z)-m^{(A)}(z)|\le \frac{C}{N\eta}\,.
\end{align}
Furthermore, since $C^{-1}N\leq M\leq CN$, for some constant $C>0$, we also have
\begin{align}
	|m(z)-m^{(A)}(z)|\le \frac{C}{M\eta}\,.
\end{align}
\end{lemma}
The lemma follows from Cauchy's interlacing property of eigenvalues of $H$ and its minor $H^{(A)}$. For a detailed proof we refer to~\cite{uwe}. For $\T\subset\llbracket 1,N+M\rrbracket$ with, say, $|\T|\le 10$, we obtain $|m-m^{(\T)}|\le \frac{C}{N\eta}$.

\subsection{Concentration estimates}
For $i\in \llbracket 1,N\rrbracket$, let $(X_i)$ and $(Y_i)$, be two families of random variables that 
\begin{equation}\label{moment condition} \E \caR_i=0,\quad \E |\caR_i|^2=1,\quad \E|\caR_i|^p\le c_p\quad(p\ge3),\end{equation}
$\caR_i=X_i,Y_i$, for all $p\in \N$ and some constants $c_p$, uniformly in $i\in\llbracket 1,N\rrbracket$.
We collect here some useful concentration estimates.

\begin{lemma}\label{lemma.LDE}
Let $(X_i)$ and $(Y_i)$ be independent families of random variables and let $(a_{ij})$ and $(b_i)$, $i,j\in\llbracket 1,N\rrbracket$, be families of complex numbers. Suppose that all entries $(X_i)$ and $(Y_i)$ are independent and satisfy~\eqref{moment condition}. Then we have the bounds

\begin{align}
	\left|\sum_i b_iX_i \right|&\prec \left(\sum_i |b_i|^2\right)^{1/2},\label{LDE1}\\[1mm]
	\left|\sum_i\sum_j a_{ij}X_iY_j \right|&\prec \left(\sum_{i,j} |a_{ij}|^2\right)^{1/2},\label{LDE2}\\[1mm]
	\left|\sum_i\sum_j a_{ij}X_iX_j - \sum_i a_{ii}X_i^2 \right|&\prec \left(\sum_{i,j} |a_{ij}|^2\right)^{1/2}.\label{LDE3}
\end{align}
If the coefficients $a_{ij}$ and $b_i$ are depend on an additional parameter $u$, then all of the above estimates are uniform in $u$, that is, the threshold $N_0=N_0(\ve,D)$ in the definition of $\prec$ depends only on the family $(c_p)$ from~\eqref{moment condition}; in particular, $N_0$ does not depend on $u$.
\end{lemma}

We also remark that
\begin{align}\label{bound on xij}
|x_{ij}|\prec \frac{1}{\sqrt{N}}\,, \qquad |x_{ij}|\prec \frac{1}{\sqrt{M}} \,,
\end{align}
which can be easily obtained from~\eqref{p moment bound} and remark~\ref{MNM}. 

\section{Proof of Main Results} \label{pf main theorem}

We begin this section by briefly outlining the idea of the proof. 

\begin{itemize}
\item To prove Theorem~\ref{general case - large d}, we follow the strategy in~\cite{dsjl}. Instead of directly analyzing the self-consistent equation~\eqref{mfc equation}, we convert it into an equation of $z$. Then, the location of the right edge of $\mu_{fc}$ and its local behavior can be proved by analyzing the behavior of $z$, which is considered as a function of $m_{fc}$, the Stieltjes transform of $\mu_{fc}$.

\item To prove Theorem~\ref{theorem:main}, we approximate $m$, the normalized trace of the resolvent, by $\mfc$ (Lemma~\ref{hat bound} and Proposition~\ref{proposition:step 2_4}). In the approximation, we introduce an intermediate random object $\whmfc$, which can be used to locate the extremal eigenvalues (Proposition~\ref{proposition:lambda_k}). Combining it with the approximate linearity of $m_{fc}$ (Lemma~\ref{mfc estimate}), we can prove Theorem~\ref{theorem:main}.

\item To prove Theorem~\ref{lemma:gaussian}, we first show that the location of the right edge of the spectrum exhibits a Gaussian fluctuation of order $M^{-1/2}$ by applying the central limit theorem for a function of the eigenvalues of $\Sigma$. We conclude the proof by showing that the distance between the largest eigenvalue and the right edge is of order $N^{-2/3}$ and hence negligible.
\end{itemize}

\subsection{Proof of Theorem~\ref{general case - large d}}\label{pf edge}
\begin{proof}[Proof of Theorem~\ref{general case - large d}]
Recall~\eqref{mfc equation}, which we rewrite as follows:
\begin{align}\label{mfc equation_z}
	z=-\frac{1}{\mfc(z)}+d^{-1}\int_\R\frac{t\dd\nu(t)}{1+t \mfc(z)}.
\end{align}
Let $\tau\deq 1/\mfc$, and consider $z$ as a function of $\tau$, which we call $F(\tau)$. We then have
\begin{align}
	F(\tau) \deq -\tau +d^{-1}\int_\R\frac{t \tau \dd\nu(t)}{\tau+t}.
\end{align}
Taking imaginary part on the both sides, then
\begin{align}
	\im F(\tau)=-\im\tau\left\{1-d^{-1}\int_\R\frac{t^2\dd\nu(t)}{(\re\tau+t)^2+(\im\tau)^2}\right\} \,.
\end{align}
Let 
\begin{align}
	H(\tau)\deq d^{-1}\int_\R\frac{t^2\dd\nu(t)}{(\re\tau+t)^2+(\im\tau)^2}\,.
\end{align}
For any fixed $\re\tau\in(-1,0)$,
$H(\tau) \to 0$ as $\lvert \im\tau \rvert \to \infty$, and $H(\tau) \to \infty$ as $\lvert \im\tau \rvert \to 0$. By monotonicity, there is a unique $\im\tau >0$ such that $H(\tau)=1$ so that $\im F(\tau)=0$, which corresponds to the bulk of the spectrum. On the other hand, for any fixed $\re\tau\in(-\infty,-1)$, $H(\tau)$ is monotone decreasing function of $\lvert \im\tau\rvert$, which implies 
\begin{align}
	\sup_{\re\tau\in(-\infty,-1)}\limits H(\tau)=H(-1)=d^{-1}\int_{l}^1 \frac{t^2\dd\nu(t)}{(-1+t)^2}=\frac{d_+}{d}<1 \,,
\end{align}
where $l=\inf \{x\in\mathbb{R} : x\in\supp \nu\}$.
We thus find that there is no solution of $\im F(\tau)=0$ when $\re\tau \in (-\infty,-1)$, which corresponds to the outside of the spectrum. This shows that $\tau=-1$ at the right edge of the spectrum. It is immediate from~\eqref{mfc equation_z} that $F(-1) = 1+\tau_+$, which is the end point we denoted by $L_+$. This proves the first part of Theorem~\ref{general case - large d}.

The proof of second part is analogous to Lemma A.4 of~\cite{dsjl} and we omit the detail.
\end{proof}
\subsection{Definition of $\whmfc$}
In this subsection, we introduce $\whmfc$, which will be used as an intermediate random object in the comparison between $m$ and $m_{fc}$. The key property of $\whmfc$ is that it directly depends on $\Sigma$ unlike $m_{fc}$, but it does not depend on $X$. 

Let~$\wh \nu$ be the ESD of $\Sigma$, i.e.,
\begin{align} \label{def hat nu}
\wh \nu \deq \frac{1}{M} \sum_{\alpha=1}^M \delta_{\sigma_\alpha}\,.
\end{align}
We define $\whmfc \equiv \whmfc(z)$ as a solution to the self-consistent equation
\begin{align} \label{eq:hat mfc}
\whmfc(z) =\left\{-z+ \frac{1}{N} \sum_{\alpha=1}^M \frac{\sigma_\alpha}{\sigma_\alpha \whmfc(z)+1}\ \right\}^{-1},\qquad \im\whmfc(z)\ge 0\,,\qquad (z\in\C^+)\,.
\end{align}
Similarly to~\eqref{mfc equation}, equation~\eqref{eq:hat mfc} also has the unique solution, which is the Stieltjes transform of a probability measure, $\wh\mu_{fc}$, which is absolutely continuous. The random measure  $\wh\nu \boxtimes \mu_{MP}$, which is the multiplicative free convolution between $\wh\nu$ and the Marchenko--Pastur law $\mu_{MP}$, and it can be recovered from $\whmfc$ by using the Stieltjes inversion formula~\eqref{stieltjes inversion}.

\subsection{Properties of $\mfc$ and $\whmfc$} \label{subsec:properties}

Recall the definitions of $\mfc$ and $\whmfc$. Let
\begin{align}\label{definition of R2 without hat}
R_2 (z) \deq d^{-1}\int \frac{t^2 |\mfc|^2\dd \nu(t)}{|t \mfc(z)+1|^2}, \quad \wh R_2 (z) \deq \frac{1}{N} \sum_{\alpha=1}^M \frac{\sigma_\alpha^2 |\whmfc|^2}{|\sigma_\alpha  \whmfc(z)+1|^2}\,,\quad (z\in\C^+) \,.
\end{align}
Recall from~\eqref{mfc equation} that 
\begin{equation}
\frac{1}{\mfc}=-z+d^{-1} \int \frac{t\dd\nu(t)}{t \mfc+1}.\end{equation} 
Taking imaginary part and rearranging, we have that
\begin{equation}
1=\im z \cdot \frac{|\mfc|^2}{\im \mfc}+d^{-1}\int \frac{t^2 |\mfc|^2\dd \nu(t)}{|t \mfc(z)+1|^2} \,.
\end{equation}
This in particular shows that $0\le R_2 (z)<1$, and by similar manner we also find that $0\le \wh R_2 (z) < 1$. We also note that the self-consistent equation~\eqref{mfc equation} implies $|\mfc|\sim1$.

In the following lemma, we show that $1/\mfc$ is approximately a linear function of $z$ near the right edge. 

\begin{lemma} \label{mfc estimate}
Let $z = L_+ - \kappa + \ii \eta \in \caD_{\phi}$. Then,
\begin{align}
	\frac{1}{\mfc(z)} = -1 + \frac{d}{d - d_+} (L_+ - z) + O \left( (\log M) (\kappa + \eta)^{\min \{ \b, 2 \} } \right)\,.
\end{align}
Similarly, if $z, z' \in \caD_{\phi}$, then
\begin{align}
	\frac{1}{\mfc(z)} - \frac{1}{\mfc(z')} = -\frac{d}{d - d_+} (z-z') + O \left( (\log M)^2 (N^{-1/(\b+1)})^{\min \{ \b-1, 1 \} } |z-z'|  \right)\,. 
\end{align}
\end{lemma}
\begin{proof}
Since
$\mfc(L_+)=-1=\frac{1}{-L_+ +\tau_+}$ (see theorem~\ref{general case - large d}), we have
\begin{align}\label{linear1} \begin{split}
		\frac{1}{\mfc(z)} - \frac{1}{\mfc(L_+)} &= L_+-z+ d^{-1}\int \frac{t\dd \nu (t)}{1+t \mfc(z)} - d^{-1}\int \frac{t\dd \nu (t)}{1 +t \mfc(L_+)} \\
		&=L_+ -z + d^{-1} \int \frac{t^2(\mfc(L_+) - \mfc(z))}{(1+t \mfc(z))(1+t\mfc(L_+))} \dd \nu (t)\\
		&=L_+-z +\left(\frac{1}{\mfc(z)}-\frac{1}{\mfc(L_+)}\right)T(z)
\end{split} \end{align}
where we set
\begin{align} \label{definition T}
	T(z) \deq d^{-1}\int \frac{t^2 \mfc(z) \mfc(L_+)}{(1+t \mfc(z))(1+t \mfc(L_+))} \dd \nu (t)\,.
\end{align}
Then we have
\begin{align}\label{definition of T1}
	\begin{split}
		|T(z)| &\le \left( d^{-1}\int \frac{t^2|\mfc|^2\dd \nu (t)}{|1+ t \mfc(z)|^2} \right)^{1/2} \left( d^{-1}\int \frac{t^2|\mfc(L_+)|^2\dd \nu (t)}{|1+ t \mfc(L_+)|^2} \right)^{1/2}\\ &\le \sqrt{R_2(z)} \sqrt{\frac{d_+}{d}} <\sqrt{\frac{d_+}{d}} < 1\,.
	\end{split}
\end{align}
Hence, for $z \in \caD_{\phi}$, we can rewrite~\eqref{linear1} as
\begin{align}
	\frac{1}{\mfc(z)} - \frac{1}{\mfc(L_+)} = L_+-z+ T(z) \left[ \frac{1}{\mfc(z)} - \frac{1}{\mfc(L_+)} \right] \,.
\end{align}
Since $\mfc(L_+)=-1$,
\begin{align} \label{Lipschitz estimate}
	\frac{1}{\mfc(z)}+1=\frac{1}{1-T(z)}(L_+-z)\,.
\end{align}
We thus obtain from~\eqref{definition of T1} and~\eqref{Lipschitz estimate} that
\begin{align}
	\left|\frac{1}{\mfc(z)}+1\right| \le \frac{1}{1-T(z)} |L_+ -z|\le \frac{\sqrt{d}}{\sqrt{d}-\sqrt{d_+}}|L_+-z|\,.
\end{align}

We now estimate the difference $T(z)-{d_+}/{d}\,$: Let $\tau \deq 1/\mfc(z)$. We have
\begin{align} \begin{split} \label{eq:estimate T}
		T(z) - \frac{d_+}{d} &= d^{-1}\int \frac{t^2 \mfc(z) \mfc(L_+)\dd \nu (t)}{(t \mfc(z)+1)(t \mfc(L+)+1)} -d^{-1} \int \frac{t^2\dd \nu (t)}{(1 - t)^2}\\ &= d^{-1} \int \frac{-t^2 (\mfc(z)+1)\dd \nu (t)}{(t \mfc(z)+1)(1-t)^2}
		=-(1+\tau)d^{-1} \int \frac{t^2 \dd \nu(t)}{(t+\tau)(1-t)^2} \,.
\end{split} \end{align}
To find an upper bound of such integral, we consider the following two cases: 
\begin{itemize}
	\item[Case 1)] $\b \ge 2$: It is not hard to see that
	\begin{align}
		\left|   \int \frac{t^2 \dd \nu(t)}{(t+\tau)(1-t)^2} \right| \le C \int \frac{\dd t}{|t+ \tau|} \le C \log M\,.
	\end{align}
	
	\item[Case 2)] $\b < 2$: We define a subset $B$ of $[0, 1]$ as 
	\begin{align}
		B \deq \{ t \in [0, 1] : t<-1-2\re\tau \}\,,
	\end{align}
	and let $B^c\equiv [0, 1] \backslash B$. Then, by estimating the integral in~\eqref{eq:estimate T} on $B$, we find that
	\begin{align}
		\left|  \int_{B} \frac{t^2 \dd \nu(t)}{(t+\tau)(1-t)^2} \right| \le C \int_{B}  \frac{t^2 \dd \nu(t)}{|1-t|^3} \le C |1 + \tau|^{\b-2}\,,
	\end{align}
	where used that, for $t \in B$,
	\begin{align} |1-t| < 2|t+\re\tau|<2|t+\tau| \,.\end{align}
	On $B^c$, we have
	\begin{align}\label{estimate T1}
		\left| \int_{B^c} \frac{t^2 \dd \nu(t)}{(t+\tau)(1-t)} \right| \le C \int_{B^c} \frac{ t^2(1-t)^{b-1} \dd t}{|t+\tau|} \le C \int_{B^c} \frac{(1-t)^{b-1} \dd t}{|t+\tau|}\le C |1+\tau|^{b-1}\log M\,, 
	\end{align}
	where we have used that, for $t \in B^c$,
	\begin{align} |1-t| \le 2|1+\re\tau|\le 2|1+\tau|.\end{align}
	We also have
	\begin{align}\label{estimate T2}
		\left| \int_{B^c} \frac{t^2\dd \nu (t)}{(1 - t)^2} \right| \le C \int_{B^c} |1 - t|^{\b -2} \dd t \le C |1 + \tau|^{\b -1}\,.
	\end{align}
	Thus, from~\eqref{eq:estimate T}, ~\eqref{estimate T1}, and~\eqref{estimate T2}, we obtain that
	\begin{align}
		\left| \int  \frac{t^2 \dd \nu(t)}{(t+\tau)(1-t)^2} \right| \le C |1 + \tau|^{\b-2} \log M\,.
	\end{align}
\end{itemize}
From the continuity of $T(z)$ and the compactness of $\caD_\phi$, it is easy to see that we can choose the constants uniformly in $z$. We thus have that
\begin{align}
	T(z) = \frac{d_+}{d} + O\left((\log M)|L_+ - z|^{\min \{ \b-1, 1 \} } \right)\,.
\end{align}
Combined with~\eqref{Lipschitz estimate}, it proves the first part of the desired lemma. The second one can be proved analogously; we omit the detail.
\end{proof}

\begin{remark} \label{remark:kappa}
Lemma~\ref{mfc estimate} reveals the local behavior of $1/m_{fc}$ at the right edge. For $z_\alpha\deq L_+-\frac{d-d_+}{d}(1-\sigma_\alpha)+\ii\eta $, we obtain
\begin{align}
	\frac{1}{ \mfc(z_\alpha)} =-\sigma_\alpha+  \ii\frac{d}{d - d_+} \eta + O\left((\log M) M^{-\min \{ \b, 2 \} / (\b+1) + 2 \phi} \right)\,.
\end{align}
\end{remark}

We consider the following subset of $\caD_{\phi}$ to estimate the difference $|\whmfc - \mfc|$.
\begin{definition}
Let $A\deq\llbracket n_0,M\rrbracket$. We define the domain $\caD_{\phi}'$ of the spectral parameter $z$ as
\begin{align}\label{a index assumption}
	\caD_{\phi}'=\left\{ z\in\caD_{\phi}\,:\, \left|1+\frac{1}{\sigma_\alpha \mfc}\right| > \frac{1}{2} M^{-1 / (\b + 1)-\phi},\,\forall \alpha\in A\right\}\,.
\end{align}
\end{definition}
In the sequel, we show that $\caD_{\phi}'$ contains $z=\lambda_\alpha+\ii \eta_0\in\mathbb{C}^+$ for $\alpha\in\llbracket 1, n_0-1\rrbracket$ with high probability. See Remark~\ref{remark:lambda_k}.

Recall that $\sigma_1 > \sigma_2 > \ldots > \sigma_M$. We now show that $\whmfc(z)$ approximates $\mfc(z)$ well for $z$ in $\caD_\phi'$. For technical reason, we compare the reciprocals of $\mfc$ and $\whmfc$, which makes the estimate more convenient when compared to estimating $|m-\whmfc|$ directly.
\begin{lemma} \label{hat bound}
For any $z \in \caD_{\phi}'$, 
\begin{align} \label{eq:hat bound}
	\left|\frac{1}{\whmfc (z)}-\frac{1}{\mfc(z)} \right| \prec \frac{1}{ M\eta_0}=M^{-1/2+\phi}\,.
\end{align}
\end{lemma}

\begin{proof}
For a given $z \in \caD_{\phi}'$, choose $\gamma \in \llbracket 1, n_0 -1 \rrbracket$ satisfying~\eqref{assumption near sigma_k} so that $\re (1/(\sigma_\gamma \mfc))$ is the closest (among $\re (1/(\sigma_\alpha \mfc))$) to $-1$. Suppose to contrary that~\eqref{eq:hat bound} does not hold. Our goal is to derive a self-consistent equation of the difference from which we obtain a contradiction. In other words, for any (small) $\epsilon>0$, we consider the event on which $|\whmfc(z)^{-1}-m_{fc}^{-1}|\geq M^{\epsilon}\frac{1}{M\eta_0}$ hold.
Using the definitions of $\mfc$ and $\whmfc$, we obtain the following equation:
\begin{align} \begin{split} \label{mfc difference}
		&\left|\frac{1}{\mfc}-\frac{1}{\whmfc}\right|= \left|d^{-1}\int \frac{t \dd \nu(t)}{t \mfc+1}-\frac{1}{N} \sum_{\alpha=1}^M \left( \frac{\sigma_\alpha}{\sigma_\alpha \whmfc+1}  \right)\right|\\
		\le &\left|d^{-1}\int \frac{t \dd \nu(t)}{t \mfc+1}-\frac{1}{N} \sum_{\alpha=1}^M \left( \frac{\sigma_\alpha}{\sigma_\alpha \mfc+1}  \right)\right|+ \left| \frac{1}{N} \sum_{\alpha=1}^M \left( \frac{\sigma_\alpha}{\sigma_\alpha \mfc+1}  \right)-\frac{1}{N} \sum_{\alpha=1}^M \left( \frac{\sigma_\alpha}{\sigma_\alpha \whmfc+1}  \right)\right| \\
		\le &\left|d^{-1}\int \frac{t \dd \nu(t)}{t \mfc+1}-\frac{1}{N} \sum_{\alpha=1}^M \left( \frac{\sigma_\alpha}{\sigma_\alpha \mfc+1}  \right)\right|+\left| \frac{1}{m_{fc}}-\frac{1}{\whmfc} \right| \left| \frac{1}{N} \sum_{\alpha=1}^M \frac{\sigma_\alpha m_{fc}\sigma_\alpha\whmfc}{(\sigma_\alpha m_{fc}+1)(\sigma_\alpha\whmfc+1)} \right| \,.
\end{split} \end{align}
From the condition~\eqref{assumption_CLT_1}, the first term in the right hand side of~\eqref{mfc difference} is bounded by $C M^{-1/2+\phi+\epsilon/2}$ for some $C>0$. 

Now we estimate the second term in the right hand side of~\eqref{mfc difference}. We decompose it into the critical term $\alpha = \gamma$ and the other terms.
When $\alpha=\gamma$, we have
\begin{align}
	\left|\frac{1}{\sigma_\gamma \whmfc}+1\right| + \left|-\frac{1}{\sigma_\gamma \mfc}-1\right| \ge \left|\frac{1}{\sigma_\gamma \whmfc} -\frac{1}{\sigma_\gamma \mfc}\right| >\left|\frac{1}{ \whmfc} -\frac{1}{ \mfc}\right| >\frac{M^{\epsilon}}{M\eta_0}\,,
\end{align}
which implies
\begin{align}
	\left|\frac{1}{\sigma_\gamma \whmfc}+1\right| \ge \frac{M^{\phi+\epsilon}}{2\sqrt M} \qquad \text{or} \qquad \left|\frac{1}{\sigma_\gamma \mfc}+1\right| \ge \frac{M^{\phi+\epsilon}}{2\sqrt M}\,.
\end{align}
In the former case, by considering the imaginary part, we find
\begin{align}\left|1+\frac{1}{\sigma_\gamma\mfc}\right|\ge \left|\im\frac{1}{\mfc}\right|\ge \eta+d^{-1}\int \frac{t^2\im\mfc}{|t\mfc+1|^2}\ge \eta, \end{align}
and hence we have
\begin{align}
	\frac{1}{N} \left| \frac{\sigma_\gamma \whmfc \sigma_\gamma \mfc}{(\sigma_\gamma \whmfc+1)(\sigma_\gamma \mfc+1)} \right| \le \frac{1}{N} \frac{2\sqrt M}{M^{\phi+\epsilon}} \frac{1}{\eta_0} \le C M^{-\epsilon}\,,\qquad\quad (z\in\caD_{\phi}')\,.
\end{align}
The latter case can be handled in a similar manner.
For the other terms with $\alpha\neq \gamma$, we use
\begin{align}
	\frac{1}{N} \left| \sum_{\alpha}^{(\gamma)} \frac{\sigma_\alpha \whmfc \sigma_\alpha \mfc}{(\sigma_\alpha \whmfc+1)(\sigma_\alpha \mfc+1)} \right| \le \frac{1}{2N} \sum_{\alpha}^{(\gamma)} \left( \frac{\sigma_\alpha^2 |\whmfc|^2}{|\sigma_\alpha \whmfc+1|^2} + \frac{\sigma_\alpha^2 |\mfc|^2}{|\sigma_\alpha \mfc+1|^2} \right)\,.
\end{align}
From~\eqref{eq:hat mfc}, we have that
\begin{align}\label{R2 hat less than 1}
	\frac{1}{N} \sum_{\alpha=1}^M \frac{\sigma_\alpha^2 |\whmfc|^2}{|\sigma_\alpha \whmfc+1|^2} = 1-\eta\frac{|\whmfc|^2}{\im \whmfc} < 1\,.
\end{align}
We also assume in~\eqref{assumption_CLT_2} that
\begin{align}
	\frac{1}{N} \sum_{\alpha}^{(\gamma)} \frac{\sigma_\alpha^2 |\mfc|^2}{|\sigma_\alpha \mfc+1|^2} < \mathfrak{c} < 1\,,
\end{align}	for some constant $\fc$. Thus, we get
\begin{align}
	\left|\frac{1}{\mfc}-\frac{1}{\whmfc}\right| < \frac{1+\mathfrak{c}}{2} \left|\frac{1}{\mfc}-\frac{1}{\whmfc}\right| + M^{-1/2 + \phi+\epsilon/2}\,,\quad\qquad (z\in\caD_{\phi}')\,,
\end{align}
which implies that
\begin{align}
	\left|\frac{1}{\mfc}-\frac{1}{\whmfc}\right| < C M^{-1/2 +\phi+\epsilon/2}\,,\qquad\quad (z\in\caD_{\phi}')\,,
\end{align}
which contradicts the assumption that~\eqref{eq:hat bound} does not hold. This concludes the proof of the desired lemma.
\end{proof}	

The following lemma provides priori estimate for imaginary part of $\whmfc$ with general $\eta$.
\begin{lemma}\label{weakboundwhmfc}
For $z=E+\ii\eta \in \caD_{\phi}'$, the following hold on $\Omega$:
\begin{align}
	\im \whmfc(z) = O(\max\{\eta,\frac{1}{N\eta}\})=O(\max\{\eta,\frac{1}{M\eta}\}) \,.
\end{align}
\end{lemma}
\begin{proof}
By the definition of $\whmfc$,
\begin{align}
	\begin{split}
		-\frac{1}{\whmfc}&=z-\frac{1}{N}\sum_{\alpha=1}^{M} \frac{1}{\sigma_\alpha^{-1}+\whmfc}=z-\frac{1}{N} \frac{1}{\sigma_\gamma^{-1}+\whmfc} - \frac{1}{N}\sum_{\alpha}^{(\gamma)}\frac{1}{\sigma_\alpha^{-1}+\whmfc}\\
		&=z+O(\frac{1}{N\eta}) -\frac{1}{N}\sum_{\alpha}^{(\gamma)}\frac{1}{\sigma_\alpha^{-1}+\whmfc} \,,
	\end{split}	
\end{align}
where $\gamma$ satisfies~\eqref{assumption near sigma_k} and we have used the trivial bound $|\frac{1}{\sigma_\alpha^{-1}+\whmfc}|\le \eta^{-1}$.\\
Taking imaginary part gives
\begin{align}
	\begin{split}
		\frac{\im \whmfc}{|\whmfc|^2}=\eta +O(\frac{1}{N\eta})+\frac{1}{N}\sum_\alpha^{(\gamma)}\frac{\im \whmfc}{|\sigma_\alpha^{-1}+\whmfc|^2} \,, \\
		\im \whmfc =\eta |\whmfc|^2 +O(\frac{|\whmfc|^2}{N\eta}) +\frac{1}{N}\sum_\alpha^{(\gamma)} \frac{|\whmfc|^2\im\whmfc}{|\sigma_\alpha^{-1}+\whmfc|^2} \,.
	\end{split}
\end{align}
Recalling $\wh{R}_2^{(\gamma)}$ from Lemma~\ref{lemma:step 1} and using $\whmfc\sim 1$,
then we have
\begin{align}
	|\im\whmfc |\le C\eta +\frac{C'}{N \eta} \,.
\end{align}
\end{proof}

\begin{remark}
The estimate on $|\mfc-\whmfc|$ easily follows from Lemma~\ref{hat bound}. To see this, we first observe that $\mfc\sim1$ implies $\mfc^{-1}\sim1$. Combining with Lemma~\ref{hat bound} above, we also find that $\whmfc\sim1$. Since $|\mfc-\whmfc|\prec M^{-1/2+\phi}$, we get the estimate 
\begin{align}\label{mfcwhmfc}
	|\mfc-\whmfc|\le C M^{-1/2+2\phi} \,,
\end{align}
with high probability.
\end{remark}
\subsection{Proof of Theorem~\ref{theorem:main}}
In this subsection, we prove Proposition~\ref{proposition:main}, which would directly imply Theorem~\ref{theorem:main}. The key idea is that we can approximate $(\lambda_\gamma)$ in terms of $(\sigma_\gamma)$ by applying the properties of $\whmfc$ in Section~\ref{subsec:properties} and hence we can estimate the locations of the largest eigenvalues $(\lambda_\gamma)$ of $\caQ$ by $(\sigma_\gamma)$. The precise statement for the idea is the following proposition.

\begin{proposition} \label{proposition:lambda_k}
Let $n_0>10$ be a fixed integer independent of $M$ and $\gamma\in\llbracket 1, n_0-1\rrbracket$. 
Suppose that the assumptions in Theorem~\ref{theorem:main} hold. Then, with $\eta_0$ defined in~\eqref{definition of kappa0}, the following holds with high probability:
\begin{align}\label{implicit equation}
	\re \frac{1}{\whmfc (\lambda_\gamma + \ii \eta_0)} = -\sigma_\gamma + O(M^{-1/2 + 3\phi})\,,
\end{align}
\end{proposition}

We postpone the proof of Proposition~\ref{proposition:lambda_k} to Section~\ref{location}.

\begin{remark} \label{remark:lambda_k}
Since $|\sigma_\alpha - \sigma_\gamma| \geq M^{-\phi} \kappa_0 \gg M^{-1/2 + 3\phi}$ for all $\alpha \neq \gamma$ by \eqref{eq4.3}, Proposition~\ref{proposition:lambda_k} implies that
\begin{align}
	\begin{split}
		&\left|1+\re \frac{1}{\sigma_\alpha\widehat m_{fc}(\lambda_\gamma + \ii \eta_0)}\right| \\&\geq \left|\re \frac{1}{\sigma_\alpha\widehat m_{fc}(\lambda_\gamma + \ii \eta_0)}-\re \frac{1}{\sigma_\gamma\widehat m_{fc}(\lambda_\gamma + \ii \eta_0)}\right| - \left|1+\re \frac{1}{\sigma_\gamma\widehat m_{fc}(\lambda_\gamma + \ii \eta_0)}\right|\\
		&\geq \frac{ \kappa_0}{2}\,.
	\end{split}
\end{align}
Hence, we find that $\lambda_\gamma + \ii \eta_0 \in \caD_{\phi}'$, $\gamma\in\llbracket 1,n_0-1\rrbracket$ with high probability.
\end{remark}

We now prove Theorem~\ref{theorem:main} by proving the following proposition.

\begin{proposition} \label{proposition:main}
Suppose that the assumptions in Proposition ~\ref{proposition:lambda_k} hold. 
Then there exists a constant $C$ such that with high probability
\begin{align}
	\left| \lambda_\gamma - \left( L_+ - \frac{d - d_+ }{d} (1-\sigma_\gamma) \right) \right| \le \frac{C}{M^{1/(\b+1)}} \left(\frac{M^{3\phi}}{ M^{\fb}}  + \frac{(\log M)^2}{M^{1/(\b+1)}}  \right)\,.
\end{align}
\end{proposition}

\begin{proof}[Proof of Theorem~\ref{theorem:main} and Proposition~\ref{proposition:main}]
From Lemma~\ref{hat bound} and Proposition~\ref{proposition:lambda_k}, with high probability
\begin{align}
	\re \left( \frac{1}{\mfc(\lambda_\gamma+\ii\eta_0)} \right)=-\sigma_\gamma+O(M^{-\frac{1}{2}+3\phi})\,.
\end{align}
Recall we have proved in Lemma~\ref{mfc estimate} that
\begin{align}
	\frac{1}{\mfc(\lambda_\gamma+\ii\eta_0)} =-1+\frac{d}{d-d_+}(L_+-\lambda_\gamma-\ii\eta_0)+O(\kappa_0^{\min\{b,2\}}(\log M)^2)\,.
\end{align}
Thus,
\begin{align}
	\re \frac{1}{\mfc(\lambda_\gamma+\ii\eta_0)}= -1+\frac{d}{d-d_+}(L_+-\lambda_\gamma)+O(\kappa_0^{\min\{b,2\}}(\log M)^2)\,.
\end{align}
We now have with high probability that
\begin{align}
	\lambda_\gamma = -(1-\sigma_\gamma)\frac{d-d_+}{d}+L_+ + O(\kappa_0^{\min\{b,2\}}(\log M)^2)+ O(M^{-1/2+3\phi})
	\,,
\end{align}
which completes the proof of Proposition~\ref{proposition:main}.

To prove Theorem~\ref{theorem:main}, we notice that the distribution of the largest eigenvalue of $\Sigma$ is given by the order statistics of $(\sigma_\alpha)$. The Fisher--Tippett--Gnedenko theorem asserts that the limiting distribution of the largest eigenvalue of $\Sigma$ is a member of either Gumbel, Fr\`echet or Weibull family, and in our case it is the Weibull distribution. This completes the proof of Theorem~\ref{theorem:main}.
\end{proof}

The following corollary provides an estimate on the speed of the convergence

\begin{corollary} \label{corollary:main}
Suppose that the assumptions in Proposition ~\ref{proposition:lambda_k} hold. Then, there exists a constant $C_1$ such that for $s\in\R^+$
\begin{align} \begin{split}
		&\p \left( M^{1/(\b+1)} \frac{d - d_+}{d} (1-\sigma_\gamma) \leq s - C_1 \left( \frac{M^{3\phi}}{M^{\fb}} + \frac{ (\log M)^2}{M^{1/(\b+1)}} \right) \right) - C_1 \frac{(\log M)^{1+2\b}}{ M^{\phi}} \\
		&\leq \p \left( M^{1/(\b+1)} (L_+ - \lambda_\gamma) \leq s \right) \\
		&\leq \p \left( M^{1/(\b+1)} \frac{d - d_+}{d} (1-\sigma_\gamma) \leq s + C_1 \left( \frac{M^{3\phi}}{M^{\fb}} + \frac{ (\log M)^2}{M^{1/(\b+1)}} \right) \right) + C_1 \frac{(\log M)^{1+2\b}}{ M^{\phi}}\,,
\end{split} \end{align}
for any sufficiently large $N$.
\end{corollary}

\begin{remark}
The constants in Proposition~\ref{proposition:main} and Corollary~\ref{corollary:main} depend only on $d$, the measure $\nu$, and the constant $c_p$ in~\eqref{p moment bound}; in particular, they do not depend on the detailed structure of the sample $X$. 
\end{remark}

\subsection{Proof of Theorem~\ref{lemma:gaussian}} \label{Gaussian fluctuation}
In this subsection, we prove Theorem~\ref{lemma:gaussian} that holds in the case $d < d_+$ and the entries of $\Sigma$ are i.i.d. random variables. Recall that $\wh \mu_{fc} \deq \wh \nu \boxtimes \mu_{MP}$ and $L_+$ is the right edge of the support of $\mu_{fc}$.
\begin{proof}
Following the proof in~\cite{k, dsjl}, we find that $L_+$ is the solution of the equations
\begin{align}
	\frac{1}{\mfc ( L_+)}=-L_++d^{-1}\int\frac{t\dd\nu(t)}{1+t\mfc(L_+)}\,, \quad d^{-1}\int\left|\frac{t\mfc(L_+)}{1+t\mfc(L_+)}\right|^2\dd\nu(t) = 1\,
\end{align}
and similarly $\wh L_+$ is the solution of the equations
\begin{align}
	\frac{1}{\whmfc (\wh L_+)} =-\wh L_+ +\frac{1}{N} \sum_{\alpha=1}^M \frac{\sigma_\alpha}{1+\sigma_\alpha\whmfc(\wh L_+)}\,, \qquad
	\frac{1}{N} \sum_{\alpha=1}^M \left|\frac{\sigma_\alpha\whmfc(\wh L_+)}{1+\sigma_\alpha\whmfc(\wh L_+)}\right|^2 = 1\,.
\end{align}

Let $\tau = 1/\mfc(L_+)$ and $\wh \tau \deq 1/\whmfc(\wh L_+)$. Since $d < d_+$, we assume that
\begin{align}
	d^{-1}\int \frac{t^2\dd \nu(t)}{(1-t)^2} > 1 + \delta\,, \qquad \frac{1}{N} \sum_{\alpha=1}^M \frac{\sigma_\alpha^2}{(1-\sigma_\alpha)^2} > 1 + \delta
\end{align}
for some $\delta > 0$, where the second inequality holds with high probability.
From the assumption, we find that $\tau, \wh \tau < -1$. Thus,
\begin{align}\label{equation C3}
	0 &= \frac{1}{N} \sum_{\alpha=1}^M \frac{\sigma_\alpha^2}{(\wh \tau+\sigma_\alpha)^2} - 1 = \frac{1}{N} \sum_{\alpha=1}^M \frac{\sigma_\alpha^2}{(\wh \tau+\sigma_\alpha)^2} -\frac{1}{N} \sum_{\alpha=1}^M\frac{\sigma_\alpha^2}{(\tau+\sigma_\alpha)^2} +O( M^{-1/2}) \nonumber \\
	&= \frac{1}{N} \sum_{\alpha=1}^M \frac{(2\sigma_\alpha + \tau + \wh \tau)(\tau - \wh \tau)}{(\wh\tau+\sigma_\alpha)^2 (\tau+\sigma_\alpha)^2} + O(M^{-1/2})\,. 
\end{align}
We also notice that $2 \sigma_\alpha + \tau + \wh \tau < 0$.
Further, with high probability, $|\{ \sigma_\alpha : \sigma_\alpha < 1/2 \}| > cN$ for some constant $c > 0$ independent of $N$. Hence,
\begin{align}
	-\frac{1}{N} \sum_{\alpha=1}^M \frac{2\sigma_\alpha+\tau+\wh\tau}{(\wh\tau+\sigma_\alpha)^2 (\tau+ \sigma_\alpha)^2} > c' > 0\,,
\end{align}

for some constant $c'$ independent of $N$. Together with~\eqref{equation C3}, we thus find that
\begin{align}
	\tau - \wh \tau = O(M^{-1/2})\,.
\end{align}

We now have that
\begin{align}
	\wh \tau + \wh L_+ &= \frac{1}{N} \sum_{\alpha=1}^M \frac{\wh\tau\sigma_\alpha}{\wh \tau +\sigma_\alpha } =\frac{1}{N} \sum_{\alpha=1}^M \frac{\tau\sigma_\alpha}{\tau +\sigma_\alpha } + \frac{1}{N} \sum_{\alpha=1}^M \frac{\sigma_\alpha^2}{(\tau +\sigma_\alpha)^2}(\wh\tau-\tau) + O(M^{-1}) \nonumber \\
	&=L_+ + \tau +Y + (\wh \tau - \tau) + O(M^{-1})\,,
\end{align}
where the random variable $Y$ is defined by
\begin{align}
	Y \deq \frac{1}{N} \sum_{\alpha=1}^M \frac{\tau\sigma_\alpha}{\tau+\sigma_\alpha} - d^{-1}\int \frac{t\tau}{t+\tau}\dd \nu(t) = \frac{1}{N} \sum_{\alpha=1}^M \left( \frac{\tau\sigma_\alpha}{\tau+\sigma_\alpha} - \E \left[ \frac{\tau\sigma_\alpha}{\tau+\sigma_\alpha} \right] \right)\\
	=\frac{d^{-1}}{M} \sum_{\alpha=1}^M \left( \frac{\tau\sigma_\alpha}{\tau+\sigma_\alpha} - \E \left[ \frac{\tau\sigma_\alpha}{\tau+\sigma_\alpha} \right] \right)
\end{align}
By the central limit theorem, $Y$ converges to a centered Gaussian random variable with variance 
\begin{equation}
	(d^{2}M)^{-1} \left\{ \int\left|\frac{t\tau}{t+\tau}\right|^2\dd\nu(t)- \left(\int\frac{t\tau}{t+\tau}\dd\nu(t)\right)^2 \right\}. 
\end{equation}
Since $	\wh L_+ - L_+ = Y + O(M^{-1})$, this completes the proof of the desired lemma.
\end{proof}
With Lemma~\ref{hat bound}, adapting the idea of the proof of Lemma A.4 in~\cite{dsjl}, we find that $1+tm_{fc}(z)\sim 1$ and hence $1+\sigma_\alpha \whmfc(z) \sim 1$. Thus, our model satisfies Condition 1.1 in~\cite{uzb} so that Theorem 4.1 therein holds and we get
\begin{align}
|L_+ - \lambda_1| \prec M^{-2/3}.
\end{align}
From $|\wh L_+ - L_+| \sim M^{-1/2}$, we find that the fluctuation of the largest eigenvalue is dominated by the Gaussian distribution in Theorem~\ref{lemma:gaussian}. Furthermore, we also have proved the sharp transition between the Gaussian limit and Weibull limit as $d$ crosses $d_+$.

\section{Estimates on the Location of the Eigenvalues} \label{location}
In this section, our main object is the proof of Proposition~\ref{proposition:lambda_k}. Let $\wh{E}_\gamma\in\mathbb{R}$ be a solution $E=\wh{E}_\gamma$ to the equation
\begin{align}\label{definition of hatzk}
1+\re \frac{1}{\sigma_\gamma \whmfc(E+\ii\eta_0)}=0
\end{align}
where $\gamma \in \llbracket 1, n_0-1 \rrbracket$ and $\eta_0$ is defined in~\eqref{definition of kappa0}.
Considering Lemma~\ref{mfc estimate} and Lemma~\ref{hat bound}, it is easy to check that there is at least one such $\wh{E}_\gamma$. If there are multiple solutions to~\eqref{definition of hatzk}, we choose the largest one as $\wh{E}_\gamma$ and set $\wh{z}_\gamma \deq \wh{E}_\gamma+\ii\eta_0$.

The key argument in the proof of Proposition~\ref{proposition:lambda_k} is similar to that of section 5 of~\cite{eejl}. The main idea is that when $\mu_{fc}$ has a convex decay (see Theorem~\ref{general case - large d}.), the imaginary part of $m(z)$ has a peak if and only if 
\begin{align}
\im\left(\frac{\sigma_\gamma\whmfc(z)}{1+\sigma_\gamma\whmfc(z)}\right) \,, \quad (z\in\mathbb{C}^+) \,,
\end{align}

becomes large enough for some $\gamma\in \llbracket 1,n_0-1 \rrbracket$. Furthermore, since the locations of the eigenvalues $\lambda_\gamma$ are correspond to the positions of the peaks of $\im m$, we are able to estimate the location of the $\gamma$-th largest	eigenvalue in terms of $\sigma_\gamma$.


\subsection{Properties of $\whmfc$ and $m$}\label{Properties of widehat \mfc and m}

In order to prove Proposition~\ref{proposition:lambda_k}, we need an prior estimate on the difference between $m(z)$ and $\whmfc(z)$ so-called ``local law" where $z$ is close to the edge. However, it is more convenient to consider the difference between their reciprocal rather than dealing with $|m(z)-\whmfc(z)|$ directly. After that, we can use the fact that $|\whmfc|$ is bounded away from zero to recover the order of $|m(z)-\whmfc(z)|$. Recall the constant $\phi>0$ in~\eqref{phi condition} and the definition of the domain $\caD_{\phi}'$ in~\eqref{a index assumption}.
In the proof of Proposition~\ref{proposition:lambda_k}, we will use the following local law as an a priori estimate. 
\begin{proposition}{\emph{[Local law near the edge]}}\label{proposition:step 2_4}
We have on $\Omega$ that
\begin{align}
	\left|\frac{1}{m(z)}-\frac{1}{\whmfc(z)}\right| \prec \frac{1}{ M\eta_0} \,,
\end{align}
for all $z \in \caD_{\phi}'$.
\end{proposition}
\begin{remark}
Since we have $\whmfc \sim 1$, the Proposition~\ref{proposition:step 2_4} implies 
\begin{align}
	|m(z)-\whmfc(z)|\prec \frac{1}{M\eta_0}  \,.
\end{align}
\end{remark}
The proof of Proposition~\ref{proposition:step 2_4} is the content of the rest of this subsection.
In the rest of this section, we gather some properties of $\whmfc(z)$ and estimate $\im m(z)$ when $z=E+\ii\eta_0\in\caD_\phi'$.

Recall the definitions of $(\wh z_\gamma)$ in~\eqref{definition of hatzk}. We begin by deriving a basic property of $\whmfc(z)$ near $(\wh z_\gamma)$. Recall the definition of $\eta_0$ in~\eqref{definition of kappa0}.
\begin{lemma} \label{lemma:step 1}
For $z = E + \ii \eta_0 \in \caD_{\phi}'$, the following hold on $\Omega$:
\begin{enumerate}
	\item[$(1)$]  if $|z- \wh z_\gamma| \ge M^{-1/2 + 3\phi}$ for all $\gamma \in\llbracket 1,n_0-1\rrbracket$, then there exists a constant $C >1$ such that
	\begin{align}
		C^{-1} \eta_0 \le -\im \frac{1}{\whmfc (z)} \le C \eta_0\,.
	\end{align}
	\item[$(2)$] if $z= \wh z_\gamma$  for some $\gamma \in\llbracket 1,n_0-1\rrbracket$, then there exists a constant $C >1$ such that
	\begin{align}
		C^{-1} M^{-1/2} \le -\im \frac{1}{\whmfc (z)} \le C M^{-1/2}\,.
	\end{align}
\end{enumerate}
\end{lemma}

\begin{proof}
Recall that
\begin{align}
	\wh R_2 (z) = 1-\eta_0\frac{|\whmfc|^2}{\im \whmfc} = \frac{1}{N} \sum_{\alpha=1}^M \frac{\sigma_\alpha^2 |\whmfc|^2}{|\sigma_\alpha\whmfc(z)+1|^2}<1\,,\qquad\quad( z\in\C^+)\,,
\end{align}
c.f.,~\eqref{definition of R2 without hat}. For given $z \in \caD_{\phi}'$ with $\im z=\eta_0$, choose $\gamma \in \llbracket 1, n_0 -1 \rrbracket$ such that~\eqref{assumption near sigma_k} is satisfied. In the first case, where $|z- \wh z_\gamma| \gg M^{-1/2 + 2\phi}$, we find from Lemma~\ref{mfc estimate} and Lemma~\ref{hat bound} that 
\begin{align} \label{eq:step 1_1}
	\left|1+\re\frac{1}{\sigma_\gamma \whmfc}\right| \gg M^{-1/2 + 2\phi}.
\end{align}
Since $z=E+\ii\eta_0$ satisfies~\eqref{assumption near sigma_k}, we also find that 
\begin{align}
	\wh R_2^{(\gamma)}(z) \deq \frac{1}{N} \sum_{\alpha}^{(\gamma)} \frac{\sigma_\alpha^2 |\whmfc|^2}{|\sigma_\alpha\whmfc(z)+1|^2} = \frac{1}{N} \sum_{\alpha}^{(\gamma)}\frac{\sigma_\alpha^2 |\mfc|^2}{|\sigma_\alpha \mfc(z)+1|^2} + o(1) < \mathfrak{c} < 1\,,
\end{align}
for some constant $\mathfrak{c}$. Thus,
\begin{align}
	\wh R_2(z) = \frac{1}{N} \frac{\sigma_\gamma^2 |\whmfc|^2}{|\sigma_\gamma\whmfc(z)+1|^2} + \frac{1}{N} \sum_{\alpha}^{(\gamma)} \frac{\sigma_\alpha^2 |\whmfc|^2}{|\sigma_\alpha\whmfc(z)+1|^2} < \mathfrak{c}' < 1\,,
\end{align}
for some constant $\mathfrak{c}'$. Recalling that
\begin{align}
	\eta_0\frac{|\whmfc|^2}{\im \whmfc}=1-\wh R_2(z)\,,
\end{align}	
\begin{align}
	-\im\frac{1}{\whmfc}=\frac{\eta_0}{1-\wh R_{2}(z)} \,,
\end{align}
hence the statement $(1)$ of the lemma follows.

Next, we consider the second case: $z=\wh z_\gamma=\wh E_\gamma+\ii\eta_0$, for some $\gamma\in\llbracket 1,n_0-1\rrbracket$. We have
\begin{align} \label{mfc quadratic}
	-\im\frac{1}{\whmfc(\wh z_\gamma)}&=\eta_0+\frac{1}{N}\sum_{\alpha} \frac{\sigma_\alpha^2\im \whmfc(\wh z_\gamma)}{|\sigma_\alpha\whmfc(\wh z_\gamma)+1|^2}\\
	&=\eta_0+\frac{1}{N}\sum_{\alpha} \frac{\sigma_\alpha^2|\whmfc(\wh z_\gamma)|^2}{|\sigma_\alpha\whmfc(\wh z_\gamma)+1|^2}\frac{\im \whmfc(\wh z_\gamma)}{|\whmfc(\wh z_\gamma)|^2}
	\,,
\end{align}
then by	solving the quadratic equation above for $\im\whmfc(\wh z_\gamma)$, we obtain
\begin{align}
	C^{-1} M^{-1/2} \le -\im \frac{1}{\whmfc (\wh z_\gamma)} \le C M^{-1/2}\,,
\end{align}
which completes the proof of the lemma.
\end{proof}

From now on we prove the local law, Proposition~\ref{proposition:step 2_4}.
Define a $z$-dependent parameter
\begin{align}
\Psi\equiv \Psi(z)\deq \sqrt{\frac{\im m(z)}{M\eta}}+\frac{1}{M\eta} \,.
\end{align}

Now we estimate the imaginary part of $m(z)$ for the smallest $\eta=\eta_0$.
\begin{lemma} \label{lemma:step 2_1}
We have on $\Omega$ that, for all $z = E + \ii \eta_0 \in \caD_{\phi}'$, 
\begin{align}
	\im m(z) \prec \frac{1}{M\eta_0}\,.
\end{align}
\end{lemma}
\begin{proof}
Fix $\eta = \eta_0$. For given $z=E+\ii\eta_0 \in \caD_{\phi}'$, choose $k \in \llbracket 1, n_0 -1 \rrbracket$ such that~\eqref{assumption near sigma_k} is satisfied. Let $\epsilon>0$ be given.
Assume that 
\begin{align}
	\im m(z)>M^{\epsilon}\frac{1}{M\eta} \,.
\end{align}
\\
We define events 
\begin{align}
	\Omega_{1}\deq \bigcap_{\alpha} \{|Z_\alpha| \leq M^{\epsilon/6}\Psi \} \,, \\
	\Omega_2 \deq \bigcap_a \{ |Z_a|\leq M^{\epsilon/6}\Psi \} 
	\,,\\
	\Omega_3 \deq \bigcap_{i,j} \{|X_{i,j}|\leq \frac{M^{\epsilon/6}}{\sqrt{M}}\} \,.
\end{align}
Note that the concentration estimates in Lemma~\ref{lemma.LDE} implies 
\begin{align}
	Z_a \prec \Psi  \,,\quad Z_\alpha\prec \Psi \,,
\end{align}
so that $\Omega_1, \Omega_2$ and $\Omega_3$ holds with high probability.
Let $\Omega_\epsilon \deq \Omega_{1} \cap \Omega_2 \cap \Omega_3 $, by the concentration estimates and definition of stochastic dominance, there exists $N_0(\epsilon,D) \in \N$ such that
\begin{align}
	\mathbb{P}(\Omega_{\epsilon}) \geq 1-N^{-D}
\end{align}
for any $N\geq N_0(\epsilon, D)$. We assume that $\Omega_\epsilon$ holds for the rest of the proof.\\
First, considering the relation~\eqref{rel between m and wt m} and (Cauchy interlacing) Lemma~\ref{cauchy interlacing},
\begin{align}
	zm=\frac{1}{N}\sum_\alpha \frac{-\sigma_\alpha^{-1}}{\sigma_\alpha^{-1}+m^{(a)}+Z_a} -\frac{N-M}{N}=\frac{1}{N}\sum_\alpha\frac{-\sigma_\alpha^{-1}}{\sigma_\alpha^{-1}+m+\caO_\prec(\Psi)}-\frac{N-M}{N} \,.
\end{align}
In addition, we have
\begin{align} \begin{split} \label{Z_a neg wrt imm}
		|Z_a|\leq M^{\epsilon/6}\Psi&= M^{\epsilon/6} \sqrt{ \frac{ \im m(z)}{M \eta}}+\frac{M^{\epsilon/6}}{M\eta} \,.\\
	\end{split}
\end{align}
Applying the arithmetic geometric mean on the first term of the right hand side, we obtain
\begin{align}  \label{fluctuation estimate}
	|Z_a| \le M^{-\epsilon/6}\im m+M^{\epsilon/2}(M\eta)^{-1}+C\frac{M^{\epsilon/6}}{M\eta} \ll \im m \,. \end{align}
Thus we have $\Psi \ll \im m$.
Hence, we can get
\begin{align}\label{mbound}
	zm=\frac{1}{N}\sum_\alpha \frac{-\sigma_\alpha^{-1}}{\sigma_\alpha^{-1}+m+o(\im m)} -\frac{N-M}{N}\,.
\end{align}

We claim that $m\sim 1$. \\
If $m \ll 1$, since $\sigma_\alpha=O(1)$, the LHS of~\eqref{mbound} tends to $0$ while its RHS goes to $-1$ as $N$ goes to infinity. Similarly, we can derive a contradiction when $m \gg 1$ hence we can conclude that $m\sim 1$.\\
Taking imaginary part on~\eqref{rel between m and wt m}, then we obtain

\begin{align}
	E\im m +\eta \re m=\frac{1}{N}\sum_\alpha\frac{\sigma_\alpha^{-1}(\im m+o(\im m))}{|\sigma_\alpha^{-1}+m^{(\alpha)}+Z_\alpha|^2} \,,
\end{align}
\begin{align}
	E+\eta\frac{\re m}{\im m}=\frac{1}{N}\sum_\alpha\frac{\sigma_\alpha^{-1}(1+o(1))}{|\sigma_\alpha^{-1}+m^{(\alpha)}+Z_\alpha|^2} \,.
\end{align}
Since $E=O(1)$, $\re m =O(1)$ and $\im m \geq C\eta$ 
\begin{align}	
	\frac{1}{N}\sum_\alpha\frac{\sigma_\alpha^{-1}(1+o(1))}{|\sigma_\alpha^{-1}+m^{(\alpha)}+Z_\alpha|^2}=O(1) \,.
\end{align}
We claim that
\begin{align}
	\frac{1}{N}\sum_\alpha\frac{\sigma_\alpha^{-1}}{|\sigma_\alpha^{-1}+m^{(\alpha)}+Z_\alpha|^2}=O(1) \,.
\end{align}
Assume that the claim is not hold so that the summation diverges to infinity. For large enough $N$, we have
\begin{align}
	\frac{1}{N}\sum_\alpha\frac{\sigma_\alpha^{-1}(1/2)}{|\sigma_\alpha^{-1}+m^{(\alpha)}+Z_\alpha|^2} \leq
	\frac{1}{N}\sum_\alpha\frac{\sigma_\alpha^{-1}(1+o(1))}{|\sigma_\alpha^{-1}+m^{(\alpha)}+Z_\alpha|^2} \leq 	\frac{1}{N}\sum_\alpha\frac{\sigma_\alpha^{-1}(3/2)}{|\sigma_\alpha^{-1}+m^{(\alpha)}+Z_\alpha|^2} \,,
\end{align}
then we have a contradiction since the first and the last terms goes to infinity while the middle term is bounded.\\
Hence we have

\begin{align}
	0\leq \frac{1}{N}\sum_\alpha \frac{1}{|\sigma_\alpha^{-1}+m^{(\alpha)}+Z_\alpha|^2}\leq \frac{1}{N}\sum_\alpha \frac{\sigma_\alpha^{-1}}{|\sigma_\alpha^{-1}+m^{(\alpha)}+Z_\alpha|^2}=O(1) \,.
\end{align}

Recalling the equation~\eqref{mbound}, we can derive
\begin{align}\begin{split}
		zm+1&=\frac{1}{N}\sum_\alpha\frac{-\sigma_\alpha^{-1}}{\sigma_\alpha^{-1}+m^{(\alpha)}+Z_\alpha}+\frac{M}{N}=\frac{1}{N}\sum_\alpha\left(\frac{-\sigma_\alpha^{-1}}{\sigma_\alpha^{-1}+m^{(\alpha)}+Z_\alpha}+1 \right)\\
		&=\frac{1}{N}\sum_\alpha\frac{m^{(\alpha)}+Z_\alpha}{\sigma_\alpha^{-1}+m^{(\alpha)}+Z_\alpha} =\frac{1}{N}\sum_\alpha\frac{m+m^{(\alpha)}-m+Z_\alpha}{\sigma_\alpha^{-1}+m^{(\alpha)}+Z_\alpha} \,.
	\end{split}
\end{align}

Since 
\begin{align}
	\frac{1}{N}\sum_\alpha\frac{1}{|\sigma_\alpha^{-1}+m^{(\alpha)}+Z_\alpha|^2} =O(1),
\end{align}
we can observe that
\begin{align*}
	\begin{split}
		\left|\frac{1}{N}\sum_\alpha\frac{m^{(\alpha)}-m+Z_\alpha}{(\sigma_\alpha^{-1}+m^{(\alpha)}+Z_\alpha)}\right|
		&\leq \frac{1}{N}\sum_\alpha \frac{|m^{(\alpha)}-m+Z_\alpha|}{|\sigma_\alpha^{-1}+m^{(\alpha)}+Z_\alpha|} \\
		&\leq \left(\frac{1}{N}\sum_\alpha \frac{1}{|\sigma_\alpha^{-1}+m^{(\alpha)}+Z_\alpha|^2}\right)^{\frac{1}{2}}  \left(\frac{1}{N}\sum_\alpha \left| m^{(\alpha)}-m+Z_\alpha \right|^2\right)^{\frac{1}{2}}\\& \ll o(\im m)
	\end{split}
\end{align*}

where we have used Cauchy inequality. \\
Hence we have
\begin{align}
	zm+1 =md^{-1}\wt{m}+o(\im m) \,,
\end{align}
so that
\begin{align}\label{minverse}
	z+\frac{1}{m}=-d^{-1}\wt{m}+o(\im m) \,.
\end{align}

Reasoning as in the proof of Lemma~\ref{hat bound}, we find the following equation for $m -\whmfc$ :
\begin{align} \begin{split} \label{m difference}
		|m - \whmfc| &=|m \whmfc|\left|\frac{1}{m}-\frac{1}{\whmfc}\right| \\&=|m\whmfc| \left| -d^{-1}\wt{m}-z+o(\im m) -\left( -z+\frac{1}{N}\sum_\alpha\frac{1}{\sigma_\alpha^{-1}+\whmfc}\right) \right| \\
		&=|m\whmfc|\left| \frac{1}{N}\sum_\alpha\frac{1}{\sigma_\alpha^{-1}+m^{(\alpha)}+Z_\alpha}-\frac{1}{N}\sum_\alpha\frac{1}{\sigma_\alpha^{-1}+\whmfc} +o(\im m) \right|
\end{split} \end{align}

Note that the assumption $\im m > M^{\epsilon}(M\eta)^{-1}$, Lemma~\ref{lemma:step 1} and boundedness of $m,\whmfc$ imply that 
\begin{align}\im \whmfc \ll \im m \,.\end{align}
Thus we have	
\begin{align}
	| m-\whmfc | \ge | \im m - \im\whmfc |=|\im m - o(\im m)| > CM^{\epsilon}\frac{1}{M\eta} \,. 
\end{align}
So we can conclude that $o(\im m)=o(|m-\whmfc|)$ and
\begin{align} \begin{split} 
		|m - \whmfc|&=|m\whmfc|\left| \frac{1}{N}\sum_\alpha\frac{1}{\sigma_\alpha^{-1}+m^{(\alpha)}+Z_\alpha}-\frac{1}{N}\sum_\alpha\frac{1}{\sigma_\alpha^{-1}+\whmfc} +o(\im m) \right| \\
		&\leq|m\whmfc|\left| \frac{1}{N}\sum_\alpha\frac{1}{\sigma_\alpha^{-1}+m^{(\alpha)}+Z_\alpha}-\frac{1}{N}\sum_\alpha\frac{1}{\sigma_\alpha^{-1}+\whmfc}  \right| +|m\whmfc|o(|m-\whmfc|) \\
		&\leq|m\whmfc|\left| \frac{1}{N}\sum_\alpha\frac{1}{\sigma_\alpha^{-1}+m^{(\alpha)}+Z_\alpha}-\frac{1}{N}\sum_\alpha\frac{1}{\sigma_\alpha^{-1}+\whmfc}  \right| +o(|m-\whmfc|)
\end{split} \end{align}
where we have used $m\sim 1$ and $\whmfc \sim 1$. \\

Abbreviate
\begin{align} \label{T_m}
	T_m\equiv T_m(z) \deq \frac{1}{N} \sum_{\alpha} \left| \frac{m \whmfc}{(m^{(\alpha)}+Z_{\alpha}+\sigma_\alpha^{-1})( \whmfc+\sigma_{\alpha}^{-1})} \right|\,.
\end{align}

We notice that 
\begin{align}
	z+\frac{1}{m}+o(\im m)=-d^{-1}\wt{m}=\frac{1}{N}\sum_\alpha \frac{1}{m^{(\alpha)}+Z_{\alpha}+\sigma_\alpha^{-1}} \,.
\end{align}
Taking imaginary part,
\begin{align}
	\eta-\frac{\im m}{|m|^2} +o(\im m) =\frac{1}{N}\sum_\alpha\frac{-\im m(1+o(1))}{|m^{(\alpha)}+Z_{\alpha}+\sigma_\alpha^{-1}|^2} \,,
\end{align}
\begin{align}
	1-\eta\frac{|m|^2}{\im m}+o(1)=\frac{1}{N}\frac{|m|^2}{|m^{(\alpha)}+Z_{\alpha}+\sigma_\alpha^{-1}|^2} 
\end{align}
thus
\begin{align}
	\frac{1}{N}\sum_\alpha\frac{|m|^2}{|m^{(\alpha)}+Z_{\alpha}+\sigma_\alpha^{-1}|^2} \leq 1 \,.
\end{align}
We get from Lemma~\ref{hat bound} that on $\Omega$,
\begin{align}
	\frac{1}{N} \sum_\alpha^{(\gamma)} \frac{|\whmfc|^2}{|\sigma_\alpha^{-1}+\whmfc|^2} = \frac{1}{N} \sum_\alpha^{(\gamma)} \frac{(1+o(1))|\mfc|^2}{|\sigma_\alpha^{-1}+\mfc|^2} < \mathfrak{c} < 1\,,
\end{align}
for some constant $\mathfrak{c} > 0$, and
\begin{align}
	\frac{1}{N} \left| \frac{m \whmfc}{(\sigma_\gamma^{-1} + m^{(\gamma)} +Z_\gamma) (\sigma_\gamma^{-1} + \whmfc)} \right| \le C \frac{1}{N} \left(\frac{M^\epsilon}{M\eta}\right)^{-1}\frac{1}{\eta} \le CM^{-\epsilon}\,.
\end{align}
Hence, we find that $T_m < \mathfrak{c}' < 1$ for some constant $c'$. Now, if we let
\begin{align}
	M \deq\max_{\alpha} |m^{(\alpha)} - m + Z_{\alpha}|\,,
\end{align}
then $M \ll |m - \whmfc|$. Thus, from~\eqref{m difference}, we get
\begin{align}
	|m - \whmfc| \le T_m (|m - \whmfc| + M)+o(1)|m-\whmfc| = \left( T_m + o(1) \right) |m - \whmfc|\,,
\end{align}
contradicting $T_m < c' < 1$.

Thus on $\Omega$, we have shown that for fixed $z\in\caD_\phi'$,
\begin{align}
	\im m(z) \le M^{\epsilon}\frac{1}{M\eta_0} \,,
\end{align}
with high probability.

Now it remains to prove the bound holds uniformly on $z$. We use the lattice argument which appears in~\cite{eejl}. For any fixed $z$ at which the assumption of the lemma satisfied, we construct a lattice $\caL$ from $z'=E'+\ii\eta_0 \in \caD_\phi'$ with $|z-z'|\leq M^{-3}$. It is obvious that the bound holds uniformly on $\caL$. For any $z=E+\ii\eta_0\notin \caL $, note that if $z'\in \caL$ and $|z-z'|\leq M^{-3}$, $|m(z)-m(z')|\leq \eta_0^{-2}|z-z'|$. Therefore, we conclude the proof.

\end{proof}
As a corollary of above lemma, we have a bound for $Z_a$ and $Z_\alpha$ in \eqref{schur using Z}. The concentration estimate implies that 
\begin{align}
|Z_\alpha| \prec \sqrt{\frac{\im m^{(\alpha)}}{N\eta}} \,, \quad
|Z_a|\prec\sqrt{\frac{\im \wt{m}^{(a)}}{M\eta}} \,.
\end{align}
The relation~\eqref{rel between m and wt m} and Lemma~\ref{cauchy interlacing} (the Cauchy interlacing property) implies that
\begin{align}
|Z_\alpha|\prec \sqrt{\frac{\im m}{N\eta}}+\frac{1}{N\eta} \,, \quad |Z_a| \prec \sqrt{\frac{\im m}{M\eta}}+\frac{1}{M\eta} \,.
\end{align}
Hence, as a corollary of Lemma~\ref{lemma:step 2_1}, we obtain:
\begin{corollary} \label{corollary:step 2_1}
We have on $\Omega$ that for all $z = E + \ii \eta_0 \in \caD_{\phi}'$,
\begin{align}
	\max_A |Z_A(z)| \prec \frac{1}{M\eta_0}\,, \qquad \max_{A} |Z_{A}^{(B)}(z)| \prec \frac{1}{M\eta_0}\,,\qquad\quad (B\in\llbracket 1,N+M\rrbracket)\,.
\end{align}
\end{corollary} 

Now, we prove the local law. To estimate the difference $\Lambda(z) \deq |m(z)-\whmfc(z)|$, we consider the imaginary part of $z$, $\eta$, to be large. Lemma~\ref{lemma:step 2_2} shows that $\Lambda$ satisfies local law for such $\eta$. After that, we prove that if $\Lambda$ has slightly bigger upper bound than our local law, we can improve the upper bound to the local law level (see Lemma~\ref{lemma:step 2_3}). Moreover, the Lipschitz continuity of the Green function and $\whmfc$ lead us to obtain that if $z$ satisfies our local law, then for any $z'$ close enough to $z$ also satisfies the bound. Applying the argument repetitively, we finally prove Proposition~\ref{proposition:step 2_4}.

\begin{lemma} \label{lemma:step 2_2}
We have on $\Omega$ that for all $z = E + \ii \eta \in \caD_{\phi}'$ with $M^{-1/2 + \phi} \le \eta \le M^{-1/(b+1)+\phi}$,
\begin{align} \label{eq:boot}
	| m(z)-\whmfc(z) | \prec \frac{1}{M\eta_0}\,.
\end{align}
\end{lemma}

\begin{proof}
We mimic the proof of Lemma~\ref{lemma:step 2_1}. Fix $z\in\caD_{\phi}'$ and $\epsilon>0$ be given. Similar to proof of Lemma~\ref{lemma:step 2_1},	
suppose that $|m(z)-\whmfc(z) | > M^{\epsilon}(M\eta_0)^{-1}$. Recall the definition of $\Omega_\epsilon$ from proof of Lemma~\ref{lemma:step 2_1} and assume that $\Omega_{\epsilon}$ holds. Consider the self-consistent equation~\eqref{m difference} and define $T_m$ as in~\eqref{T_m}.

Since $\im m(E+\ii\eta)\ge C\eta$, for $z\in\caD_{\phi}'$ and on $\Omega$, we have
\begin{align}
	\frac{1}{M\eta} \le M^{-1/2-\phi} \ll M^{-1/2+\phi}\le \eta \le C\im m \,.
\end{align}
Thus we eventually get the equation~\eqref{minverse},
\begin{align}
	z+\frac{1}{m}=-d^{-1}\wt{m} +o(\im m) \,.
\end{align}
However, in this lemma, $o(\im m)$ is not enough to proceed further. Thus we need more optimal order of $|m-m^{(\alpha)}|$ and $|Z_\alpha|$.\\
We already have
\begin{align}
	\frac{1}{M\eta}  \ll M^{\epsilon}\frac{1}{M\eta_0} < |m-\whmfc| \,,
\end{align}
hence by the Cauchy interlacing property, $|m-m^{(\alpha)}|=o(|m-\whmfc|)$.\\

Considering the definition of $\Omega_\epsilon$ from the proof of Lemma~\ref{lemma:step 2_1}, Concentration esimate implies that 
\begin{align}
	\begin{split}
		|Z_\alpha|&\leq M^{\epsilon/6}\Psi=M^{\epsilon/6}\sqrt{\frac{|\im m-\im \whmfc +\im\whmfc|}{M\eta}}+\frac{M^{\epsilon/6}}{M\eta}\\&\leq M^{\epsilon/6}\sqrt{\frac{|\im m-\im \whmfc|}{M\eta}}+M^{\epsilon/6}\sqrt{\frac{\im\whmfc}{M\eta}}+\frac{M^{\epsilon/6}}{M\eta} \,.
	\end{split}
\end{align}
The first term is $o(|m-\whmfc|)$ by assumption and the arithmetic geometric mean. For the second term, we use the prior bound for $\im \whmfc$ from Lemma~\ref{weakboundwhmfc} which implies
\begin{align}
	\sqrt{\frac{\im \whmfc}{M\eta}}=O\left(\max\left\{ \sqrt{\frac{1}{M}},\frac{1}{M\eta} \right\}\right) =O\left(\frac{1}{M\eta}\right)\,,
\end{align}
hence we have $|Z_\alpha|\ll|m-\whmfc|$ on $\Omega_\epsilon$. Hence we have 
\begin{align}
	z+\frac{1}{m}=-d^{-1}\wt{m}+o(|m-\whmfc|).
\end{align}
\\	
Then argue analogously as the proof of Lemma~\ref{lemma:step 2_1}, it contradicts to the assumption $|m(z)-\whmfc(z)| > M^\epsilon(M\eta_0)^{-1}$. To get a uniform bound, we again use the lattice argument as in the proof of Lemma~\ref{lemma:step 2_1}. This completes the proof of the lemma.
\end{proof}

\begin{lemma} \label{lemma:step 2_3}
Let $z \in \caD_{\phi}'$. Assume that $|m(z)-\whmfc(z)| \prec M^\phi (M\eta_0)^{-1} $,then we have on $\Omega$ that
\begin{align}\label{lemma:step 2_3_2}
	|m(z)-\whmfc(z)|\prec \frac{1}{M\eta_0} \,.
\end{align}
\end{lemma}

\begin{proof}
Since the proof closely follows the proof of Lemma~\ref{lemma:step 2_1}, we only check the main steps here. Fix $z\in\caD_{\phi}'$, $\epsilon>0$ be given and choose $\gamma \in \llbracket 1, n_0 -1 \rrbracket$ such that~\eqref{assumption near sigma_k} is satisfied.
Assume that  $M^\epsilon(M\eta_0)^{-1}  < |m(z)-\whmfc(z)| \le M^\epsilon M^\phi(M\eta_0)^{-1}$ and $\Omega_{\epsilon}$ hold.
Since $\whmfc \sim 1$, by the assumption, we can get $m\sim 1$.\\		
First, we estimate $|Z_\alpha|$. By the assumption, $|\im m -\im \whmfc|\leq M^{\phi+\epsilon}(M\eta_0)^{-1}$, and the definition of $\Omega_\epsilon$ we obtain
\begin{align}\begin{split}
		|Z_\alpha|\leq M^{\epsilon/6}\Psi&=M^{\epsilon/6}\sqrt{\frac{\im m -\im \whmfc +\im \whmfc}{M\eta}}+\frac{M^{\epsilon/6}}{M\eta}\\
		&\le M^{\epsilon/6}\sqrt{\frac{\im m-\im \whmfc}{M\eta}}+M^{\epsilon/6}\sqrt{\frac{\im\whmfc}{M\eta}}+\frac{M^{\epsilon/6}}{M\eta}\\
		&\ll |\im m -\im\whmfc|\,.
	\end{split}
\end{align}

Now we consider the self-consistent equation~\eqref{m difference} and define $T_m$ as in~\eqref{T_m}. We now estimate $T_m$. For $\alpha \neq \gamma$, $\alpha\in\llbracket 1,M\rrbracket$, we need to compare
\begin{align}
	A\deq\frac{m}{\sigma_{\alpha}^{-1}+m^{(\alpha)}+Z_\alpha} \quad\text{and} \quad B\deq\frac{\whmfc}{\sigma_{\alpha}^{-1}+\whmfc}.
\end{align}
Considering, 
\begin{align}
	\left| \frac{B}{A} \right| =\left| \frac{\whmfc}{m} \left( \frac{\sigma_{\alpha}^{-1}+m^{(\alpha)}+Z_\alpha}{\sigma_{\alpha}^{-1}+\whmfc} \right) \right|.
\end{align}
In addition, Lemma~\ref{cauchy interlacing}, Lemma~\ref{weakboundwhmfc} and the assumption imply that 
\begin{align} \begin{split}
		|m^{(\alpha)}-\whmfc +Z_\alpha| &\le  |m - m^{(\alpha)}| + |m - \whmfc| + |Z_\alpha| \\
		&\le  \frac{1}{M \eta} + M^{\phi+\epsilon}\frac{1}{M\eta_0} +  M^{\epsilon/6}\Psi \\
		&\ll M^{-\epsilon}\kappa_0\leq |\sigma_{\alpha} - \sigma_{\gamma}|\,,
\end{split} \end{align}
which holds for large enough $M$ on $\Omega$.
Also by the assumption,
\begin{align}
	\frac{\whmfc}{m}=1+o(1)\left|\frac{1}{m}\right|.
\end{align}
Hence,
\begin{align} \begin{split}
		\left| \frac{B}{A} \right| & = \left| \frac{\whmfc}{m} \right| \left| \left( \frac{\sigma_{\alpha}^{-1}+m^{(\alpha)}+Z_\alpha}{\sigma_{\alpha}^{-1}+\whmfc} \right) \right|=\left| \frac{\whmfc}{m} \right| \left| \frac{\sigma_{\alpha}^{-1}+\whmfc+o(M^{-\phi}\kappa_0) }{\sigma_{\alpha}^{-1}+\whmfc} \right|\\
		&=\left| \frac{\whmfc}{m} + \frac{o(M^{-\phi}\kappa_0)\whmfc }{m (\sigma_{\alpha}^{-1}+\whmfc)}   \right| =\left| 1+o(1)\frac{1}{m} \right| \,, \end{split} 
\end{align}
where we have used~\eqref{farfromhome}. Furthermore, by the fact $ \whmfc \sim 1$, we have $m \sim 1$ so that \begin{align}\left| \frac{B}{A} \right|=1+o(1) .\end{align}
Thus \begin{align}\frac{1}{N}\sum\limits_\alpha^{(\gamma)}\frac{m\whmfc}{(\sigma_{\alpha}^{-1}+m^{(\alpha)}+Z_\alpha)(\sigma_{\alpha}^{-1}+\whmfc)} =\frac{1}{N}\sum\limits_\alpha^{(\gamma)} AB=\frac{1}{N}\sum\limits_\alpha^{(\gamma)} B^2(1+o(1))<c<1.\end{align}\\
For $\alpha=\gamma$, we have
\begin{align}
	|\sigma_{\gamma}^{-1}+m^{(\gamma)}+Z_{\gamma}| + |\sigma_{\gamma}^{-1}+\whmfc|
	\ge \left| |m - \whmfc| - |m - m^{(\gamma)}| - |Z_\gamma| \right|
	\ge \frac{1}{2}M^\epsilon(M\eta_0)^{-1}\,,
\end{align}
thus, as in the proofs of Lemma~\ref{hat bound} and Lemma~\ref{lemma:step 2_1},
\begin{align}
	\frac{1}{N} \left| \frac{m\whmfc}{(\sigma_{\gamma}^{-1}+m^{(\gamma)}+Z_{\gamma}) (\sigma_{\gamma}^{-1}+\whmfc)} \right| \le C M^{-\epsilon}\,,
\end{align}
where we used trivial bounds $|G_{\gamma\gamma}|\,,|\frac{\whmfc}{\sigma_\gamma^{-1}+\whmfc}|\le \eta_0^{-1}$.

We now have that
\begin{align}
	T_m = \wh R_2^{(k)} + o(1) = R_2 + o(1)\,,
\end{align}
and, in particular, $T_m < c < 1$, with high probability on $\Omega$. Now we also apply the argument from Lemma~\ref{lemma:step 2_1} again to obtain the desired lemma.
\end{proof}

We now prove Proposition~\ref{proposition:step 2_4} using a discrete continuity argument.

\begin{proof}[Proof of Proposition~\ref{proposition:step 2_4}]
Fix $E$ such that $z=E+\ii\eta_0\in\caD_{\phi}'$. Consider a sequence $(\eta_j)$ defined by $\eta_j= \eta_0+j M^{-2}$. Let $K$ be the smallest positive integer such that $\eta_K \ge M^{-1/2 + \phi}$. We use mathematical induction to prove that for $z_j = E + \ii \eta_j$, we have on $\Omega$ that
\begin{align}
	|m(z_j)-\whmfc(z_j)  | \prec \frac{1}{M\eta_0}\,,
\end{align}
which implies that for any $\epsilon>0$, $\p(|m(z_j)-\whmfc(z_j)|\le\frac{M^{\epsilon}}{M\eta_0})\ge1-M^{-D}$ for large enough $M$. On this event, the case $j=K$ is already proved in Lemma~\ref{lemma:step 2_2}. For any $z = E + \ii \eta$, with $\eta_{j-1} \le \eta \le \eta_j$, we have
\begin{align}
	|m(z_j) - m(z)| \le \frac{|z_j - z|}{\eta_{j-1}^2} \le \frac{M^{2 \phi}}{M}\,, \qquad\quad |\whmfc(z_j) - \whmfc(z)| \le \frac{|z_j - z|}{\eta_{j-1}^2} \le \frac{M^{2 \phi}}{M}\,.
\end{align}
Thus, we find that if $|\whmfc(z_j) - m(z_j)| \prec (M\eta_0)^{-1}$ then
\begin{align}
	|m(z)-\whmfc(z) | \le |\whmfc(z_j)-m(z_j)| + \frac{2 M^{2 \phi}}{M} \prec M^\phi(M\eta_0)^{-1}\,.
\end{align}
We now refer Lemma~\ref{lemma:step 2_3} to obtain that $|m(z)-\whmfc(z) | \prec (M\eta_0)^{-1}$. This proves the desired lemma for any $z = E + \ii \eta$, with $\eta_{j-1} \le \eta \le \eta_j$. By induction on $j$, the desired lemma can be proved. Uniformity can be obtained by lattice argument.
\end{proof}

\subsection{Estimates on $|\wt m - \wt m^{(\alpha)}|$}\label{aux estimate 1}
Since we need a more precise estimate on the difference $|\im m(z) - \im\whmfc(z)|$, we construct tighter estimates on $|\wt m - \wt m^{(\alpha)}|$ and $N^{-1}\sum Z_A$. In this section, we provide enhanced bound on the difference $|\wh{m}-\wh{m}^{(\alpha)}|$. 


\begin{lemma} \label{lemma:step 3}
The following bounds hold on $\Omega$ for all $z = E + \ii \eta_0 \in \caD_{\phi}'$: For given $z$, choose $\gamma \in \llbracket 1, n_0 -1 \rrbracket$ such that~\eqref{assumption near sigma_k} is satisfied. Then, for any $\alpha \neq \gamma$, $\alpha\in\llbracket 1,M\rrbracket$,
\begin{align}
	|\wt{m}-\wt{m}^{(\gamma)}|\prec \frac{1}{M\eta_0}=M^{-1/2+\phi} \,,
\end{align}
\begin{align}
	|\wt m(z) - \wt m^{(\alpha)}(z)|\prec  M^{-1+1/(b+1)+4\phi} \,,
\end{align}
and
\begin{align}
	|\wt m^{(\gamma)}(z) - \wt m^{(\gamma\alpha)}(z)|\prec M^{-1+1/(b+1)+4\phi} \,,
\end{align}

\end{lemma}
\begin{proof}
Recall $\eta = \eta_0$. In order to prove the first estimate, we consider that the following holds:
\begin{align}
	\begin{split}
		|d^{-1}\wt{m}-d^{-1}\wt{m}^{(\gamma)}|&=\left|\frac{1}{N}\sum_\alpha \frac{1}{\sigma_\alpha^{-1}+m^{(\alpha)}+Z_\alpha} -\frac{1}{N}\sum_\alpha \frac{1}{\sigma_\alpha^{-1}+m^{(\alpha \gamma)}+Z_\alpha^{(a)}}\right| \\
		&=\left| \frac{1}{N}\sum_\alpha\frac{1}{\sigma_\alpha^{-1}+m^{(\alpha)}+Z_\alpha}-\frac{1}{N}\sum_\alpha\frac{1}{\sigma_\alpha^{-1}+m^{(\alpha )}+Z_\alpha+\caO_\prec ((M\eta_0)^{-1})} \right|\,,
	\end{split}
\end{align}
where we have used the definition of $d^{-1}\wt{m}$ and the Cauchy's interlacing property \eqref{cauchy interlacing}.
From proposition~\ref{proposition:step 2_4}, we have
\begin{align}
	\frac{1}{N}\sum_\alpha\frac{1}{\sigma_\alpha^{-1}+m^{(\alpha)}+Z_\alpha}=\frac{1}{N}\sum_\alpha \frac{1}{\sigma_\alpha^{-1}+\whmfc+\caO_\prec((M\eta_0)^{-1})}=O(1) \,,
\end{align}
hence we obtain $|\wt{m}-\wt{m}^{(\gamma)}|\prec (M\eta_0)^{-1}$ .

For $\alpha \neq \gamma$, since $|\sigma_{\alpha}^{-1}+m^{(\alpha)}+Z_\alpha|\ge |\sigma_{\alpha}^{-1}+\mfc|-|m^{(\alpha)}-\mfc| -|Z_\alpha|$, we have on $\Omega$ that
\begin{align}
	|G_{\alpha\alpha}| = \frac{1}{|\sigma_{\alpha}^{-1}+Z_\alpha+m^{(\alpha)}|} \prec \left|1+\frac{1}{\sigma_{\alpha}\mfc}\right|^{-1} \leq C M^\phi\kappa_0^{-1}\,,
\end{align}
where we have used \eqref{farfromhome}.
Recall~\eqref{definition of R2 without hat} and the trivial bound $|G_{\gamma\gamma}|\le \eta_0^{-1}$ to observe that 
\begin{align}
	\frac{1}{N}\sum_{\beta} |G_{\beta\beta}^{(\alpha)}|^2=\frac{|G_{\gamma\gamma}^{(\alpha)}|^2}{N}+\frac{1}{N}\sum_{\beta}^{(\gamma)} |G_{\beta\beta}^{(\alpha)}|^2 \prec M^{2\phi}\,.
\end{align}

From the Schur complement formula, we have
\begin{align}
	G_{\alpha\beta} = - G_{\alpha \alpha}\sum_a x_{\alpha a}G_{a\beta}^{(\alpha)}\quad\text{and}\quad G_{a\beta}^{(\alpha)} = - G_{\beta \beta}^{(\alpha)}\sum_b G_{ab}^{(\alpha\beta)}x_{\beta b}^{(\alpha)}
	\,,\end{align}
hence we find from the concentration estimates in Lemma~\ref{lemma.LDE} and the Ward identity~\eqref{ward} that on $\Omega$,
\begin{align}
	|G_{\beta \beta} - G_{\beta \beta}^{(\alpha)}| = \left| \frac{G_{\alpha \beta} G_{\beta \alpha}}{G_{\alpha \alpha}} \right| \leq C|G_{\alpha\alpha}||G_{\beta\beta}^{(\alpha)}|^2 \frac{\im m^{(\alpha\beta)}}{N \eta} \,.
\end{align}

Thus, we obtain that on $\Omega$,
\begin{align} \begin{split}
		|\wt m(z) -\wt m^{(\alpha)}(z)| &\le \frac{|G_{\alpha\alpha}|}{M} + \frac{1}{M} \sum_{\beta}^{(\alpha)} |G_{\beta\beta} - G_{\beta\beta}^{(\alpha)}|\\ &\leq \frac{|G_{\alpha\alpha}|}{M} + C\frac{M^{2\phi} \kappa_0^{-1}}{M} \sum_{\beta}^{(\alpha)} |G_{\beta\beta}^{(\alpha)}|^2 \frac{\im m^{(\alpha\beta)}}{N \eta}\prec M^{1/(\b+1)} \frac{M^{4\phi}}{M} 
\end{split} \end{align}

In order to show that the inequalities hold uniformly on $z\in\caD_{\phi}'$, we apply the lattice argument as in the proof of Lemma~\ref{lemma:step 2_1}.
\end{proof}

\subsection{Estimates on $N^{-1}\sum Z_a$ and $N^{-1}\sum Z_{\alpha}$}\label{aux estimate 2}
Recall that $n_0>10$ is an integer independent of $M$. In the following lemmas, we control the fluctuation averages $\frac{1}{N} \sum_{a=1}^N Z_a$ , $\frac{1}{N} \sum_{\alpha=n_0}^{M} Z_\alpha $ and other weighted average sums. 

\begin{lemma} \label{lemma:FAZA}
For all $z=E+\ii \eta\in \caD_{\phi}'$, the follwing bound holds on $\Omega$:
\begin{align}
	\left| \frac{1}{N}\sum_{a} Z_a \right| \prec \left(\frac{1}{M\eta_0}\right)^2 \,.
\end{align}
\end{lemma}
\begin{lemma} \label{lemma:step 4}
For all  $z\in\caD_\phi'$, the following bounds hold on $\Omega$:
\begin{align}\label{statement Zlemma 1}
	\left| \frac{1}{N} \sum_{\alpha=n_0}^{M} Z_\alpha(z) \right| \prec M^{-1/2-\fb/2+2\phi}\,,
\end{align}
and, for $\gamma\in\llbracket 1,n_0-1\rrbracket$,
\begin{align}\label{statement Zlemma 2}
	\left| \frac{1}{N} \sum_{\substack{\alpha=n_0\\ \alpha\not=\gamma}}^{M} Z_\alpha^{(\gamma)}(z) \right| \prec M^{-1/2-\fb/2+2\phi}\,.
\end{align}
\end{lemma}

\begin{corollary} \label{corollary:step 4}
For all $z\in\caD_\phi'$, the following bounds hold on $\Omega$:
\begin{align}
	\left| \frac{1}{N} \sum_{\alpha=n_0}^{M} \frac{\whmfc(z)^2 Z_\alpha(z)}{(\sigma_\alpha^{-1}+ \whmfc(z))^2} \right|  \prec M^{-1/2-\fb/2+2\phi}\,,
\end{align}
and, for $\gamma\in\llbracket 1,n_0-1\rrbracket$,
\begin{align}
	\left| \frac{1}{N} \sum_{\substack{\alpha=n_0\\ \alpha\not=\gamma}}^{M} \frac{ \whmfc(z)^2 Z_\alpha^{(\gamma)}(z)}{(\sigma_\alpha^{-1} + \whmfc(z))^2} \right|\prec  M^{-1/2-\fb/2+2\phi} \,.
\end{align}

\end{corollary}
Lemma~\ref{lemma:FAZA}, Lemma~\ref{lemma:step 4} and Corollary~\ref{corollary:step 4} are proved in Section~\ref{sec:Zlemma} 

\begin{remark}
The bounds we obtained in Lemma~\ref{lemma:step 3}, Lemma~\ref{lemma:FAZA}, Lemma~\ref{lemma:step 4}, and Corollary~\ref{corollary:step 4} are~$o(\eta)$. This will be used on several occasions in the next section.
\end{remark}

\subsection{Proof of Proposition~\ref{proposition:lambda_k}}\label{proofofprop}
Recall the definition of $(\wh z_\gamma)$ in~\eqref{definition of hatzk}. We first estimate $\im m(z)$ for $z = E + \ii \eta_0$ satisfying $|z - \wh z_\gamma| \ge M^{-1/2 + 3\phi}$, for all $\gamma \in \llbracket 1, n_0 -1 \rrbracket$.

\begin{lemma} \label{lemma:step 5}
There exists a constant $C > 1$ such that the following bound holds with high probability on $\Omega$: For any $z = E + \ii \eta_0 \in \caD_{\phi}'$, satisfying $|z - \wh z_\gamma| \ge M^{-1/2 + 3\phi}$ for all $\gamma \in \llbracket 1, n_0 -1 \rrbracket$, we have
\begin{align}
	C^{-1} \eta_0 \le \im m(z) \le C \eta_0\,.
\end{align}
\end{lemma}

This implies that the order of the imaginary part of $m(z)$ is $\eta$ when $z$ is sufficiently far from $\wh{z}_\gamma$. 

\begin{proof}
Let $z\in\caD_{\phi}'$ with $\eta=\eta_0$ and choose $\gamma \in \llbracket 1, n_0 -1 \rrbracket$ such that~\eqref{assumption near sigma_k} is satisfied. Consider
\begin{align} \label{eq:step 5_1}
	d^{-1}\wt{m} = \frac{G_{\gamma\gamma}}{N} + \frac{1}{N} \sum_{\alpha}^{(\gamma)} \frac{-1}{\sigma_\alpha^{-1}+m^{(\alpha)}+Z_{\alpha}}\,.
\end{align}
From the assumption in~\eqref{assumption near sigma_k}, Corollary~\ref{corollary:step 2_1}, and Proposition~\ref{proposition:step 2_4}, we find that with high probability on~$\Omega$,
\begin{align} \begin{split} \label{eq:step 5_2}
		&\left| \frac{1}{N} \sum_{\alpha}^{(\gamma)} \left( \frac{-1}{\sigma_\alpha^{-1}+m^{(\alpha)}+Z_{\alpha}} + \frac{1}{\sigma_\alpha^{-1}+\whmfc} - \frac{m^{(\alpha)} - \whmfc + Z_\alpha}{(\sigma_\alpha^{-1}+\whmfc)^2} \right) \right|\\
		=&\left| \frac{1}{N} \sum_{\alpha}^{(\gamma)} \left(\frac{m^{(\alpha)}-\whmfc+Z_\alpha}{\sigma_\alpha^{-1}+\whmfc}\right)\left( \frac{m^{(\alpha)}-\whmfc+Z_\alpha}{(\sigma_\alpha^{-1}+m^{(\alpha)}+Z_{\alpha})(\sigma_\alpha^{-1}+\whmfc)} \right)  \right|\\
		\prec&\frac{1}{N} \sum_{\alpha}^{(\gamma)} \frac{M^{-1 + 2\phi}}{|\sigma_\alpha^{-1}+\whmfc|^3}\le \frac{C}{N} \sum_{\alpha}^{(\gamma)} \frac{|\whmfc|^3 M^{-1 + 2\phi}}{|\sigma_\alpha^{-1}+\whmfc|^3} \\
		\le& C \frac{M^{2 \phi}}{M} M^{\phi} M^{1/(\b+1)} \frac{1}{N} \sum_{\alpha}^{(\gamma)} \frac{|\whmfc|^2}{|\sigma_\alpha^{-1}+\whmfc|^2} \ll \eta\,.
\end{split} \end{align}
We also observe that 
\begin{align}
	\left| \frac{1}{N} \sum_{ \substack{\alpha=1\\\alpha \neq \gamma}}^{n_0} \frac{|\whmfc^2| Z_\alpha}{(\sigma_\alpha^{-1}+\whmfc)^2} \right| \prec  N^{-1} M^{-1/2 + 2\phi} M^{1/(\b+1)} \ll M^{-1} \ll \eta
\end{align}
on $\Omega$.
Thus, from Lemma~\ref{lemma:step 3} and Corollary~\ref{corollary:step 4}, we find with high probability on $\Omega$ that
\begin{align} \label{eq:step 5_3}
	\begin{split}
		\frac{1}{N}\sum_{\alpha}^{(\gamma)}\frac{m^{(\alpha)}-\whmfc+Z_\alpha}{(\sigma_\alpha^{-1}+\whmfc)^2}&=\frac{1}{N}\sum_{\alpha}^{(\gamma)}\frac{m^{(\alpha)}-\whmfc}{(\sigma_\alpha^{-1}+\whmfc)^2}+\frac{1}{\whmfc^2}\frac{1}{N}\sum_{\alpha}^{(\gamma)}\frac{(\whmfc)^2Z_\alpha}{(\sigma_\alpha^{-1}+\whmfc)^2}\\
		&=\frac{1}{N}\sum_{\alpha}^{(\gamma)}\frac{m^{(\alpha)}-\whmfc}{(\sigma_\alpha^{-1}+\whmfc)^2}+o(\eta)\,.
	\end{split}
\end{align}
Recalling~\eqref{eq:step 1_1}, i.e.,
\begin{align}
	\left|1 + \re\frac{1}{\sigma_\gamma\whmfc(z)}\right| \gg M^{-1/2 + 2\phi}\,,
\end{align}
we get $|G_{\gamma\gamma}| \le M^{1/2 - 2\phi}$. We thus obtain from~\eqref{eq:step 5_1},~\eqref{eq:step 5_2}, and~\eqref{eq:step 5_3} that with high probability on~$\Omega$,
\begin{align}
	d^{-1}\wt{m} = o(\eta)+\frac{1}{N}\sum_{\alpha}^{(\gamma)}\left( \frac{-1}{\sigma_\alpha^{-1}+\whmfc}+\frac{m^{(\alpha)}-\whmfc}{(\sigma_\alpha^{-1}+\whmfc)^2}\right)\,.
\end{align}
By Taylor expansion,
\begin{align} \begin{split}
		\frac{1}{N}\sum_{\alpha}^{(\gamma)}\frac{-1}{\sigma_\alpha^{-1}+m}&=\frac{1}{N}\sum_{\alpha}^{(\gamma)}\left(\frac{-1}{\sigma_\alpha^{-1}+\whmfc}+\frac{m-\whmfc}{(\sigma_\alpha^{-1}+\whmfc)^2}\right)+O\left(\frac{1}{N}\sum_{\alpha}^{(\gamma)}\frac{(m-\whmfc)^2}{(\sigma_\alpha^{-1}+\whmfc)^3}\right)\\&=\frac{1}{N}\sum_{\alpha}^{(\gamma)}\left(\frac{-1}{\sigma_\alpha^{-1}+\whmfc}+\frac{m-\whmfc}{(\sigma_\alpha^{-1}+\whmfc)^2}\right)+o(\eta)=d^{-1}\wt{m}+o(\eta) \,.
\end{split}\end{align}
Taking imaginary parts, we get
\begin{align}
	\frac{1}{N}\sum_{\alpha}^{(\gamma)}\frac{\im{m}}{|\sigma_\alpha^{-1}+m|^2}=\frac{\im{m}}{|m|^2}\frac{1}{N}\sum_{\alpha}^{(\gamma)}\frac{|m|^2}{|\sigma_\alpha^{-1}+m|^2}=\im{d^{-1}\wt{m}}+o(\eta)\,.
\end{align}
If we take \begin{align}K_m^{(\gamma)}=\frac{1}{N}\sum_{\alpha}^{(\gamma)}\frac{|m|^2}{|\sigma_\alpha^{-1}+m|^2},\end{align}
since by~\eqref{mfcwhmfc} and~\eqref{proposition:step 2_4},
\begin{align}
	\frac{1}{N} \sum_{\alpha}^{(\gamma)} \frac{|m|^2}{|\sigma_\alpha^{-1}+ m|^2} = \frac{1}{N} \sum_{\alpha}^{(\gamma)} \frac{|\whmfc|^2}{|\sigma_\alpha^{-1}+ \whmfc|^2} + o(1) < c < 1\,,
\end{align}
for some constant $c$, then we have
\begin{align}\label{imtildem}
	\left( -\im{\frac{1}{m}}\right)\cdot K_m^{(\gamma)}=\im{d^{-1}\wt{m}}+o(\eta).
\end{align}
Recall \eqref{dinvwhmfc}, we have that
\begin{align}
	G_{aa}=\frac{1}{-z-d^{-1}\wt{m}^{(a)}-Z_a}=\frac{1}{-z-d^{-1}\wt{m}+\caO_\prec(M^{-1/2+\phi})}=\frac{1}{\frac{1}{m}+\caO_\prec(M^{-1/2+\phi})} \,,
\end{align}
which implies
\begin{equation}
	\frac{1}{G_{aa}}=\frac{1}{m}+\caO_\prec(M^{-1/2+\phi}) \,.
\end{equation}
By using~\eqref{proposition:step 2_4}, $\whmfc\sim 1$ so that $m\sim 1$. In addition, $G_{aa}\sim 1$ and
\begin{align}\label{entrywise}
	G_{aa}=m+\caO_\prec({M^{-1/2+\phi}}) \,.
\end{align}
Considering that 
\begin{align} \begin{split}
		\left| \frac{1}{N}\sum_{a} \left( \frac{1}{G_{aa}}-\frac{1}{m}\right) \right| & =\left| \frac{1}{N}\sum_{a} \left( \frac{m-G_{aa}}{m^2}\right)+\frac{1}{N}\sum_{a} \left( \frac{(G_{aa}-m)^2}{m^2 G_{aa}}\right)\right| \\&
		=\left| \frac{1}{N}\sum_{a} \left( \frac{(G_{aa}-m)^2}{m^2 G_{aa}}\right) \right| \leq \left| \frac{C}{N}\sum_{a}  (G_{aa}-m)^2 \right|\\& \leq \left| \frac{C}{N}\sum_{a}  \caO_\prec(M^{-1+2\phi}) \right|\prec M^{-1+2\phi} \ll \eta \,,
	\end{split}
\end{align}
thus we have 
\begin{align}
	-\frac{1}{N}\sum_{a}\frac{1}{G_{aa}}=-\frac{1}{m}+o(\eta).
\end{align}
By the definition of $G_{aa}$, Lemma~\ref{lemma:step 3} and Lemma~\ref{lemma:FAZA}, the left hand side of the equation can be written as
\begin{align}
	\begin{split}
		\frac{1}{N}\sum_{a} (z+d^{-1}\wt{m}^{(a)}+Z_a)&=\frac{1}{N}\sum_{a} (z+d^{-1}\wt{m}-d^{-1}\wt{m}+d^{-1}\wt{m}^{(a)}+Z_a)\\
		&=\frac{1}{N}\sum_{a}^{(\gamma)} (z+d^{-1}\wt{m}-d^{-1}\wt{m}+d^{-1}\wt{m}^{(a)}+Z_a)\\&+\frac{1}{N}(z+d^{-1}\wt{m}-d^{-1}\wt{m}+d^{-1}\wt{m}^{(\gamma)}+Z_\gamma)\\
		&=z+d^{-1}\wt{m}+o(\eta)+\frac{1}{N}\sum_{a}Z_{a}=z+d^{-1}\wt{m}+o(\eta).
	\end{split}
\end{align}
Hence,
\begin{equation} \label{-im1/m and eta}
	-\im\frac{1}{m}=\eta+\im{d^{-1}\wt{m}}+o(\eta) \,.
\end{equation}
Applying~\eqref{imtildem},
\begin{align}
	-\im{\frac{1}{m}}=\eta+\left( -\im{\frac{1}{m}}\right)\cdot K_m^{(\gamma)}+o(\eta), 
\end{align}

\begin{equation*}
	(1-K_m^{(\gamma)})\left( -\im{\frac{1}{m}} \right)=\eta+o(\eta).
\end{equation*}
Therefore we can conclude that $C^{-1} \eta \le -\im{\frac{1}{m}} \le C \eta$ with high probability for some $C > 1$. This proves the desired lemma.
\end{proof}

As a next step, we show that $\im m^{(\gamma)}(z) \sim \eta$ even though when $z$ is close to $\wh{z}_{\gamma}$. Furthermore, we find a point $\wt{z}_\gamma$ close to $\wh{z}_\gamma$ such that the imaginary part of $m(\wt{z}_\gamma)$ is much larger than $\eta$. 

\begin{lemma} \label{lemma:step 6_1}
There exists a constant $C > 1$ such that the following bound holds with high probability on $\Omega$, for all  $z = E + \ii \eta_0 \in \caD_{\phi}'$: For given $z$, choose $\gamma \in \llbracket 1, n_0 -1 \rrbracket$ such that~\eqref{assumption near sigma_k} is satisfied. Then, we have
\begin{align}
	C^{-1} \eta_0 \le \im m^{(\gamma)}(z) \le C \eta_0\,.
\end{align}
\end{lemma}

\begin{proof}
Reasoning as in the proof of Lemma~\ref{lemma:step 5}, we find from Proposition~\ref{proposition:step 2_4}, Corollary~\ref{corollary:step 2_1}, Lemma~\ref{lemma:step 3}, and Corollary~\ref{corollary:step 4} that, with high probability on $\Omega$,
\begin{align}
	d^{-1}\wt m^{(\gamma)}=\frac{1}{N} \sum_\alpha^{(\gamma)} \left(\frac{-1}{\sigma_\alpha^{-1}+\whmfc}+\frac{m^{(\alpha\gamma)}-\whmfc}{(\sigma_\alpha^{-1}+\whmfc)^2} \right) + o(\eta_0) = \frac{1}{N} \sum_\alpha^{(\gamma)} \frac{-1}{\sigma_\alpha^{-1}+m^{(\gamma)}} + o(\eta_0)\,.
\end{align}
Considering the imaginary part, we can prove the desired lemma as in the proof of Lemma~\ref{lemma:step 5}.
\end{proof}

\begin{corollary} \label{corollary:step 6_1}
The following bound holds on $\Omega$, for all $z = E + \ii \eta_0 \in \caD_{\phi}'$:
For given $z$, choose $\gamma \in \llbracket 1, n_0 -1 \rrbracket$ such that~\eqref{assumption near sigma_k} is satisfied. Then, we have
\begin{align}
	|Z_\gamma| \prec \frac{1}{\sqrt M}\,.
\end{align}
\end{corollary}

Now we are able to locate the points for which $\im m(z) \gg \eta_0$ near the edge.

\begin{lemma} \label{lemma:step 6_2}
For any $\gamma \in \llbracket 1, n_0-1 \rrbracket$, there exists $\wt E_\gamma\in\R$ such that the following holds with high probability on $\Omega$: If we let $\wt z_\gamma\deq\wt E_\gamma + \ii \eta_0$, then $|\wt z_\gamma - \wh z_\gamma| \le M^{-1/2 + 3\phi}$ and $\im m(\wt z_\gamma) \gg \eta_0$.
\end{lemma}

\begin{proof}
Note that the condition $|z - \wh z_\gamma| \ge M^{-1/2 + 3\phi}$ has not been used in the derivation of~\eqref{eq:step 5_2} and~\eqref{eq:step 5_3}, so although $|z - \wh z_\gamma| \le M^{-1/2 + 3\phi}$, we still attain that
\begin{align} \label{eq:step 6_2}
	d^{-1}\wt m = \frac{G_{\gamma\gamma}}{N} + \frac{1}{N} \sum_\alpha^{(\gamma)} \frac{-1}{\sigma_\alpha^{-1} + m^{(\alpha)} +Z_\alpha} =\frac{G_{\gamma\gamma}}{N} + \frac{1}{N} \sum_\alpha^{(\gamma)} \frac{-1}{\sigma_\alpha^{-1} + m} + o(\eta_0)
\end{align}
with high probability on $\Omega$. Consider
\begin{align}
	-\frac{1}{G_{\gamma\gamma}} = \sigma_{\gamma}^{-1}+m^{(\gamma)}+Z_\gamma.
\end{align}
Setting $z_\gamma^+ \deq \wh z_\gamma + N^{-1/2 + 3\phi}$, Lemma~\ref{mfc estimate} shows that
\begin{align}
	\re \frac{1}{\mfc(z_\gamma^+)} - \re \frac{1}{\mfc(\wh z_\gamma)} \le - C M^{-1/2 + 3\phi}\,,
\end{align}
on $\Omega$. Thus, from Lemma~\ref{hat bound} and the definition of $\wh z_k$, we find that
\begin{align}
	\re \frac{1}{\whmfc(z_\gamma^+)} + \sigma_\gamma \le - C M^{-1/2 + 3\phi}\,,
\end{align}
on $\Omega$. Similarly, if we let $z_\gamma^- \deq \wh z_\gamma - M^{-1/2 + 3\phi}$, we have that
\begin{align}
	\re \frac{1}{\whmfc(z_\gamma^-)} + \sigma_\gamma \ge C M^{-1/2 + 3\phi}\,,
\end{align}
on $\Omega$. Since
\begin{align}
	-\frac{1}{G_{\gamma\gamma}}=\frac{\whmfc}{\sigma_\gamma} \left(\sigma_\gamma+ \frac{1}{\whmfc}+o(M^{-1/2+3\phi})\right)\,,
\end{align}
with high probability on $\Omega$, we find that there exists $\wt z_\gamma = \wt E_\gamma + \ii \eta_0$ with $\wt E_\gamma \in (\wh E_\gamma - M^{-1/2 + 3\phi}, \wh E_\gamma + M^{-1/2 + 3\phi})$ such that $\re G_{\gamma\gamma}(\wt z_\gamma) = 0$. When $z = \wt z_\gamma$, we have from Lemma~\ref{lemma:step 6_1} and Corollary~\ref{corollary:step 6_1} that on $\Omega$, 
\begin{align}
	|\im G_{\gamma\gamma}(\wt z_\gamma)| = \frac{1}{|\im m^{(\gamma)}(\wt z_\gamma) + \im Z_\gamma(\wt z_\gamma)|} \geq M^{1/2-\phi/2}\,,\qquad\quad \re G_{\gamma\gamma}(\wt z_\gamma)=0\,.
\end{align}
From~\eqref{eq:step 6_2}, we obtain that
\begin{align} \label{imm at wt z_k}
	d^{-1} \im \wt m(\wt z_\gamma) = \frac{ {\im G_{\gamma\gamma}(\wt z_\gamma)} }{N} + \frac{1}{N} \sum_\alpha^{(\gamma)} \frac{\im m(\wt z_\gamma)}{|\sigma_\alpha^{-1} + m(\wt z_\gamma)|^2} + o(\eta_0)\,.
\end{align}
Combining with~\eqref{-im1/m and eta}, 
\begin{align}(1- K_m^{(\gamma)})\im\left\{-\frac{1}{m(\wt z_\gamma)}\right\}=\eta+\frac{ {\im G_{\gamma\gamma}(\wt z_\gamma)} }{N} + o(\eta_0). 
\end{align}
Since
$K_m^{(\gamma)}<c<1$ for some constant $c$,
with high probability on $\Omega$,  we get from~\eqref{imm at wt z_k} that
\begin{align}
	-\im \frac{1}{m(\wt z_\gamma)} \ge M^{-\phi/2}M^{-1/2} \gg \eta_0 \,, 
\end{align} 
with high probability on $\Omega$, which was to be proved.
\end{proof}

We now turn to the proof of Proposition~\ref{proposition:lambda_k}. Recall that we denote by $\lambda_\gamma$ the $\gamma$-th largest eigenvalue of $\caQ$, $\gamma\in\llbracket 1,n_0-1\rrbracket$.  Also recall that $\kappa_0=M^{-1/(\b+1)}$; see~\eqref{definition of kappa0}.
\begin{proof}[Proof of Proposition~\ref{proposition:lambda_k}]
First, we consider the case $\gamma=1$. From the spectral decomposition of $Q$, we have
\begin{align}
	\im m(E + \ii \eta_0) = \frac{1}{N} \sum_{i=1}^N \frac{\eta_0}{(\lambda_i - E)^2 + \eta_0^2}\,,
\end{align}
and $\im m(\lambda_1 + \ii \eta_0) \ge (M \eta_0)^{-1} \gg \eta_0$. Recall the definition of $\wh z_1=\wh E_1+\ii\eta_0$ in~\eqref{definition of hatzk}. Since, with high probability on $\Omega$, $\im m(z) \sim \eta_0$ for $z \in \caD_{\phi}'$ satisfying $|z - \wh z_1| \ge M^{-1/2 + 3\phi}$, as we proved in Lemma~\ref{lemma:step 5}, we obtain that $\lambda_1 < \wh E_1 + M^{-1/2 + 3\phi}$.

Recall the definitions for $\wh z_1$ and $z_1^-$ in the proof of Lemma~\ref{lemma:step 6_2}. Assume $\lambda_1 < \wh E_1 - M^{-1/2 + 3\phi}$, then $\im m(E + \ii \eta_0)$ is a decreasing function of $E$ on the interval $(\wh E_1 - M^{-1/2 + 3\phi}, \wh E_1 + M^{-1/2 + 3\phi})$. However, we already have shown in Lemma~\ref{lemma:step 5} and Lemma~\ref{lemma:step 6_2} that with high probability, $\im m(\wt z_1) \gg \eta_0$, $\im m(z_1^-) \sim \eta_0$, and $\re \wt z_1 > \re z_1^-$. It contradicts to previous assumption, so $\lambda_1 \ge \wh E_1 - M^{-1/2 + 3\phi}$. Now Lemma~\ref{mfc estimate} and Lemma~\ref{hat bound}, together with Lemma~\ref{lemma:step 1} conclude that
\begin{align}
	\frac{1}{\whmfc (\lambda_1 + \ii \eta_0)} = \frac{1}{\whmfc (\wh z_1)} + O (M^{-1/2 + 3\phi}) = -\sigma_1 + O (M^{-1/2 + 3\phi})\,,
\end{align}
which proves the proposition for the special choice $\gamma=1$.

Next, we consider the case $\gamma=2$; with induction, the other cases can be shown by similar manner. Consider~$H^{(1)}$, the minor of $H$ obtained by removing the first row and column and denote the largest eigenvalue of $H^{(1)}$ by $\lambda_1^{(1)}$. The Cauchy's interlacing property implies $\lambda_2 \le \lambda_1^{(1)}$. In order to estimate $\lambda_1^{(1)}$, we follow the first part of the proof which yields
\begin{align}
	\wh E_2 - M^{-1/2 + 3\phi} \le \lambda_1^{(1)} \le \wh E_2 + M^{-1/2 + 3\phi}\,,
\end{align}
where we let $\wh z_2 = \wh E_2 + \ii \eta_0$ be a solution to the equation
\begin{align}
	\sigma_2+\re \frac{1}{\whmfc(\wh z_2)}=0.
\end{align}
This shows that
\begin{align}
	\lambda_2 \le \wh E_2 + M^{-1/2 + 3\phi}\,.
\end{align}

To prove the lower bound, we may argue as in the first part of the proof. Recall that we have proved in Lemma~\ref{lemma:step 5} and Lemma~\ref{lemma:step 6_2} that with high probability on $\Omega$,
\begin{enumerate}
	\item[(1)] For $z = \wh z_2 - M^{-1/2 + 3\phi}$, we have $\im m(z) \le C \eta_0 \,.$
	\item[(2)] There exists $\wt z_2 = \wt E_2 + \ii \eta_0$, satisfying $|\wt z_2 - \wh z_2| \le M^{-1/2 + 3\phi}$, such that $\im m(\wt z_2) \gg \eta_0$.
\end{enumerate}
If $\lambda_2 < \wh E_2 - M^{-1/2 + 3\phi}$, then 
\begin{align}
	\im m(E + \ii \eta_0) - \frac{1}{N} \frac{\eta_0}{(\lambda_1 - E)^2 + \eta_0^2} = \frac{1}{N} \sum_{i=2}^N \frac{\eta_0}{(\lambda_i - E)^2 + \eta_0^2}
\end{align}
is a decreasing function of $E$. Since we know that with high probability on $\Omega$,
\begin{align}
	\frac{1}{N} \frac{\eta_0}{(\lambda_1 - \wh E_2)^2 + \eta_0^2} \le \frac{1}{N}\frac{C \eta_0}{  M^{-2\phi} \kappa_0^2} \ll \eta_0\,,
\end{align}
we have $\im m(\wt z_2) \le C \eta_0$, which contradicts to the definition of $\wt z_2$. Thus, we find that $\lambda_2 \ge \wh E_2 - M^{-1/2 + 3\phi}$ with high probability on $\Omega$.

We now proceed as above to conclude that, with high probability on $\Omega$,
\begin{align}
	\frac{1}{\whmfc (\lambda_2 + \ii \eta_0)} = \frac{1}{\whmfc (\wh z_2)} + O (M^{-1/2 + 3\phi}) = -\sigma_2 + O (M^{-1/2 + 3\phi})\,,
\end{align}
which proves the proposition for $\gamma=2$. The general case is proven in the same way.
\end{proof}

\section{Fluctuation averaging lemma}\label{sec:Zlemma}
In this section we prove Lemma~\ref{lemma:FAZA}, Lemma~\ref{lemma:step 4} and Corollary~\ref{corollary:step 4}. Recall that we denote by $\E_i$ the partial expectation with respect to the $i$-th column/row of $X$. Set $Q_i\deq\lone-\E_i$. 

We are interested in bounding the fluctuation averages
\begin{align}\label{donkey}
\frac{1}{N}\sum_{a=1}^NZ_a(z),\quad \frac{1}{N}\sum_{\alpha=n_0}^{M} Z_\alpha(z) \,,
\end{align}
where $n_0$ is a $M$-independent fixed integer. By Schur's complement formula,
\begin{align}
\frac{1}{N}\sum_{a=1}^{N}Q_a\left(\frac{1}{G_{aa}}\right)&=\frac{1}{N}\sum_{a=1}^N Q_a\left( -z-\sum_{\alpha,\beta}x_{\alpha a}G_{\alpha\beta}^{(a)}x_{\beta a}\right)\nonumber\\
&=-\frac{1}{N}\sum_{a=1}^N Z_a\,,\label{donkey a}
\end{align}
and
\begin{align}
\frac{1}{N}\sum_{\alpha=n_0}^{M}Q_\alpha\left(\frac{1}{G_{\alpha\alpha}}\right)&=\frac{1}{N}\sum_{\alpha=n_0}^{M} Q_\alpha\left( -\sigma_\alpha^{-1}-\sum_{a,b}x_{\alpha a}G_{ab}^{(\alpha)}x_{\alpha b}\right)\nonumber\\
&=-\frac{1}{N}\sum_{\alpha=n_0}^{M} Z_\alpha\,,\label{donkey alpha}
\end{align}
where we have used the concentration estimate~\eqref{LDE1}. The first main result of this section asserts that
\begin{align}
\left|\frac{1}{N}\sum_{a=1}^{N}Q_a\left(\frac{1}{G_{aa}}\right)\right|\prec  (M\eta_0)^{-2}\,,
\end{align}
and the second one implies that
\begin{align}
\left|\frac{1}{N}\sum_{\alpha=n_0}^{M}Q_\alpha\left(\frac{1}{G_{\alpha\alpha}}\right)\right|\prec M^{-1/2-\fb/2+2\phi}\,, 
\end{align}
with $z$ satisfying $|1+\re\frac{1}{\sigma_\alpha m_{fc}(z)}|\ge\frac{1}{2} M^{-1/(b+1)+\phi}$, for all $\alpha\ge n_0$.

Fluctuation average lemma or abstract decoupling lemma was used in~\cite{sce,p}. For sample covariance matrix model with general population, the lemma was used in~\cite{zb} to obtain stronger local law from a weaker one. In these works, the LSD show square-root behavior at the edge. On the other hand, due to the lack of such behavior in our model, we need different approach to prove the lemmas, which was considered in~\cite{eejl}. When the square root behavior appears, it was proved that there exists a deterministic control parameter $\Lambda_o(z)$ such that $\Lambda_o \ll 1$ with $\im z \gg M^{-1}$ and $\Lambda_o$ bounds the off-diagonal entries of the Green function and $Z_a$'s. Moreover, the diagonal entries of the Green function is bounded below.

In our circumstance, under the assumption of Lemma~\ref{lemma:step 4}, the Green function entries with the Greek indices, $(G_{\alpha\beta}(z))$, can become large, i.e., $|G_{\alpha\beta}(z)|\gg 1$ when $\im \eta \sim M^{-1/2}$, for certain choices of the spectral parameter $z$ (close to the spectral edge) and certain choice of indices $\alpha,\beta$.
However, resolvent fractions of the form $G_{\alpha\beta}(z)/G_{\beta\beta}(z)$ and $G_{\alpha\beta}(z)/G_{\alpha\alpha}(z)G_{\beta\beta}(z)$ ($\alpha,\beta\ge n_0$) are small (see Lemma~\ref{lemma: 7.1} below for a precise statement). Using this observation, we adapt the methods of~\cite{eejl} to control the fluctuation average~\eqref{donkey}.

On the other hand, the Green function entries, $(G_{ab})$, are in a different situation. Roughly speaking, Once we have the local law, $G_{aa}$ are close to $m$ which is close to $\whmfc$ so that it is bounded below and above. By this property, we can find a control parameter, $\Lambda_o$, which satisfies $|G_{ab}|\ll \Lambda_o \ll 1$ for $\im{z}\gg M^{-1}$. This is the reason why the orders of the right hand side of Lemma~\ref{lemma:FAZA} and Lemma~\ref{lemma:step 4} are different. Thus we do not have such difficulty from the formal case and we can apply the method from~\cite{p}.

\subsection{Preliminaries}
In this subsection, we introduce some notion from~\cite{eejl} which are useful to estimate the fraction of green function entries.\\
Let $a,b\in\llbracket 1,M\rrbracket$ and $\T,\T'\subset\llbracket 1,M\rrbracket$, with $\alpha,\beta \not\in\T$, $\beta \not\in \T'$, $\alpha\not=\beta$,  then we set
\begin{align}
F_{\alpha\beta}^{(\T,\T')}(z)\deq\frac{G_{\alpha \beta}^{(\T)}(z)}{G_{\beta\beta}^{(\T')}(z)}\,,\qquad \quad(z\in\C^+)\,,
\end{align}
and we often abbreviate $F_{\alpha\beta}^{(\T,\T')}\equiv F_{\alpha\beta}^{(\T,\T')}(z)$. In case $\T=\T'=\emptyset$, we simply write $F_{\alpha\beta}\equiv F_{\alpha\beta}^{(\T,\T')}$. Below we will always implicitly assume that $\{\alpha,\beta\}$ and $\T,\T'$ are compatible in the sense that $\alpha\not=\beta$, $\alpha,\beta\not\in\T$, $\beta\not\in\T'$.

Starting from~\eqref{basic resolvent}, simple algebra yields the following relations among the $\lbrace F_{\alpha\beta}^{(\T,\T')}\rbrace$.
\begin{lemma}\label{lemma: 7.1}
Let $a,b,c\in\llbracket1,M\rrbracket$, all distinct, and let $\T,\T'\subset\llbracket 1,M\rrbracket$. Then,
\begin{itemize}
	\item[(1)] for $\gamma\not\in\T\cup \T'$,
	\begin{align}\label{Zlemma expand 1}
		F_{\alpha\beta}^{(\T,\T')}=F_{\alpha\beta}^{(\T \gamma,\T')}+F_{\alpha\gamma}^{(\T,\T')}F_{\gamma\beta}^{(\T,\T')}\,;
	\end{align}
	\item[(2)] for $\gamma\not\in\T\cup \T'$,
	\begin{align}\label{Zlemma expand 2}
		F_{\alpha\beta}^{(\T,\T')}=F_{\alpha\beta}^{(\T,\T'\gamma)}- F_{\alpha\beta}^{(\T,\T'\gamma)}F_{\beta\gamma}^{(\T,\T')}F_{\gamma\beta}^{(\T,\T')}\,;
	\end{align}
	\item[(3)] for $\gamma\not\in\T$,
	\begin{align}\label{Zlemma expand basic}
		\frac{1}{G_{\alpha\alpha}^{(\T)}}=\frac{1}{G_{\alpha\alpha}^{(\T \gamma)}}\left(1-F_{\alpha\gamma}^{(\T,\T)}F_{\gamma \alpha}^{(\T,\T)}\right)\,.
	\end{align}
	
\end{itemize}

\end{lemma}

\subsection{The fluctuation averaging lemma for $Z_a$}
From section~\ref{location}, we have local law, $|m-\whmfc|\prec M^{-1/2+2\phi}$, which induces that $m\sim 1$ so that $G_{aa}\sim 1$ and $G_{aa}-G_{bb}=o(1)$. It is quite interesting that once we have local law, $G_{aa}$ are asymptotically identical and bounded below and above. This is because of the structure of $G_{aa}$. When the local law holds, the summation part of its denominiator is well averaged so that the estimates above are staisfied. This property leads us to prove the ``fluctuation average lemma" or ``abstract decoupling lemma" via mehod from~\cite{p} . Therefore, it is sufficient to prove essential bounds from~\cite{sce} or~\cite{p} to prove Lemma~\ref{lemma:FAZA}.
\begin{lemma}\label{gabound}
For any $z=E+\ii\eta\in \mathcal{D}_\phi'$ and $a,b\in \llbracket 1,N \rrbracket$, we have $|G_{aa}-G_{bb}|=o(1)$ and $ |m-G_{aa}|=o(1)$ so that $G_{aa}\sim 1$ with high probability on $\Omega$. Furthermore, for any $a\in \llbracket 1, N \rrbracket$, we have $|m-G_{aa}|\prec (M\eta_0)^{-1}$.
\end{lemma}
\begin{proof}
The proof of this lemma is contained in the proof of Lemma~\ref{lemma:step 5}. (See \eqref{entrywise} in the sequel.)
\end{proof}
Now we prove the boundedness of off diagonal entries of $G$.
\begin{lemma}
For $z\in \mathcal{D}_\phi'$ and $a,b \in \llbracket 1,N\rrbracket$, we have 
\begin{align}
	|G_{ab}|\prec \frac{1}{M\eta_0}\,.
\end{align}

\end{lemma}
\begin{proof}
By resolvent identities~\eqref{roman} and concentration estimate, Lemma~\ref{lemma.LDE}, we have
\begin{align}\begin{split}
		|G_{ab}|&=\left| G_{bb}\sum_\beta  G_{a\beta}^{(b)}x_{\beta b} \right|
		=\left| G_{aa}G_{bb}\sum_{\alpha,\beta} x_{a\alpha } G_{\alpha \beta}^{(ab)}x_{\beta b} \right|\\
		&\le C\left| \sum_{\alpha,\beta} x_{a\alpha } G_{\alpha \beta}^{(ab)}x_{\beta b} \right|\prec \sqrt{\frac{\im \wt{m}^{(ab)}}{M\eta}}\,.
	\end{split}
\end{align}

Note that by Proposition~\ref{proposition:step 2_4}, we have 
\begin{align}
	G_{\alpha\alpha}=\frac{-1}{\sigma_\alpha^{-1}+m^{(\alpha)}+Z_\alpha}=\frac{-1}{\sigma_\alpha^{-1}+\whmfc+\caO_\prec((M\eta_0)^{-1})} \,.
\end{align}
Hence we have
\begin{align}\label{dinvwhmfc}
	\begin{split}
		z+\frac{1}{m}=z+\frac{1}{\whmfc}+\caO_\prec((M\eta_0)^{-1})&=
		\frac{1}{N}\sum\frac{1}{\sigma_\alpha^{-1}+\whmfc}+\caO_\prec((M\eta_0)^{-1})\\
		&=-d^{-1}\wt{m}+\caO_\prec((M\eta_0)^{-1}) \,.
	\end{split}
\end{align}

Considering 
\begin{align}\begin{split}
		d^{-1} \wt{m}^{(ab)}&=\frac{1}{N}\sum \frac{-1}{\sigma_\alpha^{-1}+m^{(ab)}+Z_\alpha^{(ab)}}=\frac{1}{N}\sum \frac{-1}{\sigma_\alpha^{-1}+\whmfc+\caO_\prec((M\eta_0)^{-1})}\\
		&=\frac{1}{N}\sum \frac{-1}{\sigma_\alpha^{-1}+\whmfc}+\caO_\prec((M\eta_0)^{-1})=-\frac{1}{\whmfc}-z+\caO_\prec((M\eta_0)^{-1})\,,
	\end{split}
\end{align}
we have that
\begin{align}
	\sqrt{\frac{\im\wt{m}^{(ab)}}{M\eta}}\le C\sqrt{\frac{\im \whmfc}{M\eta}}+\caO_\prec((M\eta_0)^{-1})=\caO_\prec((M\eta_0)^{-1})\,,
\end{align}
where we have used~\eqref{dinvwhmfc}, $\whmfc\sim 1$ and Lemma~\ref{weakboundwhmfc}. Hence we have the desired lemma.

\end{proof}

From above lemmas, we have a rough bound for fraction of the green function entries.
\begin{corollary}
For $z\in \mathcal{D}_\phi'$ and $a,b \in \llbracket 1,N\rrbracket$, we have
\begin{align}
	\left|\frac{G_{ab}}{G_{aa}}\right|\prec \frac{1}{M\eta_0} \,.
\end{align}

\end{corollary}
Through those three bounds, we can apply the method from appendix B of~\cite{sce} so that we have the proof of the Lemma~\ref{lemma:FAZA}.

\subsection{The fluctuation averaging lemma for $Z_{\alpha}$}
Proof of the fluctuation average lemma for $Z_{\alpha}$ is more complicate than that of $Z_a$. Eventhough the local law yields the well boundedness of $G_{ab}$'s, $G_{\alpha\beta}$ might be extremely large. 
We use the technique from~\cite{eejl}. Therefore, we only need to check the core estimates which have been used in~\cite{eejl} to prove fluctuation average lemma.
\begin{remark}
Since in~\cite{eejl}, the authors used the $(\xi,\nu)$-high probability concept rather than stochastic dominance, one can also check~\cite{sce} to handle the stochastic dominance version of proof of fluctuation averaging lemma. The both proofs are identical in some degrees. 
\end{remark}

Recall the definition of the domain $\caD_{\phi}'$ of the spectral parameter in~\eqref{a index assumption} and of the constant $\fb>0$ in~\eqref{fb}. Set $A\deq\llbracket n_0,M\rrbracket$. To start with, we bound $F_{\alpha\beta}$ and $F_{\alpha\beta}^{(\emptyset,\alpha)}/G_{\alpha\alpha}$ on the domain~$\caD_{\phi}'$.  
\begin{lemma}\label{firstbound}
Assume that, for all $z\in\caD_{\phi}'$, the estimates
\begin{align}\label{weaklocallaw}
	|m(z)-\whmfc(z)|\prec \frac{1}{M\eta_0}\,,\qquad \im m(z)\prec \frac{1}{M\eta_0}\,,
\end{align}
hold on $\Omega$.

Then for all $z\in\caD_{\phi}'$,
\begin{align}\label{F bound 1}
	\max_{\substack{\alpha,\beta\in A\\ \alpha\not=\beta}}|F_{\alpha\beta}(z)|\prec  	M^{-\fb/2+\phi} \,,\qquad\quad( z\in\caD_{\phi}')\,,
\end{align}
and
\begin{align}\label{F bound 2}
	\max_{\substack{\alpha,\beta\in A\\ \alpha\not=\beta}}\left|\frac{F_{\alpha\beta}^{(\emptyset,\alpha)}(z)}{G_{\alpha\alpha}(z)}\right|\prec \frac{1}{M\eta_0}\,,\qquad\quad( z\in\caD_{\phi}')\,,
\end{align}
on $\Omega$.
\end{lemma}
\begin{proof}
Dropping the $z$-dependence from the notation, we first note that by Schur's complement formula~\eqref{schur} and inequality~\eqref{weaklocallaw}, we have with high probability on $\Omega$, for $z\in\caD_{\phi}'$,
\begin{align}\begin{split}
		\frac{1}{G^{(\beta)}_{\alpha\alpha}}&=-\sigma_\alpha^{-1}-\sum_{a,b}x_{\alpha a}G_{ab}^{(\alpha\beta)}x_{b\alpha}\\
		&=-\sigma_{\alpha}^{-1}+\whmfc-\whmfc+m-m+m^{(\alpha\beta)}-m^{(\alpha\beta)}-\sum_{a,b}x_{\alpha a}G_{ab}^{(\alpha\beta)}x_{b\alpha}\\
		&=-\sigma_{\alpha}^{-1}-\whmfc+\caO_{\prec}((M\eta_0)^{-1})
	\end{split} 
\end{align}
for all $\alpha \in A$, $\beta\in\llbracket 1,M\rrbracket$, $\alpha\not=\beta$. Thus, for $z\in\caD_{\phi}'$, Lemma~\ref{cauchy interlacing} yields
\begin{align}\label{consequence of index assumption}
	|G^{(\beta)}_{\alpha\alpha}|\leq CM^{\phi}\kappa_0^{-1}=M^{1/(b+1)+\phi}\,.
\end{align}
Further, from the resolvent formula~\eqref{greek} we obtain
\begin{align}
	F_{\alpha\beta}=-\sum_{b} G_{\alpha b}^{( \beta)}x_{\alpha b}\,,
\end{align}
for $\alpha,\beta\in A$, $\alpha\not= \beta$.
From the concentration estimate~\eqref{LDE1} and by~\eqref{consequence of index assumption} we infer that
\begin{align}\begin{split}
		\left|\sum_{b} G_{\alpha b}^{( \beta)}x_{\alpha b}\right|&\prec \left(\frac{\sum_{b}|G_{\alpha b}^{(\beta)}|^2}{M}\right)^{1/2}\\&\prec \left| \frac{\im G_{\alpha\alpha}^{(\beta)}}{M\eta} + \frac{1}{M} \right|^{1/2} \prec \left|   M^{-\fb/2+2\phi} +\frac{1}{M} \right|^{1/2} \,,
	\end{split}
\end{align}
with high probability, where we have used Lemma 4.6 of~\cite{ky}. Since $0<\fb<1/2$ so that $M^{-1} \ll M^{-\fb}$, hence we conclude  that
\begin{align}
	|F_{\alpha\beta}|\prec M^{-\fb/2+\phi}\,,
\end{align}
on $\Omega$.

To prove the second claim, we recall that, for $\alpha\not=\beta$, the resolvent formula~\eqref{greek}. Then we get
\begin{align}
	\frac{F_{\alpha\beta}^{(\emptyset, \alpha)}}{G_{\alpha\alpha}}= \frac{G_{\alpha\alpha}G_{\beta\beta}^{(\alpha)}(XG^{(\alpha\beta)}X^*)_{\alpha\beta}}{G_{\alpha\alpha}G_{\beta\beta}^{(\alpha)}}=(XG^{(\alpha\beta)}X^*)_{\alpha\beta}\,,
\end{align}
and the concentration estimates~\eqref{LDE2} and~\eqref{bound on xij} imply that
\begin{align}
	\left|\frac{F_{\alpha\beta}^{(\emptyset, \alpha)}}{G_{\alpha\alpha}}\right|\prec \sqrt{\frac{\im m^{(\alpha\beta)}}{M\eta}}\,,
\end{align}
with high probability. Since $|m-m^{(\alpha\beta)}|\le C(M\eta_0)^{-1}$ on $\caD_{\phi}'$, by Lemma~\ref{cauchy interlacing} and~\eqref{weaklocallaw} we have
\begin{align}
	\left|\frac{F_{\alpha\beta}^{(\emptyset,\alpha)}}{G_{\alpha\alpha}}\right|\prec \frac{1}{M\eta_0}\,,
\end{align}
on $\Omega$.
\end{proof}

We define an event which holds with high probability on $\Omega$ which is useful to estimate some inequalities.
\begin{definition}\label{definition of xi}
Let $\epsilon>0$ be fixed and let $\Xi_\epsilon$ be an event defined by requiring that the following holds on it:  $(1)$ for all $z\in\caD_{\phi}'$,~\eqref{weaklocallaw},~\eqref{F bound 1} and~\eqref{F bound 2} hold; $(2)$ for all~$z\in\caD_{\phi}'$ and~$\alpha\in A$,
\begin{align}
	\left|Q_{\alpha}\left(\frac{1}{G_{\alpha\alpha}}\right)\right|\le M^\epsilon\frac{1}{M\eta_0}\,;
\end{align}
and $(3)$, for all $a\in\llbracket 1,M\rrbracket$ and $\gamma\in \llbracket1,N \rrbracket$,
\begin{align}
	\max_{a,\gamma}|x_{a\gamma}|\le \frac{M^{\epsilon}}{\sqrt{M}}\,.
\end{align}
\end{definition}
By moment condition of $x_{ij}$, Lemma~\ref{lemma:step 2_3}, Corollary~\ref{corollary:step 2_1}, Lemma~\ref{lemma:step 2_1} and inequality~\eqref{bound on xij}, we know that $\Xi_\epsilon$ holds with high probability on $\Omega$.

\begin{corollary}\label{Zlemma 1}
For fixed $p\in \llbracket 1, N \rrbracket$, there exists a constant $c$, such that the following holds. For all $\T,\T',\T''\subset A$, with $|\T|\,,|\T'|\,,|\T''|\le p$, for all $\alpha,\beta\in A $, $\alpha\not=\beta$, and, for all $z\in\caD_{\phi}'$, we have
\begin{align}\label{Zlemma bound 1}
	\lone(\Xi_\epsilon)\left|{F^{(\T,\T')}_{\alpha\beta}(z)}\right|\le M^\epsilon M^{-\fb/2+\phi}\,,
\end{align}
\begin{align}\label{Zlemma bound 2}
	\lone(\Xi_\epsilon)\left|\frac{F_{\alpha\beta}^{(\T',\T'')}(z)}{G_{\alpha\alpha}^{(\T)}(z)} \right|\le \frac{M^\epsilon}{M\eta_0} \,,
\end{align}
and
\begin{align}\label{initial estimate on Q_a}
	\lone(\Xi_\epsilon)\left|Q_{\alpha}\left(\frac{1}{G_{\alpha\alpha}^{(\T)} }\right) \right|\le \frac{M^\epsilon}{M\eta_0} \,,
\end{align}
on $\Omega$, for $N$ sufficiently large.
\end{corollary}

The proof of this corollary is exactly identical with that of appendix B in~\cite{eejl}. See~\cite{eejl} for more detail.

\begin{lemma}\label{Jensen lemma}
Let $p\in\N$. Let $q\in\llbracket 0,p\rrbracket$ and consider random variables $(\caX_{\alpha})\equiv(\caX_{\alpha}(Q))$  and $(\caY_{\alpha})\equiv(\caY_{\alpha}(Q))$, $\alpha\in\llbracket 1,p\rrbracket$, satisfying
\begin{align}\label{estimates on good event}
	|\caX_{\alpha}|\prec \frac{1}{M\eta_0}\left(M^{-\fb/2+\phi}\right)^{(d_{\alpha}-1)}\,,\qquad |Q_{\alpha}\caY_{\alpha}|\prec\frac{1}{M\eta_0}\,,
\end{align}
where $d_{\alpha}\in\N_0$ satisfy $0\le s=\sum_{i=\alpha}^q (d_{\alpha}-1)\le p+2$. Assume moreover that there is a constant $K$, such that for any $r\in\N$, with $r\le 10 p$, 
\begin{align}\label{estimates on expectation}
	\E^X |\caX_{\alpha}|^{r}\prec M^{K(d_{\alpha}+1) r}\,,\qquad\E^X|\caY_{\alpha}|^{r}\prec M^{Kr}\,,\end{align}
where the $\E^X$ denote the partial expectation with respect to the random variables $(x_{ij})$ with $(\sigma_\alpha)$ kept fixed.

Then we have 
\begin{align}\label{jensen lemma bound}
	\left|\E^X\prod_{i=\alpha}^q Q_{\alpha}(\caX_{\alpha})\prod_{\alpha=q+1}^p Q_{\alpha}(\caY_{\alpha})\right|\prec \left(\frac{1}{M\eta_0}\right)^p\left(M^{-\fb/2+\phi}\right)^{s} \,.
\end{align}
(Here, we use the convention that, for $q=0$, the first product is set to one, and, similarly, for $q=p$, the second product is set to one.)
\end{lemma}
\begin{proof}
Let $h_\alpha\deq 2\lceil \frac{2+p}{1+d_\alpha} \rceil$, $\alpha\in\llbracket 1, p \rrbracket$. Fix $\phi>0$. Note that 
\begin{align}
	\E^X|Q_\beta \caX|^p\leq 2^{p-1}\E^X|\caX|^p+2^{p-1}\E^X|\E_\beta\caX|^p \,.
\end{align}
By Jensen's inequality, we also have 
\begin{align}
	\E^X|Q_\beta\caX|^p \leq 2^p\E^X|\caX|^p \,.
\end{align}
The H\"older's inequality implies that 
\begin{align}
	\left| \E^X\prod_{\alpha=1}^{q} Q_\alpha\caX_\alpha \prod_{\alpha=q+1}^{p} Q_\alpha\caY_\alpha \right|\leq 2^p \prod_{\alpha=1}^{q} (\E^X|\caX_\alpha|^{h_\alpha})^{1/h_\alpha}\prod_{\alpha=q+1}^{p}(\E^X | \caY_\alpha|^{h_\alpha})^{1/h_\alpha} \,.
\end{align}
Considering that for any $\epsilon>0$ and $D>0$, we have
\begin{align}
	\begin{split}
		\E^X[|\caX| ] &=\E^X[ |\caX| \lone(|\caX| \leq  M^\epsilon (M\eta_0)^{-1}M^{-(d_\alpha-1)(\fb/2+\phi)}) ]\\
		&+\E^X [ |\caX| \lone(|\caX| >  M^\epsilon (M\eta_0)^{-1}M^{-(d_\alpha-1)(\fb/2+\phi)})] \\
		&\leq   M^\epsilon(M\eta_0)^{-1}M^{-(d_\alpha-1)(\fb/2+\phi)}+\sqrt{\E^X|\caX|^2}\sqrt{\mathbb{P}(|\caX|> M^{\epsilon}(M\eta_0)^{-1}M^{-(d_\alpha-1)(\fb/2+\phi)})} \\ &\leq
		M^{\epsilon}(M\eta_0)^{-1}M^{-(d_\alpha-1)(\fb/2+\phi)}+M^{2K(d_\alpha+1)-D/2} 
	\end{split}
\end{align}
for large enough $N\geq N_0(\epsilon,D)$. Hence we obtain that 
\begin{align}
	\E^X |\caX| \prec  (M\eta_0)^{-1}M^{-(d_\alpha-1)(\fb/2+\phi)} \,.
\end{align}
Furthermore, by the property of stochastic dominance, 
\begin{align}
	\E^X|\caX|^n \prec  ( (M\eta_0)^{-1}M^{-(d_\alpha-1)(\fb/2+\phi)} )^n \,.
\end{align}
Similarly, we can obtain 
\begin{align}
	\E^X |\caY|^n \prec ((M\eta_0)^{-1})^n \,.
\end{align}
Then it is easy to show the desired lemma.
\end{proof}

In order to prove the fluctuation average lemma, we need to consider the random variables of the form
\begin{align}\frac{F^{\#}_{\alpha_i\beta_1}}{G_{\alpha_i\alpha_i}^{\#}}\cdot F^{\#}_{\beta_1\beta_2}F^{\#}_{\beta_2\beta_3}\cdots F^{\#}_{\beta_n\alpha_i}
\end{align}
where $\#$ stands for som appropriate $(\mathbb{T},\mathbb{T}')$ with $p\in 2\N$, $|\mathbb{T}|\leq p-2$,$|\mathbb{T}|'\leq p-1$. Moreover, $\beta_1\neq \alpha_i$, $\beta_k\leq \beta_{k+1}$, $(k\in \llbracket 1,n-1 \rrbracket)$, $\beta_n\neq \alpha_1$. \\
By using Lemma~\ref{Zlemma 1} $n$ times, we obtain an upper bound of the form that of $ \caX$ from Lemma~\ref{Jensen lemma}. In addition, in order to apply Lemma~\ref{Jensen lemma}, we also need an upper bound of $r$-th moment of the variables.

\begin{lemma}\label{remark about easy estimate1}
For any fixed even integer $p\in 2\N$, let $\#$ stands for some appropriate $(\mathbb{T},\mathbb{T}')$ with $|\mathbb{T}|\leq p-2$,$|\mathbb{T}|'\leq p-1$.  If $\beta_1\neq \alpha_i$, $\beta_k\leq \beta_{k+1}$, $(k\in \llbracket 1,n-1 \rrbracket)$, $\beta_n\neq \alpha_1$, then we have
\begin{align}\label{the easy estimate}
	\E^X\left|\frac{F^{\#}_{\alpha_i\beta_1}}{G_{\alpha_i\alpha_i}^{\#}}\cdot F^{\#}_{\beta_1\beta_2}F^{\#}_{\beta_2\beta_3}\cdots F^{\#}_{\beta_n\alpha_i}\right|^r\prec M^{Kr(n+1)}\,,
\end{align}
for some constants $K$, for all $r\le 10 p$ and $1\leq n \leq p+1$.	
\end{lemma}
\begin{proof} Starting from Schur's formula
\begin{align*}
	\frac{1}{G_{\alpha\alpha}^{(\T)}}=-\sigma_\alpha^{-1}-\sum_{k,l}^{(\mathbb{T}\alpha)}x_{\alpha k}G_{kl}^{(\mathbb{T}\alpha)}x_{l\alpha},\quad\qquad (a\not\in\T)\,,
\end{align*}
and recall the trivial bounds $|G_{\alpha\alpha}^{(\T)}|\le\eta^{-1}\le M$, $\E^X |x_{ij}|^{q}\le C_qM^{-q/2}$ and $ |\sigma_\alpha^{-1}|^q\le C^q$, which holds since $\sigma_\alpha\in [l,1]$, and the boundedness of~$\caD_\phi'$.
Then we get
\begin{align}
	\begin{split}
		|| F_{\beta_i,\beta_{i+1}}^{\#}||_{r(n+1)} \le \frac{1}{\eta} \left\| \frac{1}{G_{\beta_1,\beta_{i+1}}^{\#}}\right\|
		&\le M \left( C+\sum_{k,l}^{(\#\beta_{i+1})} ||x_{\alpha k}G_{kl}^{(\mathbb{T}\alpha)}x_{l\alpha}  ||_{r(n+1)} \right)\\
		&\leq M\left( C+M^2C'(r(n+1))\leq M^3C''r(n+1)     \right)\,,
	\end{split}
\end{align}
which implies
\begin{align}
	||  F_{\beta_i,\beta_{i+1}}^{\#}||_{r(n+1)} \prec M^3 \,.
\end{align}
Furthermore, we have
\begin{align}
	\left|\left| \frac{F_{\alpha_i\beta_1}^{\#}}{G_{\alpha_i,\alpha_i}^{\#}} \right|\right|_{r(n+1)}\prec M^4\,.
\end{align}
By H\"older's inequality,
\begin{align}\label{the easy estimate 2}
	\E^X\left|\frac{F^{\#}_{\alpha_i\beta_1}}{G_{\alpha_i\alpha_i}^{\#}}\cdot F^{\#}_{\beta_1\beta_2}F^{\#}_{\beta_2\beta_3}\cdots F^{\#}_{\beta_n\alpha_i}\right|^r\le \left|\left| \frac{F_{\alpha_i\beta_1}^{\#}}{G_{\alpha_i,\alpha_i}^{\#}} \right|\right|_{r(n+1)}^r\prod_{i=1}^n \left|\left| F_{\beta_i\beta_{i+1}}^{\#} \right|\right|_{r(n+1)}^r \,,
\end{align}
where we set $\beta_{n+1}\deq\alpha_i$.
Then we obtain
\begin{align}
	\left|\left| \frac{F_{\alpha_i\beta_1}^{\#}}{G_{\alpha_i,\alpha_i}^{\#}} \right|\right|_{r(n+1)}^r\prod_{i=1}^n \left|\left| F_{\beta_i\beta_{i+1}}^{\#} \right|\right|_{r(n+1)}^r\prec M^{4r+3rn}\,.
\end{align}
Choosing $K=4$, we obtain desired lemma.
\end{proof}
From the previous lemmas, we can derive the following significant lemma.
\begin{lemma}\label{the Zlemma}\emph{[Fluctuation Average Lemma]}
Let $A\deq\llbracket n_0,M\rrbracket$. Recall the definition of the domain $\caD_{\phi}'$ in~\eqref{a index assumption}. Let $\Xi$ denote the event in Definition~\ref{definition of xi} and assume it holds with high probability. Then there exist constants $C$, $c$, $c_0$, such that for fixed $p\in 2\N$, $p=2r$, $r\in\N$, , we have
\begin{align}\label{estimate in Zlemma}
	\E^X\left|\frac{1}{M}\sum_{\alpha\in A}Q_{\alpha}\left(\frac{1}{G_{\alpha\alpha}(z)}\right) \right|^{p}\prec M^{-p/2-p\fb/2+2p\phi}\,,
\end{align}
for all $z\in\caD_{\phi}'$, on $\Omega$.
\end{lemma}
\begin{proof}
The proof of this lemma is only rely on the identity~\eqref{basic resolvent} and the estimates. Therefore, we can follow the method from~\cite{eejl} or~\cite{sce} to prove our lemma. Check~\cite{eejl} for more detail of the proof. 
\end{proof}

\begin{proof}[Proof of Lemma~\ref{lemma:step 4}]
From Lemma~\ref{the Zlemma}, by the Chebyshev's inequality, for any fixed $\phi>0$ and $D>0$, we have
\begin{align}
	\mathbb{P}\left( \left| \frac{1}{M}\sum_{\alpha\in A} Q_\alpha\left( \frac{1}{G_{\alpha\alpha}} \right) \right| > M^\epsilon M^{-1/2-\fb/2+2\phi} \right)\leq M^{1-\epsilon p}
\end{align}
for large enough $M>M_0(\epsilon,p)$ where $p\in2\N$. If we choose $p\geq (1+D)/\epsilon$, we obtain the desired lemma.
\end{proof}
\begin{proof}[Proof of Corollary~\ref{corollary:step 4}]
Since the proof of Corollary~\ref{corollary:step 4} is the same with that of~\cite{eejl}, we omit the detail in this paper.
\end{proof}

\section*{Acknowledgements}
We thank Paul Jung for helpful discussions. The work of J. Kwak was partially supported by National Research Foundation of Korea under grant number NRF-2017R1A2B2001952. The work of J. O. Lee and J. Park was partially supported by National Research Foundation of Korea under grant number NRF-2019R1A5A1028324.

\begin{appendices}
\section{Probability of ``good configuration" $\Omega$}\label{app:omega}
In this appendix, we estimate the probabilities for the events $1.$-$3.$ in the definition of $\Omega$; see Definition~\ref{sigma assumptions}. Recall the definition of the constants $\phi$ in~\eqref{phi condition} and $\kappa_0$ in~\eqref{definition of kappa0}. In the following, we denote by $(\sigma_\alpha)_{\alpha=1}^M$ the (unordered) sample points distributed according to the measure $\nu$ with $\b>1$. We denote by $(\sigma_{(\alpha)})$ the order statistics of $(\sigma_\alpha)$, i.e., $\sigma_{(1)}\ge \sigma_{(2)}\ge\ldots\ge \sigma_{(M)}$. 

\begin{lemma} \label{lemma:app_1}
Let $(\sigma_{(\alpha)})$ be the order statisctics of sample points $(\sigma_\alpha)$ under the probability distribution $\nu$ with $\b>1$. Let $n_0 > 10$ be a fixed positive integer independent of $M$. Then, for any $\gamma \in \llbracket 1, n_0-1 \rrbracket$, we have
\begin{align}
	\p \left( M^{-\phi} \kappa_0 < |
	\sigma_{(\gamma)} - \sigma_{(\beta)}| < (\log M) \kappa_0\,, \forall \beta \neq \gamma \right) \ge 1 - C (\log M)^{1+2\b} M^{-\phi}\,.
\end{align}
In addition, for $\gamma=1$, we have
\begin{align}
	\p \left( M^{-\phi} \kappa_0 < |1 - \sigma_{(1)}| < (\log M) \kappa_0 \right) \ge 1 - C M^{-\phi (\b+1)}\,.
\end{align}
\end{lemma}

For a proof, we refer to Theorem 8.1 of~\cite{eejl}. Here, we state the key part of the proof as a following remark.

\begin{remark}
For a random variables $\sigma$ with law $\nu$ as in~\eqref{jacobi measure}, we have for any $x \ge 0$,  
\begin{align} \label{sigma tail}
	C^{-1} x^{\b+1} \le \p (1 - \sigma \le x) \le C x^{\b+1}\,,
\end{align}
for some constant $C > 1$.
\end{remark}

Next, we estimate the probability of condition $(2)$ in Definition~\ref{sigma assumptions} to hold.

\begin{lemma} \label{lemma:app_2}
Assume the conditions in Lemma~\ref{lemma:app_1}. Recall the definition of $\caD_{\phi}$ in~\eqref{domain}. Then, for any fixed (small) $\epsilon>0$, $D>0$, there exists $M_0(\epsilon,D)$ such that if $M\geq M_0$, then
\begin{align} \label{eq:CLT_fixed}
	\p \left( \bigcup_{z \in \caD_{\phi}} \ \left \{ \left| \frac{1}{N} \sum_{\alpha=1}^M \frac{\sigma_\alpha}{\sigma_\alpha\mfc(z)+1} - d^{-1} \int \frac{t \dd\nu(t)}{t\mfc(z)+1} \right| > \frac{M^{\phi+\epsilon}}{\sqrt M} \right \} \right) \le   M^{-D}.
\end{align}
\end{lemma}

\begin{proof}
Note that 
\begin{align}
	\frac{1}{N} \sum_{\alpha=1}^M \frac{\sigma_\alpha}{\sigma_\alpha\mfc(z)+1} - d^{-1} \int \frac{t \dd\nu(t)}{t\mfc(z)+1}=d^{-1}\left(\frac{1}{M} \sum_{\alpha=1}^M \frac{\sigma_\alpha}{\sigma_\alpha\mfc(z)+1} - \int \frac{t \dd\nu(t)}{t\mfc(z)+1}\right).\end{align}
Fix $z \in \caD_{\phi}$. For $\alpha\in\llbracket 1,M\rrbracket$, let $X_\alpha \equiv X_\alpha(z)$ be the random variable
\begin{align}
	X_\alpha \deq \frac{\sigma_\alpha}{\sigma_\alpha\mfc(z)+1} - \int \frac{t \dd \nu(t)}{t \mfc(z)+1}\,.
\end{align}
By definition, $\E X_\alpha = 0$. Moreover, we have
\begin{align}
	\E |X_\alpha|^2 \le \int \frac{t^2\dd \nu(t)}{|1+t\mfc(z)|^2} = dR_2(z) < d\,,\quad\qquad (z\in\C^+) \,,
\end{align}
and, for any positive integer $p \ge 2$,
\begin{align}
	\E |X_\alpha|^{p} \le \frac{1}{\eta^{p-2}} \E |X_\alpha|^2 \le CM^{(1/2 +\phi) (p-2)}\,,\qquad\quad (z\in\caD_{\phi})\,.
\end{align}
The proof of left parts are analogous to the Theorem 8.2 of~\cite{eejl}.

\end{proof}

To estimate the probability for the third condition in Definition~\ref{sigma assumptions}, we need the following two auxiliary lemmas. Recall the definition of $R_2$ in~\eqref{definition of R2 without hat}.

\begin{lemma} \label{R2 estimate}
If $0 < C^{-1} \eta \le \im \mfc(z) \le C\,\eta$, $z=E+\ii\eta$, for some constant $C \ge 1$, then we have
\begin{align}
	0 \le R_2 (z) \le 1-\frac{1}{C}\,.
\end{align}
\end{lemma}

\begin{proof}
We have 
\begin{align}
	1-C\le R_2 (z) = 1-\eta \frac{|\mfc(z)|^2}{\im \mfc(z)}=1-\eta \left\{-\im \frac{1}{\mfc(z)}\right\}^{-1} \le 1-\frac{1}{C}\,.
\end{align}
and by definition, $R_2(z)\ge0$.
\end{proof}

The imaginary part of $\mfc(z)$ can be estimated using the following lemma. We refer Lemma 8.4 of~\cite{eejl} to proof.

\begin{lemma} \label{im mfc bound}
Assume that $\mu_{fc}$ has support $[L_-, L_+]$ and there exists a constant $C>1$ such that
\begin{equation}
	C^{-1} \kappa^{\b} \le \mu_{fc} (z) \le C \kappa^{\b}\,,
\end{equation}
for any $0 \le \kappa \le L_+$. Then,
\begin{itemize}
	\item[$(1)$] for $z = L_+ - \kappa + i \eta$ with $0 \le \kappa \le L_+$ and $0 < \eta \le 3$, there exists a constant $C>1$ such that
	\begin{equation}
		C^{-1} (\kappa^{\b} + \eta) \le \im \mfc (z) \le C (\kappa^{\b} + \eta)\,;
	\end{equation}
	\item[$(2)$] for $z = L_+ + \kappa + i \eta$ with $0 \le \kappa \le 1$ and $0 < \eta \le 3$, there exists a constant $C>1$ such that
	\begin{equation}
		C^{-1} \eta \le \im \mfc (z) \le C \eta\,.
	\end{equation}
\end{itemize}
\end{lemma}

\begin{remark}
Lemma~\ref{im mfc bound} shows that for any sufficiently small $\phi>0$, there exists a constant $C_{\b} > 1$ such that
\begin{align}
	C_{\b}^{-1} \eta \le \im \mfc (z) \le C_{\b} \eta\,,
\end{align}
for all $z \in \caD_{\phi}$ satisfying $L_+ - \re z \le M^{\phi} \kappa_0$.
\end{remark}

Assuming Lemma~\ref{im mfc bound}, we have the following estimate. Recall that $\caD_{\phi}$ is defined in~\eqref{domain}.

\begin{lemma} \label{lemma:app_3}
Assume the conditions in Lemma~\ref{lemma:app_1}. Then, there exist constants $\mathfrak{c} < 1$ and $C > 0$, independent of~$N$, such that, for any $z = E + \ii \eta \in \caD_{\phi}$ satisfying
\begin{align} \label{condition app3}
	\min_{\alpha \in \llbracket 1, M \rrbracket} \left|\re \left(1+\frac{1}{\sigma_{(\alpha)} \mfc}\right)\right| = \left|\re \left(1+\frac{1}{\sigma_{(\gamma)} \mfc}\right)\right|\,,
\end{align}
for some $\gamma \in \llbracket 1, n_0 -1 \rrbracket$, we have
\begin{align}
	\p \left( \frac{1}{N} \sum_{\alpha:\alpha\ne \gamma}^{M} \frac{\sigma_{(\alpha)}^2 |\mfc|^2}{|1+\sigma_{(\alpha)} \mfc|^2} <\mathfrak{c} \right) \ge 1 - C (\log M)^{1+2\b} M^{-\phi}.
\end{align}
\end{lemma}

\begin{proof}
We only prove the case $\gamma=1$; the general case can be shown by the same argument. In the following, we assume that $ M^{-\phi}\kappa_0 < |1 - \sigma_{(1)}| < (\log M) \kappa_0$, and $|\sigma_{(1)} - \sigma_{(2)}| >  M^{-\phi}\kappa_0$. 

Recall the definition of $R_2$ in~\eqref{definition of R2 without hat}. For $\alpha\in\llbracket 1,M\rrbracket$, let $Y_\alpha \equiv Y_\alpha(z)$ be the random variable
\begin{align}
	Y_\alpha(z) \deq d^{-1} \left|\frac{\sigma_\alpha \mfc(z)}{1+\sigma_\alpha\mfc(z)}\right|^2\,, \qquad\quad(z\in\C^+).
\end{align}
Observe that $\E Y_\alpha = R_2 < 1$ for $z\in\C^+$. Moreover, we find that there exists a constant $c < 1$ independent of $N$, such that $R_2 (z) < c$ uniformly for all $z \in \caD_{\phi}$ satisfying~\eqref{condition app3}, where we combined Lemma~\ref{R2 estimate} and Lemma~\ref{im mfc bound}. We also have that $Y_\alpha(z) \le C\eta^{-2}$.

We first consider the special choice $E = L_+$. Let $\wt Y_\alpha$ be the truncated random variable defined by
\begin{align}
	\wt Y_\alpha \deq
	\begin{cases}
		Y_\alpha\,, & \text{ if  } Y_\alpha <  M^{2\phi}\kappa_0^{-2}\,, \\
		M^{2\phi} \kappa_0^{-2}\,, & \text{ if  } Y_\alpha \ge M^{2\phi}\kappa_0^{-2}\,.
	\end{cases}
\end{align}
Notice that using the estimate~\eqref{sigma tail}, we have for $z = L_+ + \ii \eta \in \caD_{\phi}$ that
\begin{align}
	\p ( Y_\alpha \neq \wt Y_\alpha ) \le C M^{-1 - (\b+1)\phi}\,.
\end{align}
Let us define
\begin{align}
	S_M \deq \sum_{\alpha=1}^M Y_\alpha\,, \qquad \wt S_M \deq \sum_{\alpha=1}^M \wt Y_\alpha\,,
\end{align}
then it follows that
\begin{align}
	\p (S_M \neq \wt S_M) \le C M^{- (\b+1)\phi}.
\end{align}

Now, we estimate the mean and variance of $\wt Y_i$. From the trivial estimate $\p (Y_\alpha \ge x) \le \p (Y_\alpha \neq \wt Y_\alpha)$ for $x \ge  M^{2\phi}\kappa_0^{-2}$, we find that 
\begin{align}
	\E Y_\alpha - \E \wt Y_\alpha \le \int_{M^{2\phi} \kappa_0^{-2}}^{C\eta^{-2}} \p ( Y_\alpha \neq \wt Y_\alpha ) \dd x \le C' M^{-(\b-1)\phi}\,,
\end{align}
for some $C' > 0$. As a consequence, we get
\begin{align}
	\E \wt Y_\alpha^2 \le M^{2\phi} \kappa_0^{-2} \E \wt Y_\alpha \le  M^{2\phi}\kappa_0^{-2} \E Y_\alpha \le M^{2\phi}\kappa_0^{-2}.
\end{align}
We thus obtain that
\begin{align} \begin{split}
		&\p \left( \left| \frac{S_M}{M} - \E Y_\alpha \right| \ge C' M^{-(\b-1)\phi} + M^{-\phi} \right) \le \p \left( \left| \frac{\wt S_M}{M} - \E \wt Y_\alpha \right| \ge M^{-\phi} \right) + \p (S_M \neq \wt S_M) \\
		&\qquad\le \frac{M^{2\phi} \E \wt Y_\alpha^2}{M} + C M^{- (\b+1)\phi} \le C M^{- (\b+1)\phi},
\end{split} \end{align}
hence, for a constant $c$ satisfying $R_2 + C' M^{-(\b-1)\phi} + M^{-\phi} < c < 1$,
\begin{align*}
	\begin{split}
		\p \left( \frac{1}{N} \sum_{\alpha=1}^M \left|\frac{\sigma_\alpha\mfc(z)}{1+\sigma_\alpha\mfc(z)}\right|^2 < c \right) &\ge 1 - \p \left( \left| \frac{S_M}{M} - \E Y_\alpha \right| \ge C' M^{-(\b-1)\phi} + M^{-\phi} \right) \\
		&\ge 1 - C M^{-(\b+1)\phi}.
	\end{split}
\end{align*}
This proves the desired lemma for $E = L_+$.

Before we extend the result to general $z \in \caD_{\phi}$, we estimate the probabilities for some typical events we want to assume. Consider the set
\begin{align}
	\Sigma_{\phi} \deq \{ \sigma_\alpha : |1 - \sigma_\alpha| \le M^{3\phi} \kappa_0 \}\,,
\end{align}
and the event
\begin{align}
	\Omega_{\phi} \deq \{ |\Sigma_{\phi}| < M^{3\phi(\b+2)} \}\,.
\end{align}
From the estimate~\eqref{sigma tail}, we have
\begin{align}
	\p (|1 - \sigma_\alpha| \ge M^{3\phi} \kappa_0) \le C M^{-1 + 3(\b+1)\phi}\,,
\end{align}
so using a Chernoff bound, we find that
\begin{align}
	\p (\Omega_{\phi}^c) \le \exp \left( -C (\log M) M^{3\phi} M^{3(\b+1) \phi} \right)\,,
\end{align}
for some constant $C$. Notice that we have, for $\sigma_\alpha \notin \Sigma_{\phi}$,
\begin{align} \label{eq:sigma_i notin Sigma}
	-1-\re\frac{1}{\sigma_\alpha\mfc(L_++\ii\eta)}> M^{3\phi} \kappa_0 \gg -\im \frac{1}{\sigma_\alpha\mfc(L_+ + \ii \eta)}\,,
\end{align}
where we have used Lemma~\ref{mfc estimate}, i.e., $|1+\mfc^{-1}(L_+ +\ii\eta)| = \caO (\eta)$. We now assume that~$\Omega_{\phi}$ holds and
\begin{align}
	\frac{1}{N} \sum_{\alpha=1}^M \left|\frac{\sigma_\alpha\mfc(L_++\ii\eta)}{1+\sigma_\alpha\mfc(L_++\ii\eta)}\right|^2 < c < 1\,.
\end{align}
Further, we recall that the condition~\eqref{condition app3} implies
\begin{align}
	-\re \mfc^{-1}(z) \ge \sigma_{(n_0)},
\end{align}
which yields, together with Lemma~\ref{mfc estimate} and Lemma~\ref{lemma:app_1} that $E \ge L_+ - M^{\phi} \kappa_0$ with probability higher than $1 - C (\log M)^{1+2\b} M^{-\phi}$. Thus we assume in the following that $E \ge L_+-M^{\phi}\kappa_0$.

Consider the following two choices for such $E$:
\begin{enumerate}
	\item[$(1)$] When $L_+ - M^{\phi} \kappa_0 \le E \le L_+ + M^{2\phi} \kappa_0$, we have that
	\begin{align}
		\left|1+\frac{1}{\sigma_\alpha\mfc(z)}\right| = \left|1+\frac{1}{\sigma_\alpha\mfc(L_++\ii\eta)}\right| + O(M^{2\phi} \kappa_0)\,,
	\end{align}
	where we used Lemma~\ref{mfc estimate}. Hence, using~\eqref{eq:sigma_i notin Sigma}, we obtain for $\sigma_\alpha \notin \Sigma_{\phi}$ that
	\begin{align} \begin{split}
			\left|\frac{\sigma_\alpha\mfc(z)}{1+\sigma_\alpha\mfc(z)}\right|^2 &\le \left|\frac{\sigma_\alpha\mfc(L_++\ii \eta)}{1+\sigma_\alpha\mfc(L_+ + \ii \eta)}\right|^2 + CM^{2\phi }\kappa_0\left|\frac{\sigma_\alpha\mfc(L_+ + \ii \eta)}{1+\sigma_\alpha\mfc(L_+ + \ii \eta)}\right|^3 \\
			&\le (1 + C M^{-\phi})\left|\frac{\sigma_\alpha\mfc(L_+ + \ii \eta)}{1+\sigma_\alpha\mfc(L_+ + \ii \eta)}\right|^2\,.
	\end{split} \end{align}
	We thus have that
	\begin{align} \begin{split}
			\frac{1}{N} \sum_{\alpha=2}^M \left|\frac{\sigma_\alpha\mfc(z)}{1+\sigma_\alpha\mfc(z)}\right|^2 &\le \frac{M^{4\phi(\b+1)}}{N} \frac{1}{ \kappa_0^2} + \frac{1+CM^{-\phi}}{N} \sum_{\alpha: \sigma_\alpha \notin \Sigma_{\phi}} \left|\frac{\sigma_\alpha\mfc(L_++\ii\eta)}{1+\sigma_\alpha\mfc(L_++\ii\eta)}\right|^2 \\
			&\le C'M^{-\phi} + \frac{1+CM^{-\phi}}{N} \sum_{\alpha=1}^M \left|\frac{\sigma_\alpha\mfc(L_++\ii\eta)}{1+\sigma_\alpha\mfc(L_++\ii\eta)}\right|^2 < c < 1\,,
	\end{split} \end{align}
	where we also used the assumption that $|\sigma_{(2)}-\sigma_{(1)}|\ge  \kappa_0$. 
	
	\item[$(2)$] When $E > L_+ + M^{2\phi} \kappa_0$, we have
	\begin{align}
		\re \mfc^{-1}(L_+ + \ii \eta)-\re \mfc^{-1}(E + \ii \eta) \gg -\im \mfc^{-1}(E + \ii \eta)\,,
	\end{align}
	where we again used Lemma~\ref{mfc estimate}, hence, from~\eqref{eq:sigma_i notin Sigma} we obtain that
	\begin{align}
		\left|1+\frac{1}{\sigma_\alpha\mfc(z)}\right| \ge \left|1+\frac{1}{\sigma_\alpha\mfc(L_++\ii\eta)}\right|
		\,.
	\end{align}
	We may now proceed as in $(1)$ to find that
	\begin{align} \begin{split}
			\frac{1}{N} \sum_{\alpha=2}^M \left|\frac{\sigma_\alpha\mfc(z)}{1+\sigma_\alpha\mfc(z)}\right|^2\le M^{-\epsilon} + \frac{1}{N} \sum_{\alpha=1}^M \left|\frac{\sigma_\alpha\mfc(L_++\ii\eta)}{1+\sigma_\alpha\mfc(L_++\ii\eta)}\right|^2 < c < 1\,,
	\end{split} \end{align}
	
\end{enumerate}

Since we proved in Lemma~\ref{lemma:app_1} that the assumptions $ M^{-\phi}\kappa_0 < |1 - \sigma_{(1)}| < (\log M) \kappa_0$ and $|\sigma_{(1)} - \sigma_{(2)}| > M^{-\phi} \kappa_0$ hold with probability higher than $1 - C (\log M)^{1+2\b} M^{-\phi}$, we find that the desired lemma holds for any $z\in\caD_{\phi}'$.
\end{proof}
\end{appendices}






\end{document}